%% file: 0_main.tex
\documentclass[reqno, 12pt]{amsart}

\usepackage{amsmath}
\usepackage{amsfonts}
\usepackage{amssymb}
\usepackage{amsthm}
\usepackage{mathrsfs}
\usepackage{bbm}
\usepackage{graphicx, color}
\usepackage{tikz}
\usetikzlibrary{calc,decorations.markings}
\usetikzlibrary{arrows.meta}
\usetikzlibrary{positioning}
\usetikzlibrary{cd, intersections, calc, decorations.pathmorphing, arrows, decorations.pathreplacing}

\usepackage{adjustbox}

\usepackage{bmpsize}

\usepackage{mathtools}

\usepackage{enumerate}

\definecolor{refkey}{rgb}{0.9451,0.2706,0.4941}
\definecolor{labelkey}{rgb}{0.9451,0.2706,0.4941}

\definecolor{mygreen}{rgb}{0,0.7,0.3}
\definecolor{myblue}{rgb}{0,0.50,1.20}
\definecolor{myorange}{rgb}{1,0.5,0.1}

\usepackage{subfiles}
\usepackage[colorlinks, linkcolor=blue, citecolor=red, urlcolor=cyan,pagebackref]{hyperref}

\allowdisplaybreaks

\numberwithin{equation}{section}

\usepackage{cleveref}
\crefname{thm}{Theorem}{Theorems}
\crefname{cor}{Corollary}{Corollaries}
\crefname{lem}{Lemma}{Lemmas}
\crefname{sublem}{Sublemma}{Sublemmas}
\crefname{prop}{Proposition}{Propositions}
\crefname{dfn}{Definition}{Definitions}
\crefname{defi}{Definition}{Definitions}
\crefname{ex}{Example}{Examples}
\crefname{claim}{Claim}{Claims}
\crefname{conj}{Conjecture}{Conjectures}
\crefname{conv}{Notation}{Notations}
\crefname{rem}{Remark}{Remarks}
\crefname{rmk}{Remark}{Remarks}
\crefname{prob}{Problem}{Problems}
\crefname{figure}{Figure}{Figures}
\crefname{table}{Table}{Tables}
\crefname{section}{Section}{Sections}
\crefname{subsection}{Section}{Sections}
\crefname{appendix}{Appendix}{Appendices}
\crefname{introthm}{Theorem}{Theorems}
\crefname{introcor}{Corollary}{Corollaries}
\crefname{introconj}{Conjecture}{Conjectures}

\newtheorem{thm}{Theorem}[section]
\newtheorem{prop}[thm]{Proposition}
\newtheorem{cor}[thm]{Corollary}
\newtheorem{lem}[thm]{Lemma}

\newtheorem{introthm}{Theorem}

\theoremstyle{definition}
\newtheorem{dfn}[thm]{Definition}

\newtheorem{ex}[thm]{Example}

\newtheorem{conv}[thm]{Notation}
\theoremstyle{remark}

\newtheorem{rem}[thm]{Remark}

\newcommand{\corrected}
{\textcolor{red}{(corrected)}}

\newcommand{\ot}{\leftarrow}

\newcommand*{\chom}{\mathcal{H}\kern -.5pt om}

\newcommand{\bZ}{\mathbb{Z}}
\newcommand{\bQ}{\mathbb{Q}}
\newcommand{\bR}{\mathbb{R}}
\newcommand{\bC}{\mathbb{C}}

\newcommand{\bT}{\mathbb{T}}
\newcommand{\bM}{\mathbb{M}}

\newcommand{\bP}{\mathbb{M}_\circ}
\newcommand{\bG}{\mathbb{G}}
\newcommand{\bA}{\mathbb{A}}
\newcommand{\bB}{\mathbb{B}}
\newcommand{\bE}{\mathbb{E}}

\newcommand{\bs}{{\boldsymbol{s}}}

\newcommand{\A}{\mathcal{A}}

\newcommand{\cF}{\mathcal{F}}

\newcommand{\cL}{\mathcal{L}}

\newcommand{\cO}{\mathcal{O}}

\def\P{{\mathcal{P}}}

\newcommand{\cU}{\mathcal{U}}

\newcommand{\cW}{\mathcal{W}}
\newcommand{\X}{\mathcal{X}}
\newcommand{\cX}{\mathcal{X}}

\newcommand{\cZ}{\mathcal{Z}}

\newcommand{\sfa}{\mathsf{a}}
\newcommand{\sfx}{\mathsf{x}}

\newcommand{\Hom}{\mathrm{Hom}}

\newcommand{\tri}{\triangle}
\newcommand{\sgn}{\mathrm{sgn}}

\newcommand{\tr}{\mathsf{T}}

\newcommand{\bExch}{\bE \mathrm{xch}}
\newcommand{\uf}{\mathrm{uf}}
\newcommand{\f}{\mathrm{f}}

\newcommand{\fsl}{\mathfrak{sl}}

\renewcommand{\cL}{\mathcal{L}_{\fsl_3}}

\newcommand{\ve}{\varepsilon}
\newcommand{\sfs}{{\mathsf{s}}}

\newcommand{\hL}{\widehat{L}}

\DeclareMathOperator{\coker}{\mathrm{coker}}

\DeclareMathOperator{\interior}{\mathrm{int}}

\makeatletter
\newcommand{\oset}[3][0ex]{%
  \mathrel{\mathop{#3}\limits^{
    \vbox to#1{\kern-2\ex@
    \hbox{$\scriptstyle#2$}\vss}}}}
\makeatother

\makeatletter
\newcommand{\osetnear}[3][0ex]{%
  \mathrel{\mathop{#3}\limits^{
    \vbox to#1{\kern-.3\ex@
    \hbox{$\scriptstyle#2$}\vss}}}}
\makeatother

\newcommand\qarrow[2]{\draw[-latex,shorten >=2pt,shorten <=2pt] (#1) -- (#2) [thick];} 
\newcommand\qdlarrow[2]{\draw[-latex,dashed,shorten >=2pt,shorten <=2pt,bend left=0.5cm] (#1) to (#2) [thick];} 
\newcommand{\uniarrow}[3]{\draw[-latex,thick,#3] (#1) to (#2);}

\def\centerarc(#1)(#2:#3:#4)
    {($(#1)+({#4*cos(#2)},{#4*sin(#2)})$) arc(#2:#3:#4) }

\tikzset{
  mid arrow/.style={postaction={decorate,decoration={
        markings,
        mark=at position .5 with {\arrow[#1]{stealth}}
      }}},
}

\tikzset{
    partial ellipse/.style args={#1:#2:#3}{
        insert path={+ (#1:#3) arc (#1:#2:#3)}
    }
}

\newcommand{\quiverplus}[3]{
\begin{scope}[>=latex]
{\color{mygreen}
    \path(#1) coordinate(x1);
    \path(#2) coordinate(x2);
    \path(#3) coordinate(x3);
    \foreach \l in {1,2}
    {
        \draw($(x1)!0.333*\l!(x2)$) circle(2pt) coordinate(x12\l);
        \draw($(x2)!0.333*\l!(x3)$) circle(2pt) coordinate(x23\l);
        \draw($(x3)!0.333*\l!(x1)$) circle(2pt) coordinate(x31\l);
    }
    \path($(x1)!0.5!(x2)$) coordinate(H);
    \draw($(x3)!0.667!(H)$) circle(2pt) coordinate(G);
    \qarrow{x121}{G}
    \qarrow{x231}{G}
    \qarrow{x311}{G}
    \qarrow{G}{x122}
    \qarrow{G}{x232}
    \qarrow{G}{x312}
    \qarrow{x312}{x121}
    \qarrow{x122}{x231}
    \qarrow{x232}{x311}
}
\end{scope}
}

\newcommand{\quiversquare}[4]{
{\color{mygreen}
    \path(#1) coordinate(x1);
    \path(#2) coordinate(x2);
    \path(#3) coordinate(x3);
	\path(#4) coordinate(x4);
    \foreach \l in {1,2}
    {
        \draw($(x1)!0.333*\l!(x2)$) circle(2pt) coordinate(x12\l);
        \draw($(x2)!0.333*\l!(x3)$) circle(2pt) coordinate(x23\l);
        \draw($(x3)!0.333*\l!(x4)$) circle(2pt) coordinate(x34\l);
		\draw($(x4)!0.333*\l!(x1)$) circle(2pt) coordinate(x41\l);
		\draw($(x1)!0.333*\l!(x3)$) circle(2pt) coordinate(x13\l);	
	    \draw($(x2)!0.333*\l!(x4)$) circle(2pt) coordinate(x24\l);
	}
}
}

\newcommand{\CoG}[3]{
    \path(#1) coordinate(x1);
    \path(#2) coordinate(x2);
    \path(#3) coordinate(x3);
    \path($(x1)!0.5!(x2)$) coordinate(H);
    \path($(x3)!0.667!(H)$) circle(2pt) coordinate(G);}
\newcommand{\sink}[3]{
    \CoG{#1}{#2}{#3}
    \draw[red,very thick,->-] (#1) -- (G);
    \draw[red,very thick,->-] (#2) -- (G);
    \draw[red,very thick,->-] (#3) -- (G);
}   
\newcommand{\source}[3]{
    \CoG{#1}{#2}{#3}
    \draw[red,very thick,-<-] (#1) -- (G);
    \draw[red,very thick,-<-] (#2) -- (G);
    \draw[red,very thick,-<-] (#3) -- (G);
}   

\newcommand{\bline}[3]{
    \path (#1)++(0,-#3) coordinate(m1);
    \path (#2)++(0,-#3) coordinate(m2);
    \filldraw[gray!30] (m1) -- (#1) -- (#2) -- (m2) --cycle;
    \draw[thick] (#1) -- (#2);
}
\newcommand{\tline}[3]{
    \path (#1)++(0,#3) coordinate(m1);
    \path (#2)++(0,#3) coordinate(m2);
    \filldraw[gray!30] (m1) -- (#1) -- (#2) -- (m2) --cycle;
    \draw[thick] (#1) -- (#2);
}

\tikzset{->-/.style 2 args={
	postaction={decorate},
	decoration={markings, mark=at position #1 with {\arrow[thick, #2]{>}}} 
    },
    ->-/.default={0.5}{}
}
\tikzset{-<-/.style 2 args={
	postaction={decorate},
	decoration={markings, mark=at position #1 with {\arrow[thick, #2]{<}}} 
    },
    -<-/.default={0.5}{}
}

\tikzset{
	overarc/.style={
		white, double=red, double distance=1.2pt, line width=2.4pt
	}
}

\tikzset{
    squigarrow/.style={-{Classical TikZ Rightarrow[length=4pt]}, decorate, decoration={snake, amplitude=1.8pt, pre length=2pt, post length=3pt}}
}

\pgfdeclarelayer{bg}    
\pgfsetlayers{bg,main}  

\title[Unbounded $\mathfrak{sl}_3$-laminations around punctures]{Unbounded $\mathfrak{sl}_3$-laminations around punctures}

\author[Tsukasa Ishibashi]{Tsukasa Ishibashi}
\address{Tsukasa Ishibashi, Mathematical Institute, Tohoku University, 
6-3 Aoba, Aramaki, Aoba-ku, Sendai, Miyagi 980-8578, Japan.}
\email{tsukasa.ishibashi.a6@tohoku.ac.jp}

\author[Shunsuke Kano]{Shunsuke Kano}
\address{Shunsuke Kano, Mathematical Science Center for Co-Creative Society, Tohoku University,
6-3 Aoba, Aramaki, Aoba-ku, Sendai, Miyagi 980-8578, Japan.}
\email{s.kano@tohoku.ac.jp}

\date{\today}

\begin{document}

\begin{abstract}
We continue to study the unbounded $\mathfrak{sl}_3$-laminations \cite{IK22}, with a focus on their structures at punctures. A key ingredient is their relation to the root data of $\mathfrak{sl}_3$. 
After giving a classification of signed $\mathfrak{sl}_3$-webs around a puncture, we describe the tropicalization of the Goncharov--Shen's Weyl group action in detail. We also clarify the relationship with several other approaches by Shen--Sun--Weng \cite{SSW} and Fraser--Pylyavskyy \cite{FP21}. Finally, we discuss a formulation of unbounded $\mathfrak{g}$-laminations for a general semisimple Lie algebra $\mathfrak{g}$ in brief. 
\end{abstract}

\maketitle

\setcounter{tocdepth}{1}
\tableofcontents

\input{1_introduction}
\input{2_webs}
\input{3_exact_sequence}

\input{4_Weyl}
\input{5_FP.tex}
\input{6_cluster}
\input{7_Weyl_cluster}

\end{document}

%% file: 1_introduction.tex
\section{Introduction}\label{sec:intro}
We continue to study the rational unbounded $\fsl_3$-laminations introduced in \cite{IK22}, especially focusing on their structure around punctures. Recall from the previous work that we have introduced the space $\cL^x(\Sigma,\bQ)$ of \emph{rational unbounded $\fsl_3$-laminations}, which is equipped with the \emph{shear coordinates} $\sfx_\tri: \cL^x(\Sigma,\bQ) \xrightarrow{\sim} \bQ^{I_\uf(\tri)}$ associated with ideal triangulations $\tri$ of $\Sigma$. This space is related to the Douglas--Sun--Kim's space $\cL^a(\Sigma,\bQ)$ of rational bounded $\fsl_3$-laminations \cite{DS20I,Kim21} via the \emph{ensemble map} $p: \cL^a(\Sigma,\bQ) \to \cL^x(\Sigma,\bQ)$. The integral points $\cL^x(\Sigma,\bZ)$ are represented by \emph{signed $\fsl_3$-webs}, which are $\fsl_3$-webs of Kuperberg \cite{Kuperberg} equipped with a sign on each end incident to a puncture. 

In this paper, we study the following structures:
\begin{itemize}
    \item The \emph{lamination exact sequence}
    \begin{align}\label{introeq:lam_exact}
    0 \to (\mathsf{Q}^\vee_\bQ)^{\bM} \to \cL^a(\Sigma,\bQ) \xrightarrow{p} \cL^x(\Sigma,\bQ) \xrightarrow{\theta} (\mathsf{P}^\vee_\bQ)^{\bP} \to 0,
    \end{align}
    which describes the kernel and the cokernel of the ensemble map. Here $\mathsf{Q}^\vee$ and $\mathsf{P}^\vee$ denote the coroot and coweight lattices of $\fsl_3$, respectively, and the subscript $\bQ$ stands for tensoring $\bQ$. The sequence \eqref{introeq:lam_exact} leads to a description of the $\fsl_3$-laminations around punctures by the root data of $\fsl_3$. 
    \item The action of the Weyl group $W(\fsl_3)$ on $\cL^x(\Sigma,\bQ)$ associated with each puncture. This is defined so that it gives the tropicalization of the Weyl group action on the moduli space $\X_{PGL_3,\Sigma}$ of framed $PGL_3$-local systems studied by Goncharov--Shen \cite{GS18}.
\end{itemize}

\subsection{Lamination exact sequence}
The key construction in the sequence \eqref{introeq:lam_exact} is the definition of the \emph{tropicalized Casimir map} $\theta=(\theta_p)_{p \in \bP}: \cL^x(\Sigma,\bQ) \to (\mathsf{P}^\vee_\bQ)^{\bP}$, which is introduced as the tropical analogue of the positive map $\Theta_p: \X_{PGL_3,\Sigma} \to H^{\bP}$ taking the semisimple part of the monodromy of $PGL_3$-local systems around each puncture $p$. The tropicalized Casimir map assigns a coweight $\theta_p(W)$ to any signed web $W$, and it turns out that it can be defined as the sum of the local contributions from each end of $W$ incident to $p$, as illustrated below:
\begin{align*}
\begin{tikzpicture}[scale=0.8]
    \draw[dashed, fill=white] (0,0) circle [radius=1];
    \draw[red,very thick,->-] (0,1) -- (0,0);
    \filldraw[fill=white](0,0) circle(2.5pt) node[above left,red] {$+$};
    \node at (0,-0.5) {$\varpi^\vee_2$};
    \begin{scope}[xshift=3cm]
    \draw[dashed, fill=white] (0,0) circle [radius=1];
    \draw[red,very thick,->-] (0,1) -- (0,0);
    \filldraw[fill=white](0,0) circle(2.5pt) node[above left,red] {$-$};
    \node at (0,-0.5) {$-\varpi^\vee_1$};
    \end{scope}
    \begin{scope}[xshift=6cm]
    \draw[dashed, fill=white] (0,0) circle [radius=1];
    \draw[red,very thick,-<-] (0,1) -- (0,0);
    \filldraw[fill=white](0,0) circle(2.5pt) node[above left,red] {$+$};
    \node at (0,-0.5) {$\varpi^\vee_1$};
    \end{scope}
    \begin{scope}[xshift=9cm]
    \draw[dashed, fill=white] (0,0) circle [radius=1];
    \draw[red,very thick,-<-] (0,1) -- (0,0);
    \filldraw[fill=white](0,0) circle(2.5pt) node[above left,red] {$-$};
    \node at (0,-0.5) {$-\varpi^\vee_2$};
    \end{scope}
\end{tikzpicture}
\end{align*}
On the `dual' side, the map $(\mathsf{Q}^\vee_\bQ)^{\bM} \to \cL^a(\Sigma,\bQ)$ is given by the action at each puncture adding/removing peripheral components. 

Here the reader is reminded that in the $\fsl_2$-case, the structure of unbounded laminations around a puncture is controlled by a single sign $\{+,-\}$ (or the `tag' in the sense of \cite{FST}) at each puncture, which corresponds to the root system $\{\alpha,-\alpha\}$ of $\fsl_2$. 

Recall that we have the \emph{cluster exact sequence} for a general cluster ensemble $(\A,\X)$ \cite{FG09} (see \cref{sec:appendix} for a brief review). 
Here is our first statement, which is obtained by describing \eqref{introeq:lam_exact} in terms of the cluster coordinates:

\begin{introthm}[\cref{thm:cluster_exact_sequence}]
The lamination exact sequence \eqref{introeq:lam_exact} coincides with the cluster exact sequence \eqref{eq:cluster_exact_seq_sl3}. In particular, it is equivariant under the action of the cluster modular group $\Gamma_{\fsl_3,\Sigma}$.
\end{introthm}

\subsection{Weyl group action}
The upshot of this paper is the correct definition of the Weyl group action $W(\fsl_3)$ on $\cL^x(\Sigma,\bQ)$. It should be defined so that 
\begin{enumerate}
    \item its coordinate expression is the tropical analogue of that of the Goncharov--Shen's action of $W(\fsl_3)$ on $\X_{PGL_3,\Sigma}$;
    \item it replaces each end of a given signed web $W$ according to some local rule;
    \item it makes the tropicalized Casimir map $\theta_p: \cL^x(\Sigma,\bQ) \to \mathsf{P}^\vee_\bQ$ equivariant, where the action on the target is the usual one on the coweight lattice.
\end{enumerate}
The requirements (2) and (3) uniquely determine the Weyl group action, as follows: 
\begin{align}
    \begin{aligned}\label{eq:local_rule1}
    r_{p,1}:\quad 
    &\tikz[baseline=-2pt, scale=.8]
    {\draw[dashed, fill=white] (0,0) circle [radius=1];
    \draw[red,thick,->-] (0,1) -- (0,0);
    \filldraw[fill=white](0,0) circle(2.5pt) node[above left,red,scale=0.8] {$+$};
    \draw[thick,|->] (1.5,0) -- (2.5,0);
    \begin{scope}[xshift=4cm]
    \draw[dashed, fill=white] (0,0) circle [radius=1];\draw[red,thick,->-] (0,1) -- (0,0);
    \filldraw[fill=white](0,0) circle(2.5pt) node[above left,red,scale=0.8] {$+$};
    \end{scope}
    }\ , \quad
    \tikz[baseline=-2pt, scale=.8]{
    \draw[dashed, fill=white] (0,0) circle [radius=1];
    \draw[red,thick,->-] (0,1) -- (0,0);
    \filldraw[fill=white](0,0) circle(2.5pt) node[above left,red,scale=0.8] {$-$};
    \draw[thick,|->] (1.5,0) -- (2.5,0);
    \begin{scope}[xshift=4cm]
    \draw[dashed, fill=white] (0,0) circle [radius=1];
    \draw[red,thick,->-={0.7}{}] (0,1) -- (0,0.6);
    \draw[red,thick,-<-] (0,0.6) arc(90:270:0.3);\draw[red,thick,-<-] (0,0.6) arc(90:-90:0.3);
    \filldraw[fill=white](0,0) circle(2.5pt) node[left=0.3em,red,scale=0.8] {$-$};\filldraw[fill=white](0,0) circle(2.5pt) node[right=0.3em,red,scale=0.8] {$+$};
    \end{scope}
    }\ ,\\[2mm]
    &\tikz[baseline=-2pt, scale=.8]{
    \draw[dashed, fill=white] (0,0) circle [radius=1];\draw[red,thick,-<-] (0,1) -- (0,0);
    \filldraw[fill=white](0,0) circle(2.5pt) node[above left,red,scale=0.8] {$+$};
    \draw[thick,|->] (1.5,0) -- (2.5,0);
    \begin{scope}[xshift=4cm]
    \draw[dashed, fill=white] (0,0) circle [radius=1];
    \draw[red,thick,-<-={0.7}{}] (0,1) -- (0,0.6);
    \draw[red,thick,->-] (0,0.6) arc(90:270:0.3);
    \draw[red,thick,->-] (0,0.6) arc(90:-90:0.3);
    \filldraw[fill=white](0,0) circle(2.5pt) node[left=0.3em,red,scale=0.8] {$-$};
    \filldraw[fill=white](0,0) circle(2.5pt) node[right=0.3em,red,scale=0.8] {$+$};
    \end{scope}
    }\ , \quad 
    \tikz[baseline=-2pt, scale=.8]{\draw[dashed, fill=white] (0,0) circle [radius=1];\draw[red,thick,-<-] (0,1) -- (0,0);
    \filldraw[fill=white](0,0) circle(2.5pt) node[above left,red,scale=0.8] {$-$};
    \draw[thick,|->] (1.5,0) -- (2.5,0);
    \begin{scope}[xshift=4cm]
    \draw[dashed, fill=white] (0,0) circle [radius=1];
    \draw[red,thick,-<-] (0,1) -- (0,0);
    \filldraw[fill=white](0,0) circle(2.5pt) node[above left,red,scale=0.8] {$-$};
    \end{scope}
    }\ .
    \end{aligned}
\end{align}

\begin{align}
    \begin{aligned}\label{eq:local_rule2}
    r_{p,2}:\quad
    &\tikz[baseline=-2pt, scale=.8]{
    \draw[dashed, fill=white] (0,0) circle [radius=1];\draw[red,thick,->-] (0,1) -- (0,0);
    \filldraw[fill=white](0,0) circle(2.5pt) node[above left,red,scale=0.8] {$+$};
    \draw[thick,|->] (1.5,0) -- (2.5,0);
    \begin{scope}[xshift=4cm]
    \draw[dashed, fill=white] (0,0) circle [radius=1];\draw[red,thick,->-={0.7}{}] (0,1) -- (0,0.6);\draw[red,thick,-<-] (0,0.6) arc(90:270:0.3);
    \draw[red,thick,-<-] (0,0.6) arc(90:-90:0.3);
    \filldraw[fill=white](0,0) circle(2.5pt) node[left=0.3em,red,scale=0.8] {$-$};\filldraw[fill=white](0,0) circle(2.5pt) node[right=0.3em,red,scale=0.8] {$+$};    
    \end{scope}
    }\ , \quad
    \tikz[baseline=-2pt, scale=.8]{
    \draw[dashed, fill=white] (0,0) circle [radius=1];
    \draw[red,thick,->-] (0,1) -- (0,0);
    \filldraw[fill=white](0,0) circle(2.5pt) node[above left,red,scale=0.8] {$-$};
    \draw[thick,|->] (1.5,0) -- (2.5,0);
    \begin{scope}[xshift=4cm]
    \draw[dashed, fill=white] (0,0) circle [radius=1];\draw[red,thick,->-] (0,1) -- (0,0);
    \filldraw[fill=white](0,0) circle(2.5pt) node[above left,red,scale=0.8] {$-$};
    \end{scope}
    }\ ,\\[2mm]
    &\tikz[baseline=-2pt, scale=.8]{
    \draw[dashed, fill=white] (0,0) circle [radius=1];
    \draw[red,thick,-<-] (0,1) -- (0,0);
    \filldraw[fill=white](0,0) circle(2.5pt) node[above left,red,scale=0.8] {$+$};
    \draw[thick,|->] (1.5,0) -- (2.5,0);
    \begin{scope}[xshift=4cm]    
    \draw[dashed, fill=white] (0,0) circle [radius=1];\draw[red,thick,-<-] (0,1) -- (0,0);
    \filldraw[fill=white](0,0) circle(2.5pt) node[above left,red,scale=0.8] {$+$};
    \end{scope}
    }\ , \quad 
    \tikz[baseline=-2pt, scale=.8]{
    \draw[dashed, fill=white] (0,0) circle [radius=1];
    \draw[red,thick,-<-] (0,1) -- (0,0);
    \filldraw[fill=white](0,0) circle(2.5pt) node[above left,red,scale=0.8] {$-$};
    \draw[thick,|->] (1.5,0) -- (2.5,0);
    \begin{scope}[xshift=4cm]
    \draw[dashed, fill=white] (0,0) circle [radius=1];\draw[red,thick,-<-={0.7}{}] (0,1) -- (0,0.6);\draw[red,thick,->-] (0,0.6) arc(90:270:0.3);\draw[red,thick,->-] (0,0.6) arc(90:-90:0.3);
    \filldraw[fill=white](0,0) circle(2.5pt) node[left=0.3em,red,scale=0.8] {$-$};\filldraw[fill=white](0,0) circle(2.5pt) node[right=0.3em,red,scale=0.8] {$+$};
    \end{scope}
    }\ .
    \end{aligned}
\end{align}

However, it turns out that such an action produces \emph{resolvable pairs} of ends at a puncture (the right-most one in \eqref{eq:puncture-admissible}) that we have prohibited in the previous paper \cite{IK22}. 
The actual problem is the following: as its name suggests, such a resolvable pair can be ``resolved'' into an arc away from puncture without changing its shear coordinates (\cref{def:resolution_move} Step 0). However, such an operation in turn produces possible intersections with other parts of the web, which are also prohibited. 
In view of the relation to the skein algebra \cite[Section 5]{IK22}, it would be natural to further ``resolve'' these additional intersections using a kind of skein relation. 

These observations lead us to a detailed study of resolution operations of the resolvable pairs and the additional intersections, without affecting the shear coordinates (\cref{subsec:Weyl_action}). Such an operation would be of independent interest, in view of the expectation that the signed webs should parametrize a basis of a certain skein algebra and that the equivalence relations among them pick up the ``highest terms" in some sense from appropriate skein relations. 

Eventually, we can resolve all such prohibited things to get a usual signed web. Then we obtain:
\begin{introthm}[{\cref{lem:Weyl_welldefined,thm:Weyl_action}}]
We have a well-defined action of $W(\fsl_3)^{\bM_\circ}$ on $\cL^x(\Sigma,\bQ)$ specified by the local rules in \eqref{eq:local_rule1} and \eqref{eq:local_rule2}. Moreover, it coincides with the tropicalization of the Goncharov--Shen's action \cite{GS18}.     
\end{introthm}

\subsection{Relation to the work of Shen--Sun--Weng}
In their recent work \cite{SSW}, Shen--Sun--Weng introduced a pairing between `$\A$-webs' and `$\X$-webs', which is expected to provide a topological construction of a duality pairing as described in \cite[Conjecture 4.3]{FG09}. 
In our terminology, their $\A$-webs correspond to the bounded integral $\fsl_3$-laminations without negative peripheral components. They also conjectured that the $\X$-webs corresponds to a subspace $\P_{PGL_3,\Sigma}^+(\bZ^T) \subset \P_{PGL_3,\Sigma}(\bZ^T)$ cut out by certain conditions \cite[Conjecture 4.10]{SSW}, where the latter denotes the set of integral tropical points of the Goncharov--Shen's moduli space $\P_{PGL_3,\Sigma}$ \cite{GS19}. 
We construct a bijection between the space $\mathscr{W}_\Sigma^\X$ of their $\X$-webs and a certain subspace $\cL^p(\Sigma,\bZ)_0$ of the integral $\P$-laminations \cite[Section 4]{IK22}, and give a proof of their conjecture based on the result \cite[Theorem 4.11]{IK22} in our previous work (together with \cref{prop:Casimir}). 

Since the definition of $\X$-webs involves no signs and pinnings, this relation would serve as a simpler understanding of the subset $\cL^p(\Sigma,\bZ)_0$. 
On the other hand, we certainly need signs (and pinnings) to describe the entire space $\P_{PGL_3,\Sigma}(\bZ^T)$. 
As the duality pairing should be equivariant under the Weyl group actions, our study on the Weyl group action will be a key ingredient to complete their construction to an entire pairing $\A_{PGL_3,\Sigma}(\bZ^T) \times \P_{PGL_3,\Sigma}(\bZ^T) \to \frac 1 3 \bZ$.

\subsection{Relation to the work of Fraser--Pylyavskyy}
In the case where $\Sigma$ has no punctures, we have seen that the unbounded integral $\fsl_3$-laminations give a basis for the (upper) cluster algebra related to the pair $(\fsl_3,\Sigma)$ \cite[Corollary 5.7]{IK22}, which is also the same as the function ring $\cO(\A_{SL_3,\Sigma}^\times)$ of (a certain generic part of) the moduli space of decorated $SL_3$-local systems \cite{FG03,IOS}. It is still natural to expect that our unbounded integral $\fsl_3$-laminations give a basis for one of these algebras in the presence of punctures. 

Fraser--Pylyavskyy \cite{FP21} gave an approach in this direction by representing certain (rational) functions on $\A_{SL_k,\Sigma}$ by \emph{pseudo-tagged diagrams}, which are webs equipped with certain decoration data at ends. In \cref{sec:FP}, we give a comparison between our signed webs and the pseudo-tagged diagrams for $k=3$. 

While the Fraser--Pylyavskyy's pseudo-tagged diagrams directly encodes the representation-theoretic decoration data of $SL_3$-local systems (\emph{i.e.}, a flat section of the associated $SL_3/U$-bundle), our data of signs at punctures would be thought of as a tropical analogue of the \emph{framings} of convex $\mathbb{RP}^2$-structures \cite{FG07c}. 

\subsection{Higher rank generalizations}
Finally, we discuss in brief higher rank generalizations, namely the unbounded $\mathfrak{g}$-laminations associated with any semisimple Lie algebra $\mathfrak{g}$. We propose to use the approach of Murakami--Ohtsuki--Yamada \cite{MOY} and Cautis--Kamnitzer--Morrison \cite{CKM} as a model of $\mathfrak{g}$-webs, and to define a rational unbounded $\mathfrak{g}$-lamination to be a certain equivalence class of the \emph{signed $\mathfrak{g}$-webs} with rational weights. Here, the signed $\mathfrak{g}$-webs are equipped with signs $\{+,-\}$ on their ends just as in the $\fsl_3$-case, and we do not expect any further decoration data needed since the colors by fundamental weights of $\mathfrak{g}$ and these signs are sufficient to exhaust the coweight lattice via the tropicalized Casimir map $\theta_p$ (\cref{fig:weight-contribution-general}). We also discuss their shear coordinates.

\subsection*{Organization of the paper}
Our notations on the marked surfaces and Lie theory are summarized in \cref{sec:notation}. In \cref{sec:lamination}, we recall the definition of unbounded $\fsl_3$-laminations and their shear coordinates.

\cref{sec:ensemble} gives a classification of unbounded $\fsl_3$-laminations around a puncture. In \cref{subsec:cluster_exact_seq}, we describe our constructions in the shear coordinates, and relate them to the cluster exact sequence coming from the structure of cluster ensemble \cite{FG09}. 

\cref{subsec:Weyl_action} is the most involved part, which describes the Weyl group action on the unbounded $\fsl_3$-laminations. Some of the technical statements are proved in \cref{subsec:proof_resolution}. 

In \cref{sec:SSW}, 
we give a proof of the Shen--Sun--Weng's conjecture by using the fact that our tropicalized Casimir map $\theta_p$ is indeed the tropicalization of the positive map $\Theta_p$ on the moduli space $\P_{PGL_3,\Sigma}$. 

In \cref{sec:FP}, we investigate several relations to the work of Fraser--Pylyavskyy. We relate our signed webs with their `pseudo-tagged diagrams'. 
Finally in \cref{sec:higher}, we briefly discuss the higher rank generalizations. 

The following diagram shows the relationships among Sections 2--8.
\[
\begin{tikzcd}[cells={nodes={draw, rounded corners}}, row sep=small, arrows={shorten=1mm}]
2 \ar[r] & 3 \ar[r] \ar[rd] \ar[rdd] & 4 \ar[r] & 5 \ar[rd]\\
&& 6 && 7 \\
&& 8
\end{tikzcd}
\]

\subsection*{Acknowledgements}
The authors are grateful to Wataru Yuasa and Zhe Sun for valuable discussions on the $\fsl_3$-skein relations at punctures. 
T. I. is partially supported by JSPS KAKENHI (20K22304).
S. K. is partially supported by scientific research support of Research Alliance Center for Mathematical Sciences, Tohoku University.

%% file: 2_webs.tex
\section{Notation}\label{sec:notation}

\subsection{Marked surfaces and their triangulations}\label{subsec:notation_marked_surface}
A marked surface $(\Sigma,\bM)$ is a compact oriented surface $\Sigma$ together with a fixed non-empty finite set $\bM \subset \Sigma$ of \emph{marked points}. 
When the choice of $\bM$ is clear from the context, we simply denote a marked surface by $\Sigma$. 
A marked point is called a \emph{puncture} if it lies in the interior of $\Sigma$, and a \emph{special point} otherwise. 
Let $\bP$ (resp. $\bM_\partial$) denote the set of punctures (resp. special points), so that $\bM=\bP \sqcup \bM_\partial$. 
Let $\Sigma^*:=\Sigma \setminus \bP$. 
We always assume the following conditions:
\begin{enumerate}
    \item[(S1)] Each boundary component (if exists) has at least one marked point.
    \item[(S2)] $-2\chi(\Sigma^*)+|\bM_\partial| >0$.
    \item[(S3)] $(\Sigma,\bM)$ is not a once-punctured disk with a single special point on the boundary.
\end{enumerate}
 We call a connected component of the punctured boundary $\partial^\ast \Sigma:=\partial\Sigma\setminus \bM_\partial$ a \emph{boundary interval}. The set of boundary intervals is denoted by $\bB$.

Unless otherwise stated, an \emph{isotopy} in a marked surface $(\Sigma,\bM)$ means an ambient isotopy in $\Sigma$ relative to $\bM$, which preserves each boundary interval setwisely. 
An \emph{ideal arc} in $(\Sigma,\bM)$ is an immersed arc in $\Sigma$ with endpoints in $\bM$ which has no self-intersection except possibly for its endpoints, and is not isotopic to one point.

An \emph{ideal triangulation} is a triangulation $\tri$ of $\Sigma$ whose set of $0$-cells (vertices) coincides with $\bM$. 
The conditions (S1) and (S2) ensure the existence of such an ideal triangulation, and the positive integer in (S2) gives the number of $2$-cells (triangles). In this paper, we always consider an ideal triangulation without \emph{self-folded triangles} of the form
\begin{align*}
\begin{tikzpicture}[scale=0.7]
    \draw [blue] (0,0) -- (0,-2);
	\draw [blue] (0,-2) .. controls (-1.3,-0.5) and (-1.3,1) .. (0,1) .. controls (1.3,1) and (1.3,-0.5) .. (0,-2);
    \node [fill, circle, inner sep=1.5pt] at (0,0) {};
    \node [fill, circle, inner sep=1.5pt] at (0,-2) {};
\end{tikzpicture}
\end{align*}
whose existence is ensured by the condition (S3). See, for instance, \cite[Lemma 2.13]{FST}.
For an ideal triangulation $\tri$, denote the set of edges (resp. interior edges, triangles) of $\tri$ by $e(\tri)$ (resp. $e_{\interior}(\tri)$, $t(\tri)$).
Since the boundary intervals belong to any ideal triangulation, $e(\tri)=e_{\interior}(\tri) \sqcup \bB$.
By a computation on the Euler characteristics, we get
\begin{align*}
    &|e(\tri)|=-3\chi(\Sigma^*)+2|\bM_\partial|, \quad |e_{\interior}(\tri)|=-3\chi(\Sigma^*)+|\bM_\partial|, \\
    &|t(\tri)|=-2\chi(\Sigma^*)+|\bM_\partial|.
\end{align*}

\begin{figure}[h]
\begin{tikzpicture}
\draw[blue] (2.5,0) -- (0,2.5) -- (-2.5,0) -- (0,-2.5) --cycle;
\draw[blue, postaction={decorate}, decoration={markings,mark=at position 0.5 with {\arrow[scale=1.5,>=stealth]{>}}}] (0,-2.5) -- node[midway,right,scale=0.8] {$E$} (0,2.5);
\quiversquare{0,-2.5}{2.5,0}{0,2.5}{-2.5,0};
\draw(x131) node[mygreen, right=0.2em,scale=0.9]{$i^1(E)$};
\draw(x132) node[mygreen, left=0.2em,scale=0.9]{$i^2(E)$};
\end{tikzpicture}
    \caption{The set $I(\tri)$ of distinguished points.}
    \label{fig:sl3_triangulation}
\end{figure}
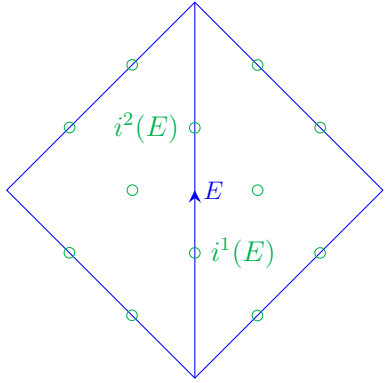

To indicate the $\fsl_3$-structure, it is useful to equip $\tri$ with two distinguished points on the interior of each edge and one point in the interior of each triangle, as shown in \cref{fig:sl3_triangulation}. The set of such points is denoted by $I(\tri)=I_{\mathfrak{sl}_3}(\tri)$. This set will give the vertex set of the quiver $Q^\tri$ associated with $\tri$: see \cref{subsec:cluster_sl3}. 
Let $I^{\mathrm{edge}}(\tri)$ (resp. $I^{\mathrm{tri}}(\tri)$) denote the set of points on the interiors of edges (resp. triangles) so that $I(\tri)=I^{\mathrm{edge}}(\tri) \sqcup I^{\mathrm{tri}}(\tri)$, where we have a canonical bijection 
\begin{align*}
    t(\tri) \xrightarrow{\sim} I^{\mathrm{tri}}(\tri), \quad T \mapsto i(T).
\end{align*}
When we need to label the two vertices on an edge $E \in e(\tri)$, we endow it with an orientation. Then let $i^1(E) \in I(\tri)$ (resp. $i^2(E) \in I(\tri)$) denote the vertex closer to the initial (resp. terminal) endpoint of $E$. 
Let $I(\tri)_\f \subset I^\mathrm{edge}(\tri)$ (\lq\lq frozen'') be the subset consisting of the points on the boundary, and let $I(\tri)_\uf:=I(\tri) \setminus I(\tri)_\f$ (\lq\lq unfrozen'').
The numbers
\begin{align*}
    |I(\tri)|= 2|e(\tri)| + |t(\tri)| =-8\chi(\Sigma^*)+5|\bM_\partial|
\end{align*}
and 
\begin{align*}
    |I(\tri)_\uf|= 2|e_{\interior}(\tri)| + |t(\tri)| =-8\chi(\Sigma^*)+3|\bM_\partial|
\end{align*}
will give the dimensions of the PL manifolds $\cL^a(\Sigma,\bR)$ and $\cL^x(\Sigma,\bR)$ respectively, which we will study in this paper. 

\subsection{Notation from Lie theory}\label{subsec:Lie_theory}
For a complex semisimple finite-dimensional Lie algebra $\mathfrak{g}$, choose a Cartan subalgebra $\mathfrak{h}$ and a set of simple roots. Let $S$ be the vertex set of the associated Dynkin diagram. 
Then we have the following lattices:
\begin{itemize}
    \item $\mathsf{Q}= \bigoplus_{s \in S} \bZ \alpha_s \subset \mathfrak{h}^\ast$ (root lattice),
    \item $\mathsf{P}= \bigoplus_{s \in S} \bZ \varpi_s \subset \mathfrak{h}^\ast$ (weight lattice), 
    \item $\mathsf{Q}^\vee= \bigoplus_{s \in S} \bZ \alpha^\vee_s \subset \mathfrak{h}$ (coroot lattice),
    \item $\mathsf{P}^\vee= \bigoplus_{s \in S} \bZ \varpi^\vee_s \subset \mathfrak{h}$ (coweight lattice).
\end{itemize}
Here $\alpha_s$, $\varpi_s$, $\alpha_s^\vee$, $\varpi^\vee_s$ are called the $s$-th simple root, fundamental weight, simple coroot, fundamental coweight, respectively for $s \in S$. We have the relations
\begin{align*}
    \alpha_t= \sum_s C_{st} \varpi_s \quad \mbox{and}\quad \alpha^\vee_t= \sum_s C_{ts} \varpi^\vee_s,
\end{align*}
where $C(\mathfrak{g})=(C_{st})_{s,t \in S}$ is the Cartan matrix of $\mathfrak{g}$. The relations among the lattices above are summarized as follows:
\begin{align*}
    \mathsf{Q}^\vee = \Hom(\mathsf{P},\bZ), \quad \mathsf{P}^\vee = \Hom(\mathsf{Q},\bZ), \quad
    \mathsf{Q} \subset \mathsf{P}, \quad \mathsf{Q}^\vee \subset \mathsf{P}^\vee.
\end{align*}
Here the inclusions are of index $\det C(\mathfrak{g})$. For a lattice $L$, let $L_\bQ:=L\otimes_\bZ \bQ$ denote the associated $\bQ$-vector space. 

Let $G$ (resp.~$G'$) be the simply-connected (resp.~adjoint) algebraic group with the Lie algebra $\mathfrak{g}$. Let $H \subset G$ be a Cartan subgroup, and $H' \subset G'$ be its image under the canonical projection. Then the lattices above are interpreted as (co)character lattices:
\begin{align*}
    &\mathsf{P}=X^\ast(H), & &\mathsf{Q}^\vee=X_\ast(H), \\ 
    &\mathsf{Q}=X^\ast(H'), & &\mathsf{P}^\vee=X_\ast(H').
\end{align*}
Here $X^\ast(T):=\Hom(T,\bG_m)$ and $X_\ast(T):=\Hom(\bG_m,T)$ for a split algebraic torus $T$. The cocharacter lattice $X_\ast(T)$ can be viewed as the tropical analogue of $T$ (cf.~\cite{FG09}). 

\smallskip
\paragraph{\textbf{The $\fsl_3$-case}} The groups are $G=SL_3$ and $G'=PGL_3$. The standard choice of Cartan subgroup is $H:=\{\mathrm{diag}(x,y,z) \mid xyz=1\} \subset SL_3$, leading to $\mathsf{P}=\bZ^3/\bZ(1,1,1)$ and $\mathsf{Q}^\vee=\{(a,b,c) \mid a+b+c=0\} \subset \bZ^3$. 
\begin{itemize}
    \item Simple positive roots:
    \begin{align*}
        \alpha_1=\mathrm{diag}(1,-1,0), \quad 
        \alpha_2=\mathrm{diag}(0,1,-1).
    \end{align*}
    \item Fundamental weights:
    \begin{align*}
        \varpi_1=\mathrm{diag}(2/3,-1/3,-1/3), \quad 
        \varpi_2=\mathrm{diag}(1/3,1/3,-2/3).
    \end{align*}
    \item The Cartan matrix:
    \begin{align*}
        C(\fsl_3)=\begin{pmatrix}  
        2 & -1 \\
        -1& 2
        \end{pmatrix}.
    \end{align*}
\end{itemize}

\section{Unbounded \texorpdfstring{$\fsl_3$}{sl(3)}-laminations}\label{sec:lamination}
\subsection{Definition}
Let us briefly recall the definition of the rational unbounded $\fsl_3$-laminations from \cite[Section 2]{IK22}. 

An $\mathfrak{sl}_3$-web (or simply a \emph{web}) on a marked surface $\Sigma$ is an immersed oriented uni-trivalent graph $W$ on $\Sigma$ such that each univalent vertex lies in $\bP \cup \partial^\ast \Sigma$, and the other part is embedded into $\interior \Sigma^\ast$. Here, the orientation is given so that each trivalent vertex is either a \emph{sink} or a \emph{source}, respectively:
\begin{align*}
    \begin{tikzpicture}[scale=0.8]
        \draw[dashed] (0,0) circle(1cm);
        {\color{red} \sink{-90:1}{30:1}{150:1};}
        \begin{scope}[xshift=5cm]
        \draw[dashed] (0,0) circle(1cm);
        {\color{red} \source{-90:1}{30:1}{150:1};}
        \end{scope}
    \end{tikzpicture}
\end{align*}
We also allow oriented loop components. 
A web is said to be \emph{non-elliptic} if it has none of the following \emph{elliptic faces}:
\begin{align}
    &\begin{tikzpicture}[scale=.8]
		\draw[dashed, fill=white] (0,0) circle [radius=1];
		\draw[very thick, red, fill=pink!60] (0,0) circle [radius=0.4];
	\end{tikzpicture}
	\hspace{2em}
	\begin{tikzpicture}[scale=.8]
		\draw[dashed, fill=white] (0,0) circle [radius=1];
		\draw[very thick, red] (0,-1) -- (0,-0.4);
		\draw[very thick, red] (0,0.4) -- (0,1);
		\draw[very thick, red, fill=pink!60] (0,0) circle [radius=0.4];
	\end{tikzpicture}
	\hspace{2em}
	\begin{tikzpicture}[scale=.8]
		\draw[dashed, fill=white] (0,0) circle [radius=1];
		\draw[very thick, red] (-45:1) -- (-45:0.5);
		\draw[very thick, red] (-135:1) -- (-135:0.5);
		\draw[very thick, red] (45:1) -- (45:0.5);
		\draw[very thick, red] (135:1) -- (135:0.5);
		\draw[very thick, red, fill=pink!60] (45:0.5) -- (135:0.5) -- (225:0.5) -- (315:0.5) -- cycle;
	\end{tikzpicture} \label{eq:elliptic face}\\
	&\begin{tikzpicture}[scale=0.8,baseline=-0.7cm]
		\coordinate (P) at (-0.4,0) {};
		\coordinate (P') at (0.4,0) {};
		\coordinate (C) at (90:0.4) {};
		\draw[very thick, red, fill=pink!60] (P) to[out=north, in=west] (C) to[out=east, in=north] (P');
		\draw[dashed] (1,0) arc (0:180:1cm);
		\bline{-1,0}{1,0}{0.2}
		\draw[fill=black] (-0.7,0) circle(2pt);
		\draw[fill=black] (0.7,0) circle(2pt);
	\end{tikzpicture}
	\hspace{2em}
	\begin{tikzpicture}[scale=0.8,baseline=-0.7cm]
		\coordinate (P) at (-0.4,0) {};
		\coordinate (P') at (0.4,0) {};
		\coordinate (C) at (90:0.4) {};
		\draw[very thick, red, fill=pink!60] (P) to[out=north, in=west] (C) to[out=east, in=north] (P');
		\draw[very thick, red] (C) -- (90:1);
		\draw[dashed] (1,0) arc (0:180:1cm);
		\bline{-1,0}{1,0}{0.2}
		\draw[fill=black] (-0.7,0) circle(2pt);
		\draw[fill=black] (0.7,0) circle(2pt);
	\end{tikzpicture}
	\hspace{2em}
	\begin{tikzpicture}[scale=.8]
		\draw[dashed, fill=white] (0,0) circle [radius=1];
		\draw[red,very thick,fill=pink!60] (0,0.3) circle[radius=0.3]; 
		\draw[fill=white] (0,0) circle(2pt);
	\end{tikzpicture}
	\label{eq:elliptic_face_peripheral}
\end{align}
A web is said to be \emph{bounded} if none of its univalent vertices lie in $\bP$.

A \emph{signed web} is a web on $\Sigma$ together with a sign ($+$ or $-$) assigned to each end incident to a puncture. The following patterns (and their orientation-reversals) of signed ends are called \emph{bad ends}:
\begin{align}\label{eq:puncture-admissible}
\begin{tikzpicture}
\draw[dashed] (0,0) circle(1cm);
\draw[red,very thick,->-] (-1,0) -- (0,0) node[below left]{$-\epsilon$};
\draw[red,very thick,-<-] (1,0) -- (0,0) node[below right]{$\epsilon$};
\filldraw[fill=white](0,0) circle(2pt);
\begin{scope}[xshift=4cm]
\draw[dashed] (0,0) circle(1cm);
\draw[red,very thick,-<-](0,0) arc(-90:90:0.3cm);
\draw[red,very thick,-<-](0,0) arc(-90:-270:0.3cm);
\draw[red,very thick,->-] (0,0.6) -- (0,1);
\node[red] at (-0.4,0) {$\epsilon$};
\node[red] at (0.4,0) {$\epsilon$};
\filldraw[fill=white](0,0) circle(2pt);
\end{scope}
\begin{scope}[xshift=8cm]
\draw[dashed] (0,0) circle(1cm);
\draw[red,very thick,->-](0,0) arc(-90:90:0.3cm);
\draw[red,very thick,->-](0,0) arc(-90:-270:0.3cm);
\draw[red,very thick,-<-] (0,0.6) -- (0,1);
\node[red] at (-0.4,0) {$\epsilon$};
\node[red] at (0.4,0) {$\epsilon$};
\filldraw[fill=white](0,0) circle(2pt);
\end{scope}
\end{tikzpicture}
\end{align}
Here $\epsilon \in \{+,-\}$. The pair of ends of the first type is called the \emph{resolvable pair}. 
A signed web is said to be \emph{admissible} if it has no bad ends. 
Unless otherwise stated, we always assume that the signed webs are admissible (until \cref{subsec:Weyl_action}). 
A bounded web is naturally regarded as a signed web, since we do not need to specify any signs. 


\paragraph{\textbf{Elementary moves of signed webs.}}
We are going to introduce several elementary moves for signed webs. The first two are defined for a web without signs.

\begin{enumerate}
\item[(E1)] Loop parallel-move (a.~k.~a.~\emph{flip move} \cite{FS20} or \emph{global parallel move} \cite{DS20I}):  
\begin{align}\label{eq:loop_parallel-move}
\begin{tikzpicture}[scale=1]
\draw(0,0) -- (2,0);
\draw(0,1.5) -- (2,1.5);
\draw[red,very thick,->-] (0.7,0.75) [partial ellipse=90:-90:0.2cm and 0.75cm];
\draw[red,very thick,dashed] (0.7,0.75) [partial ellipse=-90:-270:0.2cm and 0.75cm];
\draw[red,very thick,-<-] (1.3,0.75) [partial ellipse=90:-90:0.2cm and 0.75cm];
\draw[red,very thick,dashed] (1.3,0.75) [partial ellipse=-90:-270:0.2cm and 0.75cm];
\node[scale=1.5] at (3,0.75) {$\sim$};
\begin{scope}[xshift=4cm]
\draw(0,0) -- (2,0);
\draw(0,1.5) -- (2,1.5);
\draw[red,very thick,-<-] (0.7,0.75) [partial ellipse=90:-90:0.2cm and 0.75cm];
\draw[red,very thick,dashed] (0.7,0.75) [partial ellipse=-90:-270:0.2cm and 0.75cm];
\draw[red,very thick,->-] (1.3,0.75) [partial ellipse=90:-90:0.2cm and 0.75cm];
\draw[red,very thick,dashed] (1.3,0.75) [partial ellipse=-90:-270:0.2cm and 0.75cm];
\end{scope}
\end{tikzpicture}
\end{align}

\item[(E2)] Boundary H-move:
\begin{align}\label{eq:boundary_H-move}
\begin{tikzpicture}
\filldraw[pink!60] (0.4,0) -- (0.4,0.4) -- (-0.4,0.4) -- (-0.4,0) --cycle;
\draw[very thick,red,-<-] (0.4,0) -- (0.4,0.4);
\draw[very thick,red,->-] (-0.4,0) -- (-0.4,0.4);
\draw[very thick,red,->-] (0.4,0.4) -- (0.4,0.917);
\draw[very thick,red,-<-] (-0.4,0.4) -- (-0.4,0.917);
\draw[very thick,red,->-] (0.4,0.4) -- (-0.4,0.4);
\draw[dashed] (1,0) arc (0:180:1cm);
\bline{-1,0}{1,0}{0.2}
\draw[fill=black] (-0.8,0) circle(2pt);
\draw[fill=black] (0.8,0) circle(2pt);
\node[scale=1.5] at (2,0.5) {$\sim$};
\begin{scope}[xshift=4cm]
\draw[very thick,red,->-] (0.4,0) -- (0.4,0.917);
\draw[very thick,red,-<-] (-0.4,0) -- (-0.4,0.917);
\draw[dashed] (1,0) arc (0:180:1cm);
\bline{-1,0}{1,0}{0.2}
\draw[fill=black] (-0.8,0) circle(2pt);
\draw[fill=black] (0.8,0) circle(2pt);
\end{scope}
\end{tikzpicture}
\end{align}
Similarly for the opposite orientation. We call the face in the left-hand side a \emph{boundary H-face}.

\item[(E3)] Puncture H-moves:
\begin{align}\label{eq:puncture_H-move_1}
\begin{tikzpicture}
\draw[dashed, fill=white] (0,0) circle [radius=1];
\draw(60:0.6) coordinate(A);
\draw(120:0.6) coordinate(B);
\fill[pink!60] (A) -- (B) -- (0,0) --cycle; 
\draw[red,very thick,-<-={0.6}{}] (0,0)--(A);
\draw[red,very thick,->-={0.6}{}] (0,0)--(B);
\draw[red,very thick,->-={0.6}{}] (A)--(60:1);
\draw[red,very thick,-<-={0.6}{}] (B)--(120:1);
\draw[red,very thick,->-] (A) -- (B);
\node[red,scale=0.8,anchor=west] at (0.1,0) {$\epsilon$};
\node[red,scale=0.8,anchor=east] at (-0.1,0) {$\epsilon$};
\draw[fill=white] (0,0) circle(2pt);
\node[scale=1.5] at (2,0) {$\sim$};
\begin{scope}[xshift=4cm]
\draw[dashed, fill=white] (0,0) circle [radius=1];
\draw(60:0.6) coordinate(A);
\draw(120:0.6) coordinate(B);
\draw[red,very thick,->-={0.6}{}] (0,0)--(60:1);
\draw[red,very thick,-<-={0.6}{}] (0,0)--(120:1);
\node[red,scale=0.8,anchor=west] at (0.1,0) {$\epsilon$};
\node[red,scale=0.8,anchor=east] at (-0.1,0) {$\epsilon$};
\draw[fill=white] (0,0) circle(2pt);
\end{scope}
\end{tikzpicture}
\end{align}
for $\epsilon \in \{+,-\}$, and
\begin{align}\label{eq:puncture_H-move_2}
\begin{tikzpicture}
\draw[dashed] (0,0) circle(1cm);
\filldraw[draw=red,very thick,->-={0.7}{red},-<-={0.42}{red},fill=pink!60] (0,0) ..controls (0.5,0.1) and (0.2,0.4).. (0,0.4) ..controls (-0.2,0.4) and (-0.5,0.1).. (0,0);
\draw[red,very thick,->-={0.7}{}] (0,0.4) -- (0,0.7);
\draw[red,very thick,-<-] (0,0.7) -- (60:1);
\draw[red,very thick,-<-] (0,0.7) -- (120:1);
\node[red,scale=0.8,anchor=east] at (-0.1,-0.1) {$+$};
\node[red,scale=0.8,anchor=west] at (0.1,-0.1) {$-$};
\draw[fill=white] (0,0) circle(2pt);
\node[scale=1.5] at (2,0) {$\sim$};
\begin{scope}[xshift=4cm]
\draw[dashed] (0,0) circle(1cm);
\filldraw[draw=red,very thick,->-={0.7}{red},-<-={0.42}{red},fill=pink!60] (0,0) ..controls (0.5,0.1) and (0.2,0.4).. (0,0.4) ..controls (-0.2,0.4) and (-0.5,0.1).. (0,0);
\draw[red,very thick,->-={0.7}{}] (0,0.4) -- (0,0.7);
\draw[red,very thick,-<-] (0,0.7) -- (60:1);
\draw[red,very thick,-<-] (0,0.7) -- (120:1);
\node[red,scale=0.8,anchor=east] at (-0.1,-0.1) {$-$};
\node[red,scale=0.8,anchor=west] at (0.1,-0.1) {$+$};
\draw[fill=white] (0,0) circle(2pt);
\node[scale=1.5] at (2,0) {$\sim$};
\end{scope}
\begin{scope}[xshift=8cm]
\draw[dashed] (0,0) circle(1cm);
\draw[red,very thick,-<-] (0,0) -- (60:1);
\draw[red,very thick,-<-] (0,0) -- (120:1);
\node[red,scale=0.8,anchor=east] at (-0.1,0) {$+$};
\node[red,scale=0.8,anchor=west] at (0.1,0) {$-$};
\draw[fill=white] (0,0) circle(2pt);
\end{scope}
\end{tikzpicture}
\end{align}
Similarly for the opposite orientation. We call the face in the left-hand side of \eqref{eq:puncture_H-move_1} a \emph{puncture H-face}. 
\end{enumerate}
The moves (E2) and the first one in (E3) lead to the parallel-moves for the arcs. See \cite[Lemma 2.4]{IK22}.

Also note that we can always transform any signed web to a signed web without boundary H-faces (resp. puncture H-faces) by applying (E2) and (E3), respectively. Slightly generalizing the terminology in \cite{FS20}, such a signed web is said to be \emph{boundary-reduced} (resp. \emph{puncture-reduced}). It is said to be \emph{reduced} if it is both boundary- and puncture-reduced.

\begin{enumerate}
    \item [(E4)] Peripheral move: removing or creating a peripheral component:
    \begin{align}\label{eq:peripheral}
    \begin{tikzpicture}[scale=.8]
    \draw[dashed, fill=white] (0,0) circle [radius=1];
    \draw[very thick, red, fill=pink!60] (0,0) circle [radius=0.5];
    \filldraw[draw=black,fill=white] (0,0) circle(2.5pt); 
    \begin{scope}[xshift=5cm]
    \coordinate (P) at (-0.5,0) {};
    \coordinate (P') at (0.5,0) {};
    \coordinate (C) at (0,0.5) {};
    \draw[very thick, red, fill=pink!60] (P) to[out=north, in=west] (C) to[out=east, in=north] (P');
    \draw[dashed] (1,0) arc (0:180:1cm);
    \bline{-1,0}{1,0}{0.2}
    \draw[fill=black] (0,0) circle(2pt);
    \end{scope}
    \end{tikzpicture}
    \end{align}
    Moreover, we have the moves
    \begin{align*}
    \begin{tikzpicture}
    \draw[dashed, fill=white] (0,0) circle [radius=1];
    \fill[pink!60] (0,0) circle(0.6cm);
    \draw[red,very thick,->-={0.8}{}] (0,0) circle(0.6cm);
    \foreach \i in {120,60,0}
    { 
        \draw[red,very thick,-<-={0.8}{}] (\i+15:0) -- (\i+15:0.6);
        \draw[red,very thick,-<-] (\i:0.6) -- (\i:1);
    }
    \node[red,scale=0.8] at (-0.2,0) {$+$};
	\node[red,scale=0.8] at (-0.05,0.3) {$+$};
	\node[red,scale=0.8] at (0.25,0.2) {$+$};
    \draw[red,very thick,dotted] (-30:0.3) arc(-30:-90:0.3);
    \draw[fill=white] (0,0) circle(2pt);
    \node[scale=1.7] at (2,0) {$\sim$};
    \begin{scope}[xshift=4cm]
    \draw[dashed, fill=white] (0,0) circle [radius=1];
    \foreach \i in {120,60,0} \draw[red,very thick,-<-={0.7}{}] (\i:0) -- (\i:1);
    \node[red,scale=0.8] at (-0.2,0.1) {$+$};
    \node[red,scale=0.8] at (0,0.3) {$+$};
    \node[red,scale=0.8] at (0.3,0.15) {$+$};
    \draw[red,very thick,dotted] (-30:0.3) arc(-30:-90:0.3);
    \draw[fill=white] (0,0) circle(2pt);
    \end{scope}
    {\begin{scope}[xshift=8cm]
    \draw[dashed, fill=white] (0,0) circle [radius=1];
    \fill[pink!60] (0,0) circle(0.6cm);
    \draw[red,very thick,->-={0.8}{}] (0,0) circle(0.6cm);
    \foreach \i in {120,60,0}
    { 
        \draw[red,very thick,->-={0.85}{}] (\i-15:0) -- (\i-15:0.6);
        \draw[red,very thick,->-] (\i:0.6) -- (\i:1);
    }
    \node[red,scale=0.8] at (-0.2,0.1) {$-$};
    \node[red,scale=0.8] at (0.1,0.3) {$-$};
    \node[red,scale=0.8] at (0.25,0.05) {$-$};
    \draw[red,very thick,dotted] (-45:0.3) arc(-45:-105:0.3);
    \draw[fill=white] (0,0) circle(2pt);
    \node[scale=1.7] at (2,0) {$\sim$};
    \end{scope}}
    {\begin{scope}[xshift=12cm]
    \draw[dashed, fill=white] (0,0) circle [radius=1];
    \foreach \i in {120,60,0} \draw[red,very thick,->-={0.7}{}] (\i:0) -- (\i:1);
    \node[red,scale=0.8] at (-0.3,0.1) {$-$};
    \node[red,scale=0.8] at (0,0.3) {$-$};
    \node[red,scale=0.8] at (0.3,0.1) {$-$};
    \draw[fill=white] (0,0) circle(2pt);
    \draw[red,very thick,dotted] (-30:0.3) arc(-30:-90:0.3);
    \end{scope}}
    \end{tikzpicture}
    \end{align*}
    Similarly for the opposite orientation.
\end{enumerate}
We will consider the equivalence relation on signed webs generated by isotopies of marked surfaces and the elementary moves (E1)--(E4). Observe that the moves (E1)--(E4) preserves the admissibility. On the other hand, a non-elliptic signed web may be equivalent to an elliptic web. See \cite[Example 2.5]{IK22} for an example. 

\begin{dfn}[rational unbounded $\mathfrak{sl}_3$-laminations]\label{def:unbounded laminations}
A \emph{rational unbounded $\mathfrak{sl}_3$-lamination} (or a rational $\mathfrak{sl}_3$-$\X$-lamination) on $\Sigma$ is an admissible, non-elliptic signed $\mathfrak{sl}_3$-web $W$ on $\Sigma$ equipped with a positive rational number (called the \emph{weight}) on each component, which is considered modulo the equivalence relation generated by isotopies and the following operations:
\begin{enumerate}
    \item Elementary moves (E1)--(E4) for the underlying signed webs. Here the corresponding components are assumed to have the same weights. 
    \item Combine a pair of isotopic loops with the same orientation with weights $u$ and $v$ into a single loop with the weight $u+v$. Similarly combine a pair of isotopic oriented arcs with the same orientation (and with the same signs if some of their ends are incident to punctures) into a single one by adding their weights. 
    \item For an integer $n \in \bZ_{>0}$ and a rational number $u \in \bQ_{>0}$, replace a component with weight $nu$ with its \emph{$n$-cabling} with weight $u$, which locally looks like
    \begin{align*}
        \begin{tikzpicture}[scale=1.2]
        \draw[dashed] (0,0) circle(0.76cm);
        \foreach \i in {30,150,270}
        \draw[red,very thick] (0,0) -- (\i:0.76);
        \node[red] at (0,0.3) {$nu$};
        \node at (1.5,0) {\scalebox{1.2}{$\sim$}}; 
        \begin{scope}[xshift=3cm]
        \draw[red, very thick] (-30:0.4) -- (90:0.4) -- (210:0.4) --cycle;
        \foreach \i in {1,4}
        {
        \draw[red,very thick] ($(-30:0.4)!\i/5!(90:0.4)$) --++(30:0.5) coordinate(A\i);
        \draw[red,very thick] ($(90:0.4)!\i/5!(210:0.4)$) --++(150:0.5) coordinate(B\i);
        \draw[red,very thick] ($(210:0.4)!\i/5!(-30:0.4)$) --++(-90:0.5) coordinate(C\i);
        }
        \draw[red,very thick,dotted] (15:0.5) -- (45:0.5);
        \draw[red,very thick,dotted] (135:0.5) -- (165:0.5);
        \draw[red,very thick,dotted] (255:0.5) -- (285:0.5);
        \draw[decorate,decoration={brace,amplitude=3pt,raise=4pt}] (A4) --node[midway,xshift=12pt,yshift=8pt]{$n$} (A1);
        \draw[decorate,decoration={brace,amplitude=3pt,raise=4pt}] (B4) --node[midway,xshift=-12pt,yshift=8pt]{$n$} (B1);
        \draw[decorate,decoration={brace,amplitude=3pt,raise=4pt}] (C4) --node[midway,below=7pt]{$n$} (C1);
        \node[red] at (0,0.6) {$u$};
        \draw[dashed] (0,0) circle(0.76cm);
        \end{scope}
        \begin{scope}[xshift=6cm]
        \draw[dashed] (0,0) circle(0.76cm);
        \node at (1.5,0) {\scalebox{1.2}{$\sim$}};
        \clip (0,0) circle(0.76cm);
        \draw[red,very thick,->-={0.8}{}] (-1,0) --node[midway,above]{$nu$} (1,0);
        \end{scope}
        \begin{scope}[xshift=9cm]
        \draw[dashed] (0,0) circle(0.76cm);
        \draw[decorate,decoration={brace,amplitude=3pt,raise=2pt}] (0.76,0.2) --node[midway,right=0.3em]{$n$} (0.76,-0.2);
        \clip (0,0) circle(0.76cm);
        \draw[red,very thick,->-={0.8}{}] (-1,0.2) --node[above]{$u$} (1,0.2);
        \draw[red,very thick,dotted] (0,0.2) -- (0,-0.2);
        \draw[red,very thick,->-={0.8}{}] (-1,-0.2) --node[below]{$u$} (1,-0.2);
        \end{scope}
        \end{tikzpicture}
    \end{align*}
    For a loop or arc component, it is just the $n-1$ times applications of the operation (2). One can also verify that the cabling operation is associative in the sense that the $n$-cabling followed by the $m$-cabling agrees with the $nm$-cabling. 
\end{enumerate}
\end{dfn}

Let $\cL^x(\Sigma,\bQ)$ denote the set of equivalence classes of the rational unbounded $\mathfrak{sl}_3$-laminations on $\Sigma$. We have a natural $\bQ_{>0}$-action on $\cL^x(\Sigma,\bQ)$ that simultaneously rescales the weights. 
A rational unbounded $\fsl_3$-lamination is said to be \emph{integral} if all the weights are integers. The subset of integral unbounded $\mathfrak{sl}_3$-laminations is denoted by $\cL^x(\Sigma,\bZ)$.

\begin{conv}\label{notation:division of honeycombs}
In view of the equivalence relation (4), we will occasionally use the following equivalent notations for honeycombs:
\begin{align*}
\begin{tikzpicture}[scale=1.2]
\draw[red,very thick] (-30:0.4) -- (90:0.4) -- (210:0.4) --cycle;
\foreach \i in {1,4}
{
\draw[red,very thick] ($(-30:0.4)!\i/5!(90:0.4)$) --++(30:0.5) coordinate(A\i);
\draw[red,very thick] ($(90:0.4)!\i/5!(210:0.4)$) --++(150:0.5) coordinate(B\i);
\draw[red,very thick] ($(210:0.4)!\i/5!(-30:0.4)$) --++(-90:0.5) coordinate(C\i);
}
\draw[red,very thick,dotted] (15:0.5) -- (45:0.5);
\draw[red,very thick,dotted] (135:0.5) -- (165:0.5);
\draw[red,very thick,dotted] (255:0.5) -- (285:0.5);
\draw[decorate,decoration={brace,amplitude=3pt,raise=4pt}] (A4) --node[midway,xshift=12pt,yshift=8pt]{$n$} (A1);
\draw[decorate,decoration={brace,amplitude=3pt,raise=4pt}] (B4) --node[midway,xshift=-12pt,yshift=8pt]{$n$} (B1);
\draw[decorate,decoration={brace,amplitude=3pt,raise=4pt}] (C4) --node[midway,below=7pt]{$n$} (C1);
\draw[dashed] (0,0) circle(0.76cm);
\node at (1.5,0) {\scalebox{1.2}{$\sim$}}; 
\begin{scope}[xshift=3cm]
\draw[red,very thick] (-30:0.4) -- (90:0.4) -- (210:0.4) --cycle;
\draw[red,very thick] ($(-30:0.4)!0.5!(90:0.4)$) --node[midway,above]{$n$} ++(30:0.5);
\draw[red,very thick] ($(90:0.4)!0.5!(210:0.4)$) --node[midway,above]{$n$} ++(150:0.5);
\draw[red,very thick] ($(210:0.4)!0.5!(-30:0.4)$) --node[midway,right]{$n$} ++(-90:0.5);
\draw[dashed] (0,0) circle(0.76cm);
\node at (1.5,0) {\scalebox{1.2}{$\sim$}}; 
\end{scope}
\begin{scope}[xshift=6cm]
\draw[red,very thick] (-30:0.4) -- (90:0.4) -- (210:0.4) --cycle;
\draw[red,very thick] ($(-30:0.4)!0.5!(90:0.4)$) --node[midway,above]{$n$} ++(30:0.5);
\draw[red,very thick] ($(90:0.4)!0.5!(210:0.4)$) --node[midway,above]{$n$} ++(150:0.5);
\draw[red,very thick] ($(210:0.4)!0.3!(-30:0.4)$) --node[midway,left,scale=0.9]{$n_1$} ++(-90:0.5);
\draw[red,very thick] ($(210:0.4)!0.7!(-30:0.4)$) --node[midway,right,scale=0.9]{$n_2$} ++(-90:0.5);
\draw[dashed] (0,0) circle(0.76cm);
\end{scope}
\end{tikzpicture}
\end{align*}
with $n_1+n_2=n$. We may also split an edge of weight $n$ with $k$ edges of weight $n_1,\dots,n_k$ with $n_1+\dots+n_k=n$.
\end{conv}

\begin{dfn}[Dynkin involution]\label{def:Dynkin_geometric}
The \emph{Dynkin involution} is the involutive automorphism
\begin{align*}
    \ast: \cL^x(\Sigma,\bQ) \to \cL^x(\Sigma,\bQ), \quad \hL \mapsto \hL^\ast,
\end{align*}
where $\hL^\ast$ is obtained from $\hL$ by reversing the orientation of every components of the underlying web, and keeping the signs at punctures intact. 
Since all the elementary moves (E1)--(E4) are equivariant under the orientation-reversion, this indeed defines an automorphism on $\cL^x(\Sigma,\bQ)$. 
\end{dfn}

\paragraph{\textbf{Bounded laminations and the ensemble map}}

\begin{dfn}[rational bounded $\fsl_3$-laminations]\label{def:bounded laminations}
A \emph{rational bounded $\fsl_3$-lamination} (or a \emph{rational $\fsl_3$-$\A$-lamination}) on $\Sigma$ is a bounded non-elliptic $\fsl_3$-web $W$ on $\Sigma$ equipped with a rational number (called the \emph{weight}) on each component such that the weight on a non-peripheral component is positive. It is considered modulo the equivalence relation generated by isotopies and the operations (2)--(4) in \cref{def:unbounded laminations}.
\end{dfn}
Let $\cL^a(\Sigma,\bQ)$ denote the space of rational bounded $\fsl_3$-laminations. We have a natural $\bQ_{>0}$-action on $\cL^a(\Sigma,\bQ)$ that simultaneously rescales the weights. A rational bounded $\fsl_3$-lamination is said to be \emph{integral} if all the weights are integers. The subset of integral bounded $\fsl_3$-laminations is denoted by $\cL^a(\Sigma,\bZ)$. For the relation to the works of Douglas--Sun \cite{DS20I,DS20II} and Kim \cite{Kim21}, see \cite[Remark 2.10]{IK22}.


By forgetting the peripheral components, we get the \emph{geometric ensemble map}
\begin{align}\label{eq:ensemble_unfrozen}
    p: \cL^a(\Sigma,\bQ) \to \cL^x(\Sigma,\bQ).
\end{align}
In the next subsection, we will describe the `kernel' and the `cokernel` of the geometric ensemble map $p$.

\subsection{Lamination exact sequence}\label{subsec:lam_exact_seq}

\subsubsection{
The tropicalized Casimir map on $\cL^x(\Sigma,\bQ)$}
For each puncture $p \in \bP$, we define a function 
\begin{align*}
    \theta_p: \cL^x(\Sigma,\bQ) \to \mathsf{P}^\vee_\bQ
\end{align*}
which we call the \emph{tropicalized Casimir map} associated with $p$, as follows. 
For a rational $\mathfrak{sl}_3$-lamination $\hL \in \cL^x(\Sigma,\bQ)$ with a support $W$, the coweight $\theta_p(\hL) \in \mathsf{P}^\vee_\bQ$ is defined to be the sum of the contributions from each end of the web $W$ incident to $p$, given as in \cref{fig:weight-contribution}. Then it is clearly invariant under the equivalence relation in \cref{def:unbounded laminations}. 

\begin{figure}[ht]
\begin{tikzpicture}[scale=0.8]
    \draw[dashed, fill=white] (0,0) circle [radius=1];
    \draw[red,very thick,->-] (0,1) -- (0,0);
    \filldraw[fill=white](0,0) circle(2.5pt) node[above left,red] {$+$};
    \node at (0,-0.5) {$\varpi^\vee_2$};
    
    \begin{scope}[xshift=3cm]
    \draw[dashed, fill=white] (0,0) circle [radius=1];
    \draw[red,very thick,->-] (0,1) -- (0,0);
    \filldraw[fill=white](0,0) circle(2.5pt) node[above left,red] {$-$};
    \node at (0,-0.5) {$-\varpi^\vee_1$};
    \end{scope}
    
    \begin{scope}[xshift=6cm]
    \draw[dashed, fill=white] (0,0) circle [radius=1];
    \draw[red,very thick,-<-] (0,1) -- (0,0);
    \filldraw[fill=white](0,0) circle(2.5pt) node[above left,red] {$+$};
    \node at (0,-0.5) {$\varpi^\vee_1$};
    \end{scope}
    
    \begin{scope}[xshift=9cm]
    \draw[dashed, fill=white] (0,0) circle [radius=1];
    \draw[red,very thick,-<-] (0,1) -- (0,0);
    \filldraw[fill=white](0,0) circle(2.5pt) node[above left,red] {$-$};
    \node at (0,-0.5) {$-\varpi^\vee_2$};
    \end{scope}
\end{tikzpicture}
    \caption{Contribution to $\theta_p(\hL)$ from each end incident to the puncture $p$.}
    \label{fig:weight-contribution}
\end{figure}

Let 
$\theta:=(\theta_p)_{p \in \bP}: \cL^x(\Sigma,\bQ) \to (\mathsf{P}^\vee_\bQ)^{\bP}$. 

It is also verified by inspection that
\begin{align}\label{eq:Dynkin-theta}
    \theta_p(\hL^\ast) = (\theta_p(\hL))^\ast
\end{align}
holds for all $p \in \bP$, where the Dynkin involution on the coweights is given by $(\varpi^\vee_s)^\ast:=\varpi^\vee_{3-s}$ for $s=1,2$.

\subsubsection{The 
peripheral action on $\cL^a(\Sigma,\bQ)$}
We have a natural action of $(\bQ^2)^{\bM}$ on $\cL^a(\Sigma,\bQ)$ that adds peripheral leaves to a given bounded $\fsl_3$-lamination. It turns out to be natural to interpret it as actions of the coroot lattice, as follows. 
For a marked point $m \in \bM$, a copy of $\mathsf{Q}^\vee_\bQ$  acts on $\cL^a(\Sigma,\bQ)$, where $\mu=v_1\alpha_1^\vee + v_2\alpha_2^\vee \in \mathsf{Q}^\vee_\bQ$ adds the peripheral components around $m$ of the following forms:
\begin{itemize}
    \item oriented counter-clockwisely with weight $2v_1-v_2$, and 
    \item oriented clockwisely with weight $-v_1+2v_2$.
\end{itemize}
This action is designed so that the fundamental coweight $\varpi^\vee_1=\frac{1}{3}(2\alpha_1^\vee+\alpha_2^\vee)$ (resp. $\varpi^\vee_2=\frac{1}{3}(\alpha_1^\vee+2\alpha_2^\vee)$) gives rise to a single peripheral component with weight $1$ and the counter-clockwise (resp. clockwise) orientation. 
These actions assigned to the marked points commute with each other.
In \cref{subsec:cluster_exact_seq}, we will see that this action coincides with the tropicalization of the $H_\A$-action on the cluster variety $\A_{\fsl_3,\Sigma}$.

\subsubsection{The lamination exact sequence}
The previously discussed structures are combined as follows.

\begin{prop}\label{prop:exact_seq_lamination}
We have $\theta^{-1}(0) = p(\cL^a(\Sigma,\bQ))$. Therefore we have the exact sequence
\begin{align*}
    0 \to (\mathsf{Q}^\vee_\bQ)^{\bM} \to \cL^a(\Sigma,\bQ) \xrightarrow{p} \cL^x(\Sigma,\bQ) \xrightarrow{\theta} (\mathsf{P}^\vee_\bQ)^{\bP} \to 0.
\end{align*}
Here the map $(\mathsf{Q}^\vee_\bQ)^{\bM} \to \cL^a(\Sigma,\bQ)$ is given by the action on the empty $\fsl_3$-lamination.
\end{prop}
Indeed, since a pair of ends with opposite coweights are not allowed by the admissible condition, $\theta^{-1}(0)$ consists of the bounded rational $\mathfrak{sl}_3$-laminations, namely the image of $p$. The exactness of the other parts are obvious.  

Let us denote the space of bounded $\fsl_3$-laminations without peripheral components by $\cL^u(\Sigma,\bQ):= \theta^{-1}(0) = p(\cL^a(\Sigma,\bQ))$. Then we have an obvious splitting
\begin{align}\label{eq:splitting}
    (p,(\mathsf{w}_m)_{m \in \bM}): \cL^a(\Sigma,\bQ) \xrightarrow{\sim} \cL^u(\Sigma,\bQ) \times (\mathsf{Q}^\vee_\bQ)^{\bM},
\end{align}
where $\mathsf{w}_m: \cL^a(\Sigma,\bQ) \to \mathsf{Q}_\bQ^\vee$ extracts the weights of peripheral components around $m \in \bM$. We will see that $\mathsf{w}_m$ is given by the tropicalized Goncharov--Shen potential in \cref{prop:Weyl_action_A} (cf.~\cite[Section 2]{DS20II}). 


\subsection{The spiralling diagram associated with a signed web}
In order to define the shear coordinates of a non-elliptic signed web $W$, we first deform it to the \emph{spiralling diagram} $\cW$. See \cite[Section 3]{IK22} for a detail. Roughly speaking, the spiralling diagram $\cW$ is obtained as follows.
\begin{enumerate}
    \item In a small disk neighborhood $D_p$ of each puncture $p \in \bP$, deform each end of $W$ incident to $p$ into an infinitely spiralling curve, according to their signs as shown in \cref{fig:spiral}. 
    \item The resulting (possibly infinite) self-intersections are replaced with H-shaped parts as shown in \cref{fig:spiral_example}.
\end{enumerate}

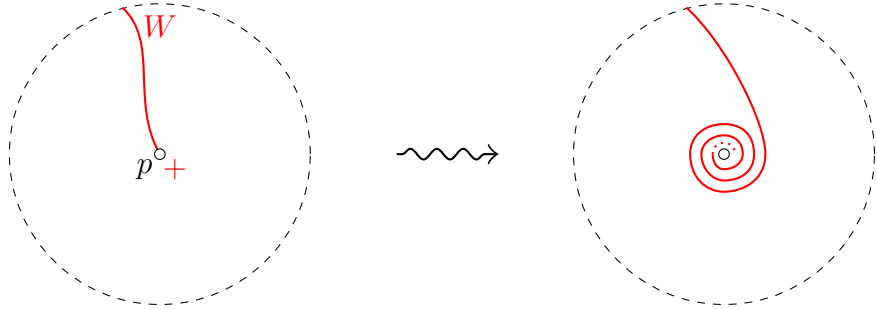
\begin{figure}[htbp]
    \centering
\begin{tikzpicture}
\draw[dashed] (-2.5,-1.5) circle(2cm);
\draw [red,thick](-3,0.45) .. controls (-2.5,0) and (-2.9,-0.8) .. (-2.5,-1.5);
\filldraw[fill=white] (-2.5,-1.5) circle(2pt);
\node[red] at (-2.5,0.2) {$W$};
\node[red] at (-2.3,-1.7) {$+$};
\node at (-2.7,-1.7) {$p$};
\draw (5,-1.5) circle(2pt);
\draw[dashed] (5,-1.5) circle(2cm);

\draw [red,thick](4.5,0.45) .. controls (5,0) and (5.55,-1.05) .. (5.55,-1.5) .. controls (5.55,-1.85) and (5.25,-2) .. (5,-2) .. controls (4.75,-2) and (4.55,-1.8) .. (4.55,-1.5) .. controls (4.55,-1.25) and (4.75,-1.1) .. (5,-1.1) .. controls (5.25,-1.1) and (5.4,-1.25) .. (5.4,-1.5) .. controls (5.4,-1.75) and (5.2,-1.85) .. (5,-1.85) .. controls (4.85,-1.85) and (4.7,-1.7) .. (4.7,-1.5) .. controls (4.7,-1.35) and (4.85,-1.25) .. (5,-1.25) .. controls (5.15,-1.25) and (5.25,-1.35) .. (5.25,-1.5) .. controls (5.25,-1.6) and (5.15,-1.7) .. (5,-1.7) .. controls (4.9,-1.7) and (4.85,-1.6) .. (4.85,-1.5);
\draw [red, thick, dotted](4.85,-1.5) .. controls (4.85,-1.3) and (5.15,-1.3) .. (5.15,-1.5);

\draw [thick,-{Classical TikZ Rightarrow[length=4pt]},decorate,decoration={snake,amplitude=2pt,pre length=2pt,post length=3pt}](0.65,-1.5) -- (2,-1.5);
\end{tikzpicture}
    \caption{Construction of a spiralling diagram. The negative sign similarly produces an end spiralling counter-clockwisely.}
    \label{fig:spiral}
\end{figure}

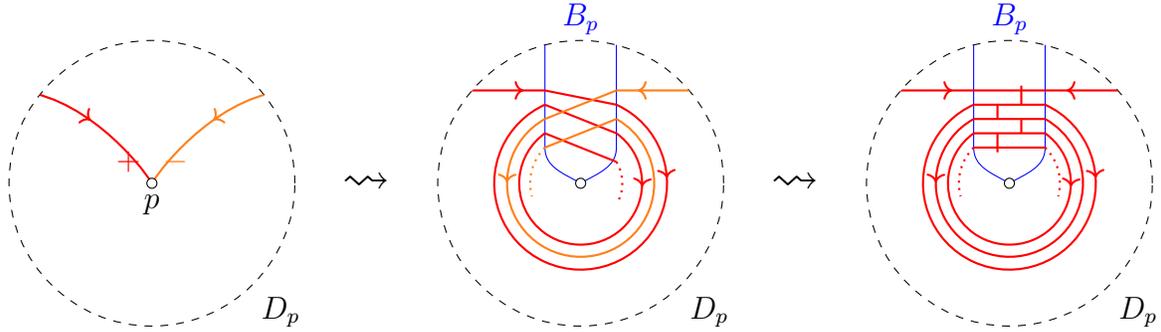
\begin{figure}[htbp]
\begin{tikzpicture}[scale=.95]
\begin{scope}
\draw[dashed] (0,0) circle(2cm);
\node at (1.8,-1.8) {$D_p$};
\clip(0,0) circle(2cm);
\draw[red,thick,->-] (-2,1.3) ..controls (-1,1.3) and (120:0.3).. (0,0) node[above left]{$+$};
\draw[myorange,thick,->-] (2,1.3) ..controls (1,1.3) and (60:0.3).. (0,0) node[above right]{$-$};
\draw[fill=white](0,0) circle(2pt) node[below]{$p$};
\end{scope}
\node[scale=1.5] at (3,0) {$\rightsquigarrow$};

\begin{scope}[xshift=6cm]
\draw[dashed] (0,0) circle(2cm);
\node at (1.8,-1.8) {$D_p$};
\node[blue] at (0,2.3) {$B_p$};
\clip(0,0) circle(2cm);
\draw[blue] (0,0) to[out=150,in=-90] (-0.5,0.5) -- (-0.5,2);
\draw[blue] (0,0) to[out=30,in=-90] (0.5,0.5) -- (0.5,2);
\draw[fill=white](0,0) circle(2pt);
\draw[red,thick,->-={0.8}{}] (-2,1.3) -- (-0.5,1.3);
\draw[red,thick] (-0.5,1.3) -- (0.5,1.1);
\draw[red,thick] (-0.5,1.1) -- (0.5,0.7);
\draw[red,thick] (-0.5,0.7) -- (0.5,0.3);
\draw[red,thick,->-={0.2}{}] (0.5,1.1) arc(65.56:-180-65.56:1.209);
\draw[red,thick,->-={0.2}{}] (0.5,0.7) arc(54.46:-180-54.46:0.86);
\draw[red,thick,dotted] (0.5,0.3) arc(30.96:-30:0.583);

\draw[myorange,thick,->-={0.8}{}] (2,1.3) -- (0.5,1.3);
\draw[myorange,thick] (0.5,1.3) -- (-0.5,0.9);
\draw[myorange,thick] (0.5,0.9) -- (-0.5,0.5);
\draw[myorange,thick,->-={0.2}{}] (-0.5,0.9) arc(-180-60.95:60.95:1.0296);
\draw[myorange,thick,dotted] (-0.5,0.5) arc(-180-45:-180+15:0.5*1.414);
\end{scope}
\node[scale=1.5] at (9,0) {$\rightsquigarrow$};

\begin{scope}[xshift=12cm]
\draw[dashed] (0,0) circle(2cm);
\node at (1.8,-1.8) {$D_p$};
\node[blue] at (0,2.3) {$B_p$};
\clip(0,0) circle(2cm);
\draw[blue] (0,0) to[out=150,in=-90] (-0.5,0.5) -- (-0.5,2);
\draw[blue] (0,0) to[out=30,in=-90] (0.5,0.5) -- (0.5,2);
\draw[fill=white](0,0) circle(2pt);
\draw[red,thick,->-={0.8}{}] (-2,1.3) -- (-0.5,1.3);
\draw[red,thick] (-0.5,1.3) -- (0.5,1.3);
\draw[red,thick] (-0.5,1.1) -- (0.5,1.1);
\draw[red,thick] (-0.5,0.9) -- (0.5,0.9);
\draw[red,thick] (-0.5,0.7) -- (0.5,0.7);
\draw[red,thick] (-0.5,0.5) -- (0.5,0.5);

\draw[red,thick] ($(-0.5,1.5)!2/3!(0.5,1.3)$) -- ($(-0.5,1.1)!2/3!(0.5,1.1)$);
\draw[red,thick] ($(-0.5,1.1)!1/3!(0.5,1.1)$) -- ($(-0.5,0.9)!1/3!(0.5,0.9)$);
\draw[red,thick] ($(-0.5,0.9)!2/3!(0.5,0.9)$) -- ($(-0.5,0.7)!2/3!(0.5,0.7)$);
\draw[red,thick] ($(-0.5,0.7)!1/3!(0.5,0.7)$) -- ($(-0.5,0.5)!1/3!(0.5,0.3)$);

\draw[red,thick,->-={0.2}{}] (0.5,1.1) arc(65.56:-180-65.56:1.209);
\draw[red,thick,->-={0.2}{}] (0.5,0.7) arc(54.46:-180-54.46:0.86);
\draw[red,thick,dotted] (0.5,0.5) arc(45:-15:0.5*1.414);

\draw[red,thick,->-={0.8}{}] (2,1.3) -- (0.5,1.3);
\draw[red,thick,->-={0.2}{}] (-0.5,0.9) arc(-180-60.95:60.95:1.0296);
\draw[red,thick,dotted] (-0.5,0.5) arc(-180-45:-180+15:0.5*1.414);
\end{scope}
\end{tikzpicture}
    \caption{Construction of a spiralling diagram. Replace intersections with H-webs in a periodic manner.}
    \label{fig:spiral_example}
\end{figure}
Given an ideal triangulation $\tri$ of $\Sigma$ without self-folded triangles, the associated \emph{split ideal triangulation $\widehat{\tri}$} is obtained from $\tri$ by replacing each edge $E$ with a biangle $B_E$. The following is an unbounded version of the notion of \emph{good position} of bounded webs \cite{FS20,DS20I}.

\begin{dfn}
The spiralling diagram $\cW$ is in a \emph{good position} with respect to a split triangulation $\widehat{\tri}$ if the intersection $\cW\cap B_E$ (resp. $\cW \cap T$) is an unbounded essential (resp. reduced essential) local web (\cite[Section 3.1]{IK22}) for each $E \in e(\tri)$ and $T \in t(\tri)$.  
\end{dfn}
Here,
\begin{itemize}
    \item An unbounded essential web on a biangle $B$ is the ladder-web $W(S)$ associated with an asymptotically periodic symmetric strand set $S$ as shown in the left of \cref{fig:essential_webs}. 
    \item An unbounded reduced essential web on a triangle $T$ is the disjoint union of at most one honeycomb and at most one semi-infinite asymptotically periodic collection of corner arcs around each corner. See right of \cref{fig:essential_webs}.
\end{itemize}
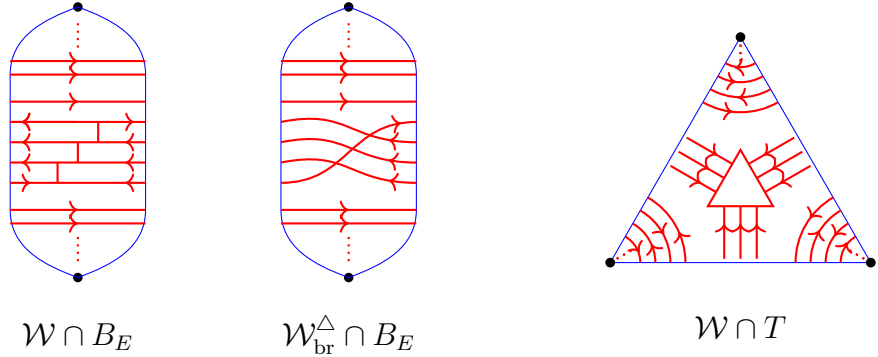
\begin{figure}[h]
    \centering
    \begin{tikzcd}[column sep=huge]
    \begin{tikzpicture}[scale=0.9]
    
    \begin{scope}[xshift=4cm,rotate=90]
    \fill(2,0) circle(2pt);
    \fill(-2,0) circle(2pt);
    \draw[blue](-2,0) to[out=70,in=180] (-1,1) -- (1,1) to[out=0,in=110] (2,0);
    \draw[blue](-2,0) to[out=-70,in=180] (-1,-1) -- (1,-1) to[out=0,in=-110] (2,0);
    \foreach \i in {-0.6,-0.3,0}
    \draw[red,thick,->-={0.2}{}] (\i,-1) ..controls (\i,0) and (\i+0.5,0).. (\i+0.3,1);
    \draw[red,thick,->-={0.9}{}] (-0.6,1) ..controls (-0.6,0) and (0.3,0).. (0.3,-1);
    \draw[red,thick,->-] (0.6,1) -- (0.6,-1);
    \foreach \i in {-1.0,-1.2,1.0,1.2} 
    \draw[red,thick,->-] (\i,1) -- (\i,-1);
    \draw[red,thick,dotted] (-1.4,0) -- (-1.8,0);
    \draw[red,thick,dotted] (1.4,0) -- (1.8,0);
    \node at (-3,0) {$\cW_{\mathrm{br}}^\tri \cap B_E$};
    \end{scope}
    
    \begin{scope}[xshift=0cm,rotate=90]
    \fill(2,0) circle(2pt);
    \fill(-2,0) circle(2pt);
    \draw[blue](-2,0) to[out=70,in=180] (-1,1) -- (1,1) to[out=0,in=110] (2,0);
    \draw[blue](-2,0) to[out=-70,in=180] (-1,-1) -- (1,-1) to[out=0,in=-110] (2,0);
    \foreach \i in {-0.6,-0.3,0}
    {
    \draw[red,thick,->-={0.3}{}] (\i,-1) -- (\i,0);
    \draw[red,thick,-<-={0.3}{}] (\i+0.3,1) -- (\i+0.3,0);
    \draw[red,thick] (\i,-\i-0.3) --++(0.3,0);
    }
    \draw[red,thick,->-={0.3}{}] (-0.6,1) -- (-0.6,0);
    \draw[red,thick,-<-={0.3}{}] (0.3,-1) -- (0.3,0);
    \draw[red,thick,-<-={0.3}{}] (0.3,1) -- (0.3,0);
    \draw[red,thick,->-] (0.6,1) -- (0.6,-1);
    \foreach \i in {-1.0,-1.2,1.0,1.2} 
    \draw[red,thick,->-] (\i,1) -- (\i,-1);
    \draw[red,thick,dotted] (-1.4,0) -- (-1.8,0);
    \draw[red,thick,dotted] (1.4,0) -- (1.8,0);
    \node at (-3,0) {$\cW \cap B_E$};
    \end{scope}
    \end{tikzpicture}
    &
    \begin{tikzpicture}[baseline=6mm]
    \draw[blue] (-30:2) node[fill, circle, inner sep=1.3pt, black] (A){} -- (90:2) node[fill, circle, inner sep=1.3pt, black] (B){} -- (210:2) node[fill, circle, inner sep=1.3pt, black] (C){} --cycle;
    \begin{scope}
    \clip (-30:2) -- (90:2) -- (210:2) --cycle;
    \draw[red,thick] (-30:0.5) -- (90:0.5) -- (210:0.5) --cycle;
    \foreach \i in {1,2,3}
    {
    \foreach \x in {0,120,240}
    \draw[red,thick,rotate=\x,-<-] ($(-30:0.5)!\i/4!(90:0.5)$) --++(30:0.7);
    }
    \foreach \i in {0.6,1}
    {
    \foreach \x in {0,120,240}
        {
        \draw(-30+\x:2)++(120+\x:\i) coordinate(a);
        \draw(-30+\x:2)++(120+\x:\i-0.2) coordinate(b);
        \draw[red,thick,->-={0.4}{}] (a) arc(120+\x:180+\x:\i);
        \draw[red,thick,-<-={0.6}{}] (b) arc(120+\x:180+\x:\i-0.2);
        \draw[red,thick,dotted] (-30+\x:1.9) -- (-30+\x:1.65);
        }
    }
    \end{scope}
    \node at (0,-2) {$\cW \cap T$};
    \end{tikzpicture}
    \end{tikzcd}
    \caption{Unbounded essential webs on a biangle (Left) and a triangle (Right). The braid representative of the former is shown in the Middle.}
    \label{fig:essential_webs}
\end{figure}
Again, see \cite[Section 3.3]{IK22} for a detail. 
Then, after some technical discussion, we can verify that any spiralling diagram arising from a non-elliptic signed web on $\Sigma$ can be isotoped into a good position with respect to $\widehat{\tri}$ by a finite sequence of certain elementary moves
\cite[Theorem 3.10]{IK22}. Moreover, such a good position is unique up to certain elementary moves. 

While the spiralling diagram itself is suited to discuss its good position, the following \emph{braid representation} will be useful to define the shear coordinates:

\begin{dfn}[Braid representation of a spiralling diagram]
Let $\cW$ be a spiralling diagram in a good position with respect to $\widehat{\tri}$. Then its \emph{braid representation} $\cW_{\mathrm{br}}^\tri$ is obtained from $\cW$ by replacing the unbounded essential web $\cW\cap B_E$ on each biangle $B_E$ with its braid representation. 
\end{dfn}
\subsection{Shear coordinates}\label{subsec:shear}
Let us recall the shear coordinates associated with an ideal triangulation $\tri$ of $\Sigma$ without self-folded triangles. Let $\widehat{\tri}$ be the associated split triangulation. 

Given a rational $\fsl_3$-lamination $\hL \in \cL^x(\Sigma,\bQ)$, represent it by an $\fsl_3$-web $W$ together with rational weights on its components and signs at the ends incident to punctures. Let $\cW$ be the associated spiralling diagram together with rational weights on the components, placed in good position with respect to $\widehat{\tri}$. Let $\cW_{\mathrm{br}}^\tri$ be its braid representation, together with well-assigned rational weights on its components. The shear coordinates of $\hL$ are going to be defined out of $\cW_{\mathrm{br}}^\tri$. 

For each $E \in e_{\interior}(\tri)$, let $Q_E$ be the unique quadrilateral containing $E$ as its diagonal, regarded as the union of two triangles $T_L,T_R$ and the biangle $B_E$. By \cite[Theorem 19]{FS20} (\cite[Proposition 22]{DS20I}), the restriction of $\cW_{\mathrm{br}}^\tri$ to each of $T_L$ and $T_R$ has at most one honeycomb web, which is represented by a triangular symbol as in \cref{notation:division of honeycombs}. We call any strand in the braid representative $\cW_{\mathrm{br}}^\tri \cap Q_E$ that is incident to the triangular symbol in $T_L$ (if exists) a \emph{$T_L$-strand}. Similarly, we define \emph{$T_R$-strands}. It is possible that an arc is both $T_L$- and $T_R$-strand, in which case it connects the two honeycombs. 
By removing the $T_L$- and $T_R$-strands, the remaining is a collection of (possibly intersecting) oriented curves, which we call the \emph{curve components}.


\begin{dfn}[$\fsl_3$-shear coordinates]
The \emph{($\fsl_3$-)shear coordinate system} 
\begin{align*}
    \sfx^\uf_\tri(\hL) =(\sfx_i^\tri(\hL))_{i \in I_\uf(\tri)} \in \bQ^{I_\uf(\tri)}
\end{align*}
is defined as follows. First, for each $E \in e_{\interior}(\tri)$, the coordinates assigned to the four vertices in the interior of $Q_E$ only depend on the restriction $\cW_{\mathrm{br}}^\tri \cap Q_E$. 
\begin{enumerate}
    \item Each curve component contributes to the edge coordinates according to the rule shown in \cref{fig:shear_curve}.
    \item The honeycomb on the triangle $T_L$ contributes to $\sfx^\uf_\tri(\hL)$ as in \cref{fig:shear_honeycomb}. Namely, the face coordinate counts the height of the honeycomb web, where a sink (resp. source) is counted positively (resp. negatively). 
    The edge coordinates count the contributions from $T_L$-strands, where we have $n_1$ left-turning ones, $n_2$ straight-going ones (which are also $T_R$-strands), and $n_3$ right-turning ones. 
    \item The honeycomb on the triangle $T_R$ and the $T_R$-strands contribute in the symmetric way with respect to the $\pi$ rotation of the figure. 
\end{enumerate}
Then the shear coordinates are defined to be the weighted sums of these contributions. 
\end{dfn}

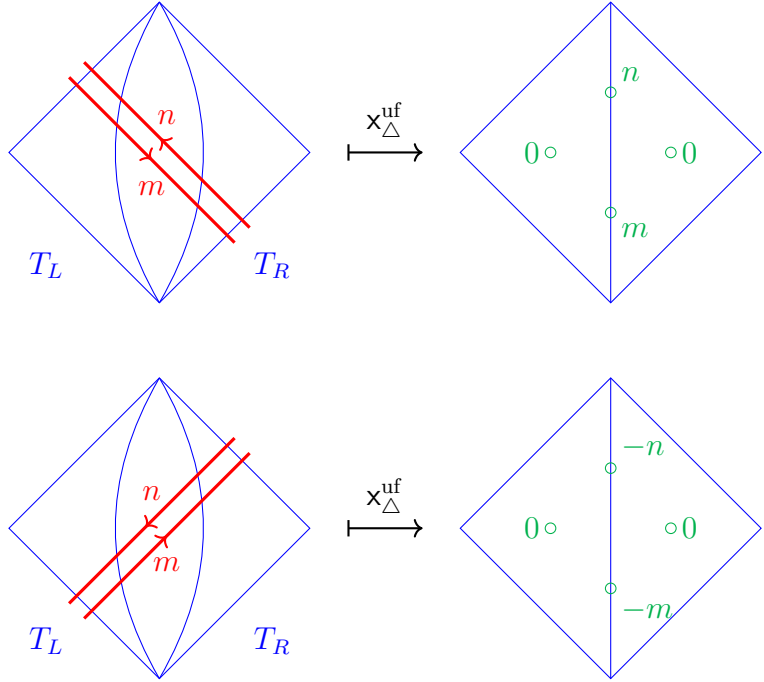
\begin{figure}[tbp]
    \centering
\begin{tikzpicture}
\draw[blue] (2,0) -- (0,2) -- (-2,0) -- (0,-2) --cycle;
\draw[blue] (0,-2) to[bend left=30pt] (0,2);
\draw[blue] (0,-2) to[bend right=30pt] (0,2);
\node[blue] at (-1.5,-1.5) {$T_L$};
\node[blue] at (1.5,-1.5) {$T_R$};
\draw[very thick,red,->-] (-1.2,1) --node[midway,below=0.3em]{$m$} (1,-1.2);
\draw[very thick,red,-<-] (-1,1.2) --node[midway,above=0.3em]{$n$} (1.2,-1);
\draw[thick,|->] (2.5,0) --node[midway,above]{$\sfx^\uf_\tri$} (3.5,0); 
\begin{scope}[xshift=6cm]
\draw[blue] (2,0) -- (0,2) -- (-2,0) -- (0,-2) --cycle;
\draw[blue] (0,-2) to (0,2);
{\color{mygreen}
\draw (0.8,0) circle(2pt) node[right]{$0$};
\draw (-0.8,0) circle(2pt) node[left]{$0$};
\draw (0,0.8) circle(2pt) node[above right]{$n$};
\draw(0,-0.8) circle(2pt)node[below right]{$m$};
}
\end{scope}

\begin{scope}[yshift=-5cm]
\draw[blue] (2,0) -- (0,2) -- (-2,0) -- (0,-2) --cycle;
\draw[blue] (0,-2) to[bend left=30pt] (0,2);
\draw[blue] (0,-2) to[bend right=30pt] (0,2);
\node[blue] at (-1.5,-1.5) {$T_L$};
\node[blue] at (1.5,-1.5) {$T_R$};
\draw[very thick,red,-<-] (-1.2,-1) --node[midway,above=0.3em]{$n$} (1,1.2);
\draw[very thick,red,->-] (-1,-1.2) --node[midway,below=0.3em]{$m$} (1.2,1);
\draw[thick,|->] (2.5,0) --node[midway,above]{$\sfx^\uf_\tri$} (3.5,0); 
{\begin{scope}[xshift=6cm]
\draw[blue] (2,0) -- (0,2) -- (-2,0) -- (0,-2) --cycle;
\draw[blue] (0,-2) to (0,2);
{\color{mygreen}
\draw (0.8,0) circle(2pt) node[right]{$0$};
\draw (-0.8,0) circle(2pt) node[left]{$0$};
\draw (0,0.8) circle(2pt) node[above right]{$-n$};
\draw(0,-0.8) circle(2pt)node[below right]{$-m$};
}
\end{scope}}
\end{scope}
\end{tikzpicture}
    \caption{Contributions from curve components.}
    \label{fig:shear_curve}
\end{figure}

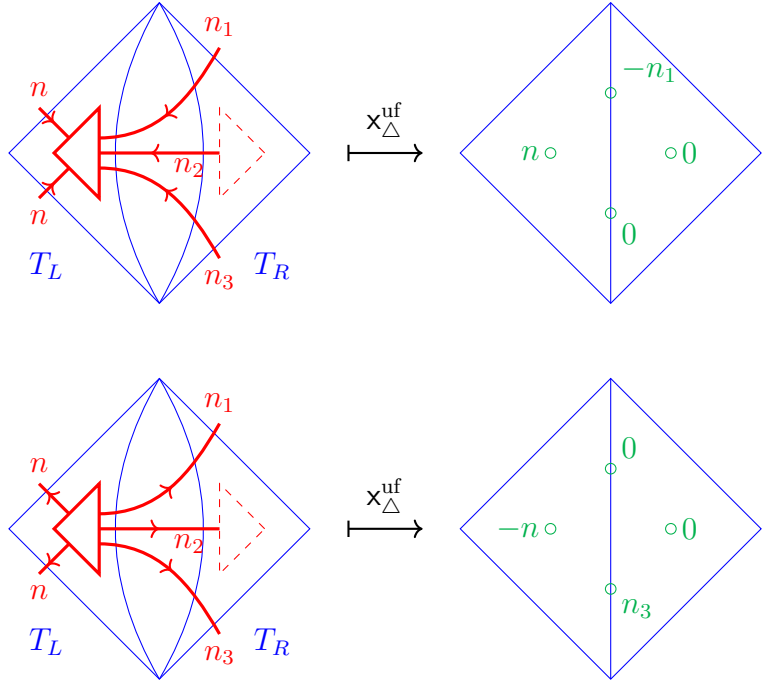
\begin{figure}[htbp]
    \centering
\begin{tikzpicture}
\draw[blue] (2,0) -- (0,2) -- (-2,0) -- (0,-2) --cycle;
\draw[blue] (0,-2) to[bend left=30pt] (0,2);
\draw[blue] (0,-2) to[bend right=30pt] (0,2);
\node[blue] at (-1.5,-1.5) {$T_L$};
\node[blue] at (1.5,-1.5) {$T_R$};
\draw[red,very thick] (-1.4,0) --++(0.6,0.6) --++(0,-1.2) --cycle;
\draw[red,very thick,->-] (-1.6,0.6) node[above]{$n$} -- (-1.2,0.2);
\draw[red,very thick,->-] (-1.6,-0.6) node[below]{$n$} -- (-1.2,-0.2);
\draw[red,dashed] (1.4,0) --++(-0.6,0.6) --++(0,-1.2) --cycle;
\draw[red,very thick,-<-] (-0.8,0) -- (0.8,0);
\node[red] at (0.4,-0.2) {$n_2$};
\draw[red,very thick,-<-] (-0.8,0.2) to[out=0, in=-120] (0.8,1.4) node[above]{$n_1$};
\draw[red,very thick,-<-] (-0.8,-0.2) to[out=0, in=120] (0.8,-1.4) node[below]{$n_3$};
\draw[thick,|->] (2.5,0) --node[midway,above]{$\sfx^\uf_\tri$} (3.5,0); 
\begin{scope}[xshift=6cm]
\draw[blue] (2,0) -- (0,2) -- (-2,0) -- (0,-2) --cycle;
\draw[blue] (0,-2) to (0,2);
{\color{mygreen}
\draw (0.8,0) circle(2pt) node[right]{$0$};
\draw (-0.8,0) circle(2pt) node[left]{$n$};
\draw (0,0.8) circle(2pt) node[above right]{$-n_1$};
\draw(0,-0.8) circle(2pt)node[below right]{$0$};
}
\end{scope}

\begin{scope}[yshift=-5cm]
\draw[blue] (2,0) -- (0,2) -- (-2,0) -- (0,-2) --cycle;
\draw[blue] (0,-2) to[bend left=30pt] (0,2);
\draw[blue] (0,-2) to[bend right=30pt] (0,2);
\node[blue] at (-1.5,-1.5) {$T_L$};
\node[blue] at (1.5,-1.5) {$T_R$};
\draw[red,very thick] (-1.4,0) --++(0.6,0.6) --++(0,-1.2) --cycle;
\draw[red,very thick,-<-] (-1.6,0.6) node[above]{$n$} -- (-1.2,0.2);
\draw[red,very thick,-<-] (-1.6,-0.6) node[below]{$n$} -- (-1.2,-0.2);
\draw[red,dashed] (1.4,0) --++(-0.6,0.6) --++(0,-1.2) --cycle;
\draw[red,very thick,->-] (-0.8,0) -- (0.8,0);
\node[red] at (0.4,-0.2) {$n_2$};
\draw[red,very thick,->-] (-0.8,0.2) to[out=0, in=-120] (0.8,1.4) node[above]{$n_1$};
\draw[red,very thick,->-] (-0.8,-0.2) to[out=0, in=120] (0.8,-1.4) node[below]{$n_3$};
\draw[thick,|->] (2.5,0) --node[midway,above]{$\sfx^\uf_\tri$} (3.5,0); 
{\begin{scope}[xshift=6cm]
\draw[blue] (2,0) -- (0,2) -- (-2,0) -- (0,-2) --cycle;
\draw[blue] (0,-2) to (0,2);
{\color{mygreen}
\draw (0.8,0) circle(2pt) node[right]{$0$};
\draw (-0.8,0) circle(2pt) node[left]{$-n$};
\draw (0,0.8) circle(2pt) node[above right]{$0$};
\draw(0,-0.8) circle(2pt)node[below right]{$n_3$};
}
\end{scope}}
\end{scope}
\end{tikzpicture}
    \caption{Contributions from the honeycomb of height $n=n_1+n_2+n_3$ on the triangle $T_L$. Observe that the $n_2$ straight-going $T_L$-strands do not contribute.}
    \label{fig:shear_honeycomb}
\end{figure}

\begin{rem}\label{rem:decomposition_honeycomb}
 The shear coordinates of the first honeycomb component shown in \cref{fig:shear_honeycomb} is the same as the sum of shear coordinates of the three honeycomb components shown in \cref{fig:shear_honeycomb_decomposition}. 
\end{rem}

\begin{figure}[htbp]
\centering
\begin{tikzpicture}[scale=.9]
\draw[blue] (2,0) -- (0,2) -- (-2,0) -- (0,-2) --cycle;
\draw[blue] (0,-2) to[bend left=30pt] (0,2);
\draw[blue] (0,-2) to[bend right=30pt] (0,2);
\draw[red,very thick] (-1.4,0) --++(0.6,0.6) --++(0,-1.2) --cycle;
\draw[red,very thick,->-] (-1.6,0.6) node[above]{$n_1$} -- (-1.2,0.2);
\draw[red,very thick,->-] (-1.6,-0.6) node[below]{$n_1$} -- (-1.2,-0.2);
\draw[red,very thick,-<-] (-0.8,0) to[out=0, in=-120] (0.8,1.4) node[above]{$n_1$};

{\begin{scope}[xshift=5cm]
\draw[blue] (2,0) -- (0,2) -- (-2,0) -- (0,-2) --cycle;
\draw[blue] (0,-2) to[bend left=30pt] (0,2);
\draw[blue] (0,-2) to[bend right=30pt] (0,2);
\draw[red,very thick] (-1.4,0) --++(0.6,0.6) --++(0,-1.2) --cycle;
\draw[red,very thick,->-] (-1.6,0.6) node[above]{$n_2$} -- (-1.2,0.2);
\draw[red,very thick,->-] (-1.6,-0.6) node[below]{$n_2$} -- (-1.2,-0.2);
\draw[red,dashed] (1.4,0) --++(-0.6,0.6) --++(0,-1.2) --cycle;
\draw[red,very thick,-<-] (-0.8,0) -- (0.8,0) node[below left=0.2em]{$n_2$};
\end{scope}}

{\begin{scope}[xshift=10cm]
\draw[blue] (2,0) -- (0,2) -- (-2,0) -- (0,-2) --cycle;
\draw[blue] (0,-2) to[bend left=30pt] (0,2);
\draw[blue] (0,-2) to[bend right=30pt] (0,2);
\draw[red,very thick] (-1.4,0) --++(0.6,0.6) --++(0,-1.2) --cycle;
\draw[red,very thick,->-] (-1.6,0.6) node[above]{$n_3$} -- (-1.2,0.2);
\draw[red,very thick,->-] (-1.6,-0.6) node[below]{$n_3$} -- (-1.2,-0.2);
\draw[red,very thick,-<-] (-0.8,0) to[out=0, in=120] (0.8,-1.4) node[below]{$n_3$};
\end{scope}}
\end{tikzpicture}
    \caption{Basic honeycomb components}
    \label{fig:shear_honeycomb_decomposition}
\end{figure}
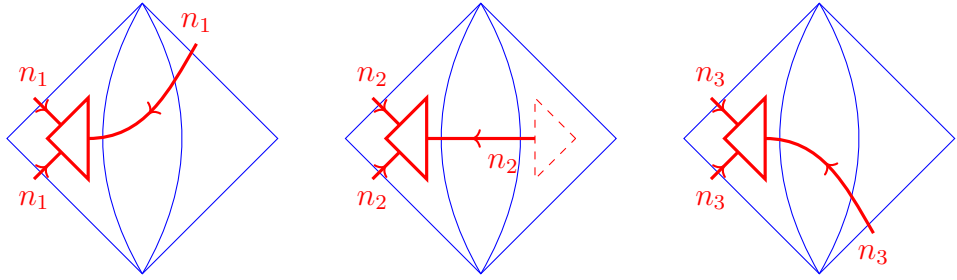

\begin{thm}[{\cite[Theorem 1]{IK22}}]\label{thm:shear_coordinate}
The shear coordinate system $\sfx^\uf_\tri(\hL) \in \bQ^{I_\uf(\tri)}$ gives a well-defined bijection
\begin{align*}
    \sfx^\uf_\tri: \cL^x(\Sigma,\bQ) \to \bQ^{I_\uf(\tri)}.
\end{align*}
Moreover, for any another ideal triangulation $\tri'$ of $\Sigma$, the coordinate transformation $\sfx_{\tri'}\circ \sfx_\tri^{-1}$ is a composite of tropical cluster $\X$-transformations. 
\end{thm}
As a consequence, the shear coordinates combine to give an $MC(\Sigma)$-equivariant bijection 
\begin{align}\label{eq:shear_combined_X}
    \sfx^\uf_\bullet: \cL^x(\Sigma,\bQ) \xrightarrow{\sim} \X_{\fsl_3,\Sigma}^\uf(\bQ^T).
\end{align}
In other words, the space $\cL^x(\Sigma,\bQ)$ can be viewed as a tropical analogue of the moduli space $\X_{PGL_3,\Sigma}$ of framed $PGL_3$-local systems \cite{FG03} with respect to its positive structure.

\begin{conv}\label{conv:shear_coordinates}
We will write $\sfx^\tri_T:=\sfx^\tri_{i(T)}$ for a triangle $T$ of $\tri$, and $\sfx^\tri_{E,s}:=\sfx^\tri_{i^s(E)}$ for an oriented edge $E$ of $\tri$ and $s=1,2$. Here recall the notations in \cref{subsec:notation_marked_surface}.
\end{conv}

\paragraph{\textbf{Dynkin involution}}
At the rest of this section, we recall the equivariance of the shear coordinates under the Dynkin involution (\cref{def:Dynkin_geometric}). 
The \emph{cluster action} $\ast_\tri$ (see the last paragraph of \cref{sec:appendix}) of the Dynkin involution in the cluster chart associated to $\tri$ is given by some mutation sequence, and it induces the tropical cluster $\X$-transformation
\begin{align*}
    \ast_\tri^x: 
    \sfx_T &\mapsto -\sfx_T, & \mbox{for $T \in t(\tri)$}, \\
    \sfx_{E,1} &\mapsto \sfx_{E,2}+[\sfx_{T_L}]_+ -[-\sfx_{T_R}]_+, &\\
    \sfx_{E,2} &\mapsto \sfx_{E,1}+[\sfx_{T_R}]_+ -[-\sfx_{T_L}]_+ & \mbox{for $E \in e(\tri)$}.
\end{align*}

\begin{prop}[{\cite[Proposition 4.13]{IK22}}]\label{prop:Dynkin-cluster}
We have the commutative diagram
\begin{equation*}
    \begin{tikzcd}
    \cL^x(\Sigma,\bQ) \ar[r,"\sfx_\tri"] \ar[d,"\ast"'] & \bQ^{I_\uf(\tri)}  \ar[d,"\ast_\tri"] \\
    \cL^x(\Sigma,\bQ) \ar[r,"\sfx_\tri"'] & \bQ^{I_\uf(\tri)}. 
    \end{tikzcd}
\end{equation*}
In particular, the orientation-reversing action of the Dynkin involution coincides with the cluster action.
\end{prop}

%% file: 3_exact_sequence.tex
\section{Structure of unbounded \texorpdfstring{$\fsl_3$}{sl(3)}-laminations around punctures}\label{sec:structure}

\subsection{Classification of unbounded \texorpdfstring{$\fsl_3$}{sl(3)}-laminations around punctures}\label{sec:ensemble}
In this section, we focus on the rational unbounded $\fsl_3$-laminations around a puncture as a preparation for the study of the cluster exact sequence and the Weyl group actions at punctures in the subsequent sections. 

Let us first clarify the elementary pieces that we are going to classify. 
Fix an ideal triangulation $\tri$ of a marked surface $\Sigma$. 
Given a non-elliptic signed web $W$ on $\Sigma$, let $\cW$ be the associated spiralling diagram in a good position with respect to $\widehat{\tri}$, and $\cW^\tri_{\mathrm{br}}$ its braid representation. Recall that each of the shear coordinates of $W$ is defined to be a sum of contributions from the components of $\cW^\tri_{\mathrm{br}}$ (\cref{fig:shear_curve,fig:shear_honeycomb}). 

In view of the rules in \cref{fig:shear_honeycomb}, we can further decompose the contributions from honeycomb components as follows. For each triangle $T \in t(\tri)$, split a honeycomb of height $n$ into a union of $n$ honeycombs of height one. See \cref{fig:braid_decomp} for an example. Then $\cW^\tri_{\mathrm{br}}$ is decomposed into a union $\cW^\tri_{\mathrm{br}}=\bigcup_{\alpha} \cW_\alpha$ of webs only with honeycombs of height one. 
Let us call each web $\cW_\alpha$ appearing in this decomposition an \emph{elementary braid} with respect to $\tri$. Let $W_\alpha$ denote the signed web corresponding to $\cW_\alpha$.

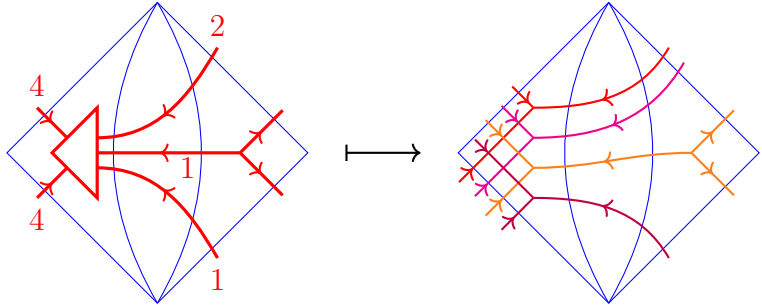
\begin{figure}[ht]
\begin{tikzpicture}
\draw[blue] (2,0) -- (0,2) -- (-2,0) -- (0,-2) --cycle;
\draw[blue] (0,-2) to[bend left=30pt] (0,2);
\draw[blue] (0,-2) to[bend right=30pt] (0,2);
\draw[red,very thick] (-1.4,0) --++(0.6,0.6) --++(0,-1.2) --cycle;
\draw[red,very thick,->-] (-1.6,0.6) node[above]{$4$} -- (-1.2,0.2);
\draw[red,very thick,->-] (-1.6,-0.6) node[below]{$4$} -- (-1.2,-0.2);
\draw[red,very thick,-<-] (-0.8,0) -- (1.1,0);
\draw[red,very thick,->-] (1.1,0) --++(45:0.8);
\draw[red,very thick,->-] (1.1,0) --++(-45:0.8);
\node[red] at (0.4,-0.2) {$1$};
\draw[red,very thick,-<-] (-0.8,0.2) to[out=0, in=-120] (0.8,1.4) node[above]{$2$};
\draw[red,very thick,-<-] (-0.8,-0.2) to[out=0, in=120] (0.8,-1.4) node[below]{$1$};
\draw[thick,|->] (2.5,0) -- (3.5,0); 
\begin{scope}[xshift=6cm]
\draw[blue] (2,0) -- (0,2) -- (-2,0) -- (0,-2) --cycle;
\draw[blue] (0,-2) to[bend left=30pt] (0,2);
\draw[blue] (0,-2) to[bend right=30pt] (0,2);
\draw[red,thick,-<-] (-1.0,0.6) --++(135:0.4);
\draw[red,thick,-<-={0.9}{}] (-1.0,0.6) --++(-135:1.4);
\draw[red,thick,-<-] (-1.0,0.6) to[out=0,in=-120] (0.8,1.4);
\draw[magenta,thick,-<-={0.7}{}] (-1.0,0.2) --++(135:0.6);
\draw[magenta,thick,-<-={0.9}{}] (-1.0,0.2) --++(-135:1.1);
\draw[magenta,thick,-<-] (-1.0,0.2) to[out=0,in=-120] (1.0,1.2);
\draw[myorange,thick,-<-={0.8}{}] (-1.0,-0.2) --++(135:0.9);
\draw[myorange,thick,-<-={0.8}{}] (-1.0,-0.2) --++(-135:0.9);
\draw[myorange,thick,-<-] (-1.0,-0.2) to[out=0,in=180] (1.1,0);
\draw[purple,thick,-<-={0.9}{}] (-1.0,-0.6) --++(135:1.1);
\draw[purple,thick,-<-={0.8}{}] (-1.0,-0.6) --++(-135:0.6);
\draw[purple,thick,-<-] (-1.0,-0.6) to[out=0,in=120] (0.8,-1.4);
\draw[myorange,thick,->-] (1.1,0) --++(45:0.8);
\draw[myorange,thick,->-] (1.1,0) --++(-45:0.8);
\end{scope}
\end{tikzpicture}
    \caption{Decomposition of a braid component into elementary braids.}
    \label{fig:braid_decomp}
\end{figure}

\begin{lem}\label{lem:decomposition_positive}
The shear coordinates of $W$ can be computed as sums of contributions from elementary braids:
\begin{align*}
    \sfx_i^\tri(W) = \sum_\alpha \sfx_i^\tri(W_\alpha)
\end{align*}
for $i \in I_\uf(\tri)$. 
Moreover, the signs of shear coordinates are coherent: $\sfx_i^\tri(W_\alpha)\cdot \sfx_i^\tri(W_{\alpha'}) \geq 0$ for $i \in I(\tri)$ and $\alpha \neq \alpha'$.
\end{lem}

\begin{proof}
The first assertion follows from \cref{rem:decomposition_honeycomb}. The signs of face coordinates $\sfx_T^\tri(W_\alpha)$, $T \in t(\tri)$ are coherent since the original web $W$ has a sink (resp. source) honeycomb if and only if each $W_\alpha$ has a sink (resp. source) honeycomb. 

For the edge coordinates, let us focus on the quadrilateral shown in \cref{fig:shear_curve,fig:shear_honeycomb}. If there were two elementary braids which contribute positively and negatively to the same edge coordinate, then they would have an intersection inside the bigon $B_E$ with the same orientation (in the sense `left-to-right' or `right-to-left'). It contradicts to the construction of the braid representative (see also the left picture in \cref{fig:essential_webs}). Hence we have $\sfx_i^\tri(W_\alpha)\cdot \sfx_i^\tri(W_{\alpha'}) \geq 0$ for all $i \in I_\uf(\tri)$.
\end{proof}

Let $D^\ast$ be a once-punctured disk with two special points, together with the ideal triangulation shown in \cref{fig:Casimir_check}. 

\begin{prop}[Classification of elementary braids on $D^\ast$]\label{prop:classification}
All possible elementary braids on $D^\ast$ are the eleven ones shown in \cref{fig:Casimir_check}, up to orientation-reversion and the $\pi$-rotation.
\end{prop}

\begin{proof}
By cutting $D^\ast$ along one of the interior edges, say $E$, we get a quadrilateral. Then by the discussion in the beginning of \cref{subsec:shear}, $\cW_{\mathrm{br}} \setminus E$ has at most two honeycomb components, and the other parts are curve components. 
For elementary braids, these honeycomb components have height $1$. 
Then we consider all possible patterns to connect them across the edge $E$. Those arising from the curve components are listed in the top row of \cref{fig:Casimir_check}. 
The second and third rows exhaust all the patterns with one and two honeycomb components, respectively.
\end{proof}

\begin{rem}
Thanks to the sign coherence statement in \cref{lem:decomposition_positive}, we can reduce the computations of a piecewise linear map (such as a cluster transformation) to that for the elementary braids, whenever its bending loci are of the form $\sfx_i^\tri=0$. It will be useful, for example in the proof of \cref{thm:Weyl_action}.
\end{rem}

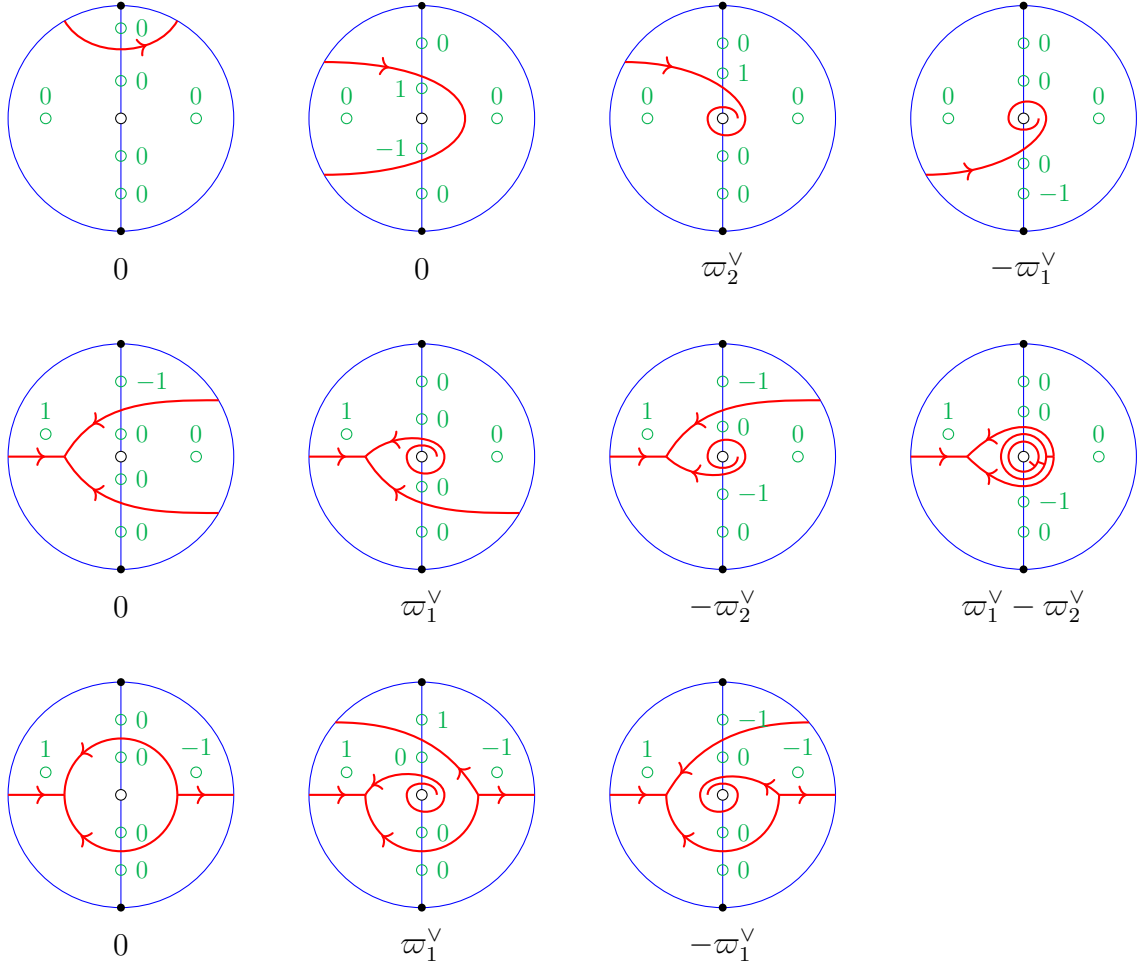
\begin{figure}[ht]
\begin{tikzpicture}
\draw[blue] (0,0) circle(1.5cm);
\draw[blue] (0,1.5) -- (0,-1.5);
\fill(0,1.5) circle(1.5pt);
\fill(0,-1.5) circle(1.5pt);
\filldraw[fill=white](0,0) circle(2pt);
\draw[red,thick,->-={0.7}{}] (120:1.5) to[out=-60,in=240] (60:1.5);
\node at (0,-2) {$0$};
{\color{mygreen}
\draw (0,0.5) circle(2pt) node[right=0.2em,scale=0.8]{$0$};
\draw (0,1.2) circle(2pt) node[right=0.2em,scale=0.8]{$0$};
\draw (0,-0.5) circle(2pt) node[right=0.2em,scale=0.8]{$0$};
\draw (0,-1) circle(2pt) node[right=0.2em,scale=0.8]{$0$};
\draw (1,0) circle(2pt) node[above=0.2em,scale=0.8]{$0$};
\draw (-1,0) circle(2pt) node[above=0.2em,scale=0.8]{$0$};
}
{\begin{scope}[xshift=4cm]
\draw[blue] (0,0) circle(1.5cm);
\draw[blue] (0,1.5) -- (0,-1.5);
\node[fill,circle,inner sep=1pt] at (0,1.5) {};
\node[fill,circle,inner sep=1pt] at (0,-1.5) {};
\filldraw[fill=white](0,0) circle(2pt);
\draw[red,thick,->-={0.2}{}] (150:1.5) ..controls ++(2.5,0) and ($(210:1.5)+(2.5,0)$).. (210:1.5);
\node at (0,-2) {$0$};
{\color{mygreen}
\draw (0,0.4) circle(2pt) node[left=0.2em,scale=0.8]{$1$};
\draw (0,1) circle(2pt) node[right=0.2em,scale=0.8]{$0$};
\draw (0,-0.4) circle(2pt) node[left=0.2em,scale=0.8]{$-1$};
\draw (0,-1) circle(2pt) node[right=0.2em,scale=0.8]{$0$};
\draw (1,0) circle(2pt) node[above=0.2em,scale=0.8]{$0$};
\draw (-1,0) circle(2pt) node[above=0.2em,scale=0.8]{$0$};
}
\end{scope}}
{\begin{scope}[xshift=8cm]
\draw[blue] (0,0) circle(1.5cm);
\draw[blue] (0,1.5) -- (0,-1.5);
\fill(0,1.5) circle(1.5pt);
\fill(0,-1.5) circle(1.5pt);
\filldraw[fill=white](0,0) circle(2pt);
\draw[red,thick,->-={0.2}{}] (150:1.5) ..controls ++(1,0) and (0.3,0.3).. (0.3,0)
..controls (0.3,-0.3) and (-0.2,-0.3).. (-0.2,0)
..controls (-0.2,0.2) and (0.2,0.2).. (0.2,0);
\node at (0,-2) {$\varpi^\vee_2$};
{\color{mygreen}
\draw (0,0.6) circle(2pt) node[right=0.2em,scale=0.8]{$1$};
\draw (0,1) circle(2pt) node[right=0.2em,scale=0.8]{$0$};
\draw (0,-0.5) circle(2pt) node[right=0.2em,scale=0.8]{$0$};
\draw (0,-1) circle(2pt) node[right=0.2em,scale=0.8]{$0$};
\draw (1,0) circle(2pt) node[above=0.2em,scale=0.8]{$0$};
\draw (-1,0) circle(2pt) node[above=0.2em,scale=0.8]{$0$};
}
\end{scope}}
{\begin{scope}[xshift=12cm]
\draw[blue] (0,0) circle(1.5cm);
\draw[blue] (0,1.5) -- (0,-1.5);
\fill(0,1.5) circle(1.5pt);
\fill(0,-1.5) circle(1.5pt);
\filldraw[fill=white](0,0) circle(2pt);
\draw[red,thick,->-={0.2}{}] (210:1.5) ..controls ++(1,0) and (0.3,-0.3).. (0.3,0)
..controls (0.3,0.3) and (-0.2,0.3).. (-0.2,0)
..controls (-0.2,-0.2) and (0.2,-0.2).. (0.2,0);
\node at (0,-2) {$-\varpi^\vee_1$};
{\color{mygreen}
\draw (0,0.5) circle(2pt) node[right=0.2em,scale=0.8]{$0$};
\draw (0,1) circle(2pt) node[right=0.2em,scale=0.8]{$0$};
\draw (0,-0.6) circle(2pt) node[right=0.2em,scale=0.8]{$0$};
\draw (0,-1) circle(2pt) node[right=0.2em,scale=0.8]{$-1$};
\draw (1,0) circle(2pt) node[above=0.2em,scale=0.8]{$0$};
\draw (-1,0) circle(2pt) node[above=0.2em,scale=0.8]{$0$};
}
\end{scope}}

{\begin{scope}[yshift=-4.5cm]
\draw[blue] (0,0) circle(1.5cm);
\draw[blue] (0,1.5) -- (0,-1.5);
\fill(0,1.5) circle(1.5pt);
\fill(0,-1.5) circle(1.5pt);
\filldraw[fill=white](0,0) circle(2pt);
\draw(-0.75,0) coordinate(L);
\draw[red,thick,-<-] (L) --++(-0.75,0);
\draw[red,thick,-<-={0.3}{}] (L) to[out=60,in=180] (30:1.5);
\draw[red,thick,-<-={0.3}{}] (L) to[out=-60,in=180] (-30:1.5);
\node at (0,-2) {$0$};
{\color{mygreen}
\draw (0,0.3) circle(2pt) node[right=0.2em,scale=0.8]{$0$};
\draw (0,1) circle(2pt) node[right=0.2em,scale=0.8]{$-1$};
\draw (0,-0.3) circle(2pt) node[right=0.2em,scale=0.8]{$0$};
\draw (0,-1) circle(2pt) node[right=0.2em,scale=0.8]{$0$};
\draw (1,0) circle(2pt) node[above=0.2em,scale=0.8]{$0$};
\draw (-1,0.3) circle(2pt) node[above=0.2em,scale=0.8]{$1$};
}
\end{scope}}
{\begin{scope}[xshift=4cm,yshift=-4.5cm]
\draw[blue] (0,0) circle(1.5cm);
\draw[blue] (0,1.5) -- (0,-1.5);
\fill(0,1.5) circle(1.5pt);
\fill(0,-1.5) circle(1.5pt);
\filldraw[fill=white](0,0) circle(2pt);
\draw(-0.75,0) coordinate(L);
\draw[red,thick,-<-] (L) --++(-0.75,0);
\draw[red,thick,-<-={0.3}{}] (L) to[out=-60,in=180] (-30:1.5);
\draw[red,thick,-<-={0.2}{}] (L) ..controls ++(45:0.5) and (0.3,0.3).. (0.3,0)
..controls (0.3,-0.3) and (-0.2,-0.3).. (-0.2,0)
..controls (-0.2,0.2) and (0.2,0.2).. (0.2,0);
\node at (0,-2) {$\varpi^\vee_1$};
{\color{mygreen}
\draw (0,0.5) circle(2pt) node[right=0.2em,scale=0.8]{$0$};
\draw (0,1) circle(2pt) node[right=0.2em,scale=0.8]{$0$};
\draw (0,-0.4) circle(2pt) node[right=0.2em,scale=0.8]{$0$};
\draw (0,-1) circle(2pt) node[right=0.2em,scale=0.8]{$0$};
\draw (1,0) circle(2pt) node[above=0.2em,scale=0.8]{$0$};
\draw (-1,0.3) circle(2pt) node[above=0.2em,scale=0.8]{$1$};
}
\end{scope}}
{\begin{scope}[xshift=8cm,yshift=-4.5cm]
\draw[blue] (0,0) circle(1.5cm);
\draw[blue] (0,1.5) -- (0,-1.5);
\fill(0,1.5) circle(1.5pt);
\fill(0,-1.5) circle(1.5pt);
\filldraw[fill=white](0,0) circle(2pt);
\draw(-0.75,0) coordinate(L);
\draw[red,thick,-<-] (L) --++(-0.75,0);
\draw[red,thick,-<-={0.3}{}] (L) to[out=60,in=180] (30:1.5);
\draw[red,thick,-<-={0.2}{}] (L) ..controls ++(-45:0.5) and (0.3,-0.3).. (0.3,0)
..controls (0.3,0.3) and (-0.2,0.3).. (-0.2,0)
..controls (-0.2,-0.2) and (0.2,-0.2).. (0.2,0);
\node at (0,-2) {$-\varpi^\vee_2$};
{\color{mygreen}
\draw (0,0.4) circle(2pt) node[right=0.2em,scale=0.8]{$0$};
\draw (0,1) circle(2pt) node[right=0.2em,scale=0.8]{$-1$};
\draw (0,-0.5) circle(2pt) node[right=0.2em,scale=0.8]{$-1$};
\draw (0,-1) circle(2pt) node[right=0.2em,scale=0.8]{$0$};
\draw (1,0) circle(2pt) node[above=0.2em,scale=0.8]{$0$};
\draw (-1,0.3) circle(2pt) node[above=0.2em,scale=0.8]{$1$};
}
\end{scope}}
{\begin{scope}[xshift=12cm,yshift=-4.5cm]
\draw[blue] (0,0) circle(1.5cm);
\draw[blue] (0,1.5) -- (0,-1.5);
\fill(0,1.5) circle(1.5pt);
\fill(0,-1.5) circle(1.5pt);
\filldraw[fill=white](0,0) circle(2pt);
\draw(-0.75,0) coordinate(L);
\draw[red,thick,-<-] (L) --++(-0.75,0);
\draw[red,thick,-<-={0.3}{}] (L) ..controls ++(45:0.9)and (0.4,0.4).. (0.4,0);
\draw[red,thick,-<-={0.3}{}] (L) ..controls ++(-45:0.9) and (0.4,-0.4).. (0.4,0);
\draw[red,thick] (0,0) circle(0.3cm);
\draw[red,thick] (0,0) circle(0.2cm);
\draw[red,thick] (0.3,0) -- (0.4,0);
\draw[red,thick] (-20:0.2) -- (-20:0.3);
\draw[red,thick] (-40:0.1) -- (-40:0.2);
\node at (0,-2) {$\varpi^\vee_1-\varpi^\vee_2$};
{\color{mygreen}
\draw (0,0.6) circle(2pt) node[right=0.2em,scale=0.8]{$0$};
\draw (0,1) circle(2pt) node[right=0.2em,scale=0.8]{$0$};
\draw (0,-0.6) circle(2pt) node[right=0.2em,scale=0.8]{$-1$};
\draw (0,-1) circle(2pt) node[right=0.2em,scale=0.8]{$0$};
\draw (1,0) circle(2pt) node[above=0.2em,scale=0.8]{$0$};
\draw (-1,0.3) circle(2pt) node[above=0.2em,scale=0.8]{$1$};
}
\end{scope}}

{\begin{scope}[yshift=-9cm]
\draw[blue] (0,0) circle(1.5cm);
\draw[blue] (0,1.5) -- (0,-1.5);
\fill(0,1.5) circle(1.5pt);
\fill(0,-1.5) circle(1.5pt);
\filldraw[fill=white](0,0) circle(2pt);
\draw(-0.75,0) coordinate(L);
\draw(0.75,0) coordinate(R);
\draw[red,thick,-<-] (L) --++(-0.75,0);
\draw[red,thick,-<-={0.3}{}] (L) arc(180:0:0.75);
\draw[red,thick,-<-={0.3}{}] (L) arc(-180:0:0.75);
\draw[red,thick,->-] (R) --++(0.75,0);
\node at (0,-2) {$0$};
{\color{mygreen}
\draw (0,0.5) circle(2pt) node[right=0.2em,scale=0.8]{$0$};
\draw (0,1) circle(2pt) node[right=0.2em,scale=0.8]{$0$};
\draw (0,-0.5) circle(2pt) node[right=0.2em,scale=0.8]{$0$};
\draw (0,-1) circle(2pt) node[right=0.2em,scale=0.8]{$0$};
\draw (1,0.3) circle(2pt) node[above=0.2em,scale=0.8]{$-1$};
\draw (-1,0.3) circle(2pt) node[above=0.2em,scale=0.8]{$1$};
}
\end{scope}}

{\begin{scope}[xshift=4cm,yshift=-9cm]
\draw[blue] (0,0) circle(1.5cm);
\draw[blue] (0,1.5) -- (0,-1.5);
\fill(0,1.5) circle(1.5pt);
\fill(0,-1.5) circle(1.5pt);
\filldraw[fill=white](0,0) circle(2pt);
\draw(-0.75,0) coordinate(L);
\draw(0.75,0) coordinate(R);
\draw[red,thick,-<-] (L) --++(-0.75,0);
\draw[red,thick,-<-={0.3}{}] (L) arc(-180:0:0.75);
\draw[red,thick,->-] (R) --++(0.75,0);
\draw[red,thick,-<-={0.1}{}] (L) ..controls ++(60:0.5) and (0.3,0.3).. (0.3,0)
..controls (0.3,-0.3) and (-0.2,-0.3).. (-0.2,0)
..controls (-0.2,0.2) and (0.2,0.2).. (0.2,0);
\draw[red,thick,->-={0.2}{}] (R) to[out=120,in=0] (-220:1.5);
\node at (0,-2) {$\varpi^\vee_1$};
{\color{mygreen}
\draw (0,0.5) circle(2pt) node[left=0.2em,scale=0.8]{$0$};
\draw (0,1) circle(2pt) node[right=0.2em,scale=0.8]{$1$};
\draw (0,-0.5) circle(2pt) node[right=0.2em,scale=0.8]{$0$};
\draw (0,-1) circle(2pt) node[right=0.2em,scale=0.8]{$0$};
\draw (1,0.3) circle(2pt) node[above=0.2em,scale=0.8]{$-1$};
\draw (-1,0.3) circle(2pt) node[above=0.2em,scale=0.8]{$1$};
}
\end{scope}}

{\begin{scope}[xshift=8cm,yshift=-9cm]
\draw[blue] (0,0) circle(1.5cm);
\draw[blue] (0,1.5) -- (0,-1.5);
\fill(0,1.5) circle(1.5pt);
\fill(0,-1.5) circle(1.5pt);
\filldraw[fill=white](0,0) circle(2pt);
\draw(-0.75,0) coordinate(L);
\draw(0.75,0) coordinate(R);
\draw[red,thick,-<-] (L) --++(-0.75,0);
\draw[red,thick,-<-={0.3}{}] (L) arc(-180:0:0.75);
\draw[red,thick,->-] (R) --++(0.75,0);
\draw[red,thick,->-={0.1}{}] (R) ..controls ++(-225:0.5) and (-0.3,0.3).. (-0.3,0)
..controls (-0.3,-0.3) and (0.2,-0.3).. (0.2,0)
..controls (0.2,0.2) and (-0.2,0.2).. (-0.2,0);
\draw[red,thick,-<-={0.2}{}] (L) to[out=60,in=180] (40:1.5);
\node at (0,-2) {$-\varpi^\vee_1$};
{\color{mygreen}
\draw (0,0.5) circle(2pt) node[right=0.2em,scale=0.8]{$0$};
\draw (0,1) circle(2pt) node[right=0.2em,scale=0.8]{$-1$};
\draw (0,-0.5) circle(2pt) node[right=0.2em,scale=0.8]{$0$};
\draw (0,-1) circle(2pt) node[right=0.2em,scale=0.8]{$0$};
\draw (1,0.3) circle(2pt) node[above=0.2em,scale=0.8]{$-1$};
\draw (-1,0.3) circle(2pt) node[above=0.2em,scale=0.8]{$1$};
}
\end{scope}}
\end{tikzpicture}
    \caption{Elementary braids on $D^\ast$ and their shear coordinates. Below each picture, shown is the value of the tropicalized Casimir function assigned at the central puncture.}
    \label{fig:Casimir_check}
\end{figure}

\subsection{Equivalence to the cluster exact sequence}\label{subsec:cluster_exact_seq}
Here we compare the lamination exact sequence in \cref{prop:exact_seq_lamination} with the tropicalization of the cluster exact sequence \cite{FG09} (see \cref{sec:appendix}). 
Recall that for each puncture $p \in \bP$, associated is the tropicalized Casimir functions $\theta_p: \cL^x(\Sigma,\bQ) \to \mathsf{P}_\bQ^\vee$. Let $\langle \ ,\ \rangle:\mathsf{Q}_\bQ \otimes \mathsf{P}^\vee_\bQ \to \bQ$ denote the dual pairing, for which we have $\langle \alpha_s,\varpi_t^\vee\rangle =\delta_{st}$. 
Take an ideal triangulation $\tri(p)$ of $\Sigma$ so that the star neighborhood of $p$ is a once-punctured disk with two marked points on its boundary. Label the vertices of the quiver $Q^{\tri(p)}$ as prescribed in  \cref{fig:puncture_quiver_body}. 

\begin{figure}[ht]
\begin{tikzpicture}[>=latex]
\draw[blue] (0,0) circle(2cm);
\draw[blue] (0,2) -- (0,-2);
\node[fill,circle,inner sep=1pt] at (0,2) {};
\node[fill,circle,inner sep=1pt] at (0,-2) {};
\filldraw[fill=white](0,0) circle(2pt);
{\color{mygreen}
\draw(0,0.667) circle(2pt) coordinate(A1) node[above right,scale=0.8]{$1$};
\draw(0,1.333) circle(2pt) coordinate(A2) node[above right,scale=0.8]{$3$};
\draw(0,-0.667) circle(2pt) coordinate(B1) node[above right=0.2em,scale=0.8]{$2$};
\draw(0,-1.333) circle(2pt) coordinate(B2) node[below right,scale=0.8]{$5$};
\draw(1.333,0) circle(2pt) coordinate(C) node[right,scale=0.8]{$4$};
\draw(-1.333,0) circle(2pt) coordinate(D) node[left,scale=0.8]{$6$};
\draw(45:2) circle(2pt) coordinate(E1) node[above right,scale=0.8]{$7$};
\draw(-45:2) circle(2pt) coordinate(E2) node[below right,scale=0.8]{$8$};
\draw(135:2) circle(2pt) coordinate(F1) node[above left,scale=0.8]{$10$};
\draw(-135:2) circle(2pt) coordinate(F2) node[below left,scale=0.8]{$9$};
\qarrow{A2}{C}
\qarrow{C}{B2}
\qarrow{B2}{D}
\qarrow{D}{A2}
\qarrow{B1}{C}
\qarrow{C}{A1}
\qarrow{A1}{D}
\qarrow{D}{B1}
\qarrow{C}{E1}
\qarrow{E1}{A2}
\qarrow{B2}{E2}
\qarrow{E2}{C}
\qarrow{A2}{F1}
\qarrow{F1}{D}
\qarrow{D}{F2}
\qarrow{F2}{B2}
}
\end{tikzpicture}
    \caption{The quiver $Q^{\tri(p)}$ around a puncture $p$. Here possible (half-) arrows between boundary vertices are omitted.}
    \label{fig:puncture_quiver_body}
\end{figure}

\begin{prop}\label{prop:Casimir}
We have
\begin{align*}
    \langle \alpha_1, \sfx_{\tri(p)}^* \theta_p \rangle &= \sfx_3+ \sfx_4 + \sfx_5 + \sfx_6, \\
    \langle \alpha_2, \sfx_{\tri(p)}^* \theta_p \rangle &= \sfx_1+ \sfx_2.
\end{align*}
In particular, the map $\theta_p$ is identified with the tropicalization of the cluster Casimir map $\theta: \X_{\fsl_3,\Sigma} \to H_\X$ (\cref{subsec:cluster_sl3}). 
\end{prop}

\begin{proof}
By the $\bQ_{>0}$-equivariance, it suffices to show the equations for integral $\fsl_3$-laminations. Let $\hL$ be an integral $\fsl_3$-lamination represented by a non-elliptic signed web $W$, and $\cW$ the associated spiralling diagram in good position with respect to the split triangulation $\widehat{\tri}(p)$. Since both sides of the equations only depend on the restriction of $\cW$ to the once-punctured disk $D^\ast(p)$ containing the puncture $p$ with two special points, we can concentrate on the diagram $\cW_p:=\cW\cap D^\ast(p)$. Then it suffices to consider the elementary braids listed in \cref{fig:Casimir_check}, for which one can easily verify the desired equations. The cases for the opposite orientation can be similarly verified, or one can deduce them by applying \eqref{eq:Dynkin-theta} and \cref{prop:Dynkin-cluster}. 
The assertion is proved.
\end{proof}

\begin{thm}\label{thm:cluster_exact_sequence}
The sequence
\begin{align*}
    0 \to H_\A(\bQ^T) \to \cL^a(\Sigma,\bQ) \xrightarrow{p} \cL^x(\Sigma,\bQ) \xrightarrow{\theta} H_\X(\bQ^T) \to 0
\end{align*}
of PL maps (\cref{prop:exact_seq_lamination}) coincides with the cluster exact seqeunce \eqref{eq:cluster_exact_seq_sl3}. 
\end{thm}

\begin{proof}
By \cref{lem:tropical_torus_X,lem:tropical_torus_A}, the $\bQ$-vector spaces $H_\X(\bQ^T)$ and $H_\A(\bQ^T)$ introduced in \cref{subsec:lam_exact_seq} are the same as the tropicalizations of the tori $H_\X$ and $H_\A$, respectively.
Moreover, the ensemble map $p: \cL^a(\Sigma,\bQ) \to \cL^x(\Sigma,\bQ)$ and the tropicalized Casimir map $\theta:\cL^x(\Sigma,\bQ) \to H_\X(\bQ^T)$ coincide with the cluster ones by \cite[Proposition 4.10]{IK22} and \cref{prop:Casimir}, respectively. It remains to show that the $\mathsf{Q}_\bQ^\vee$-action adding peripheral components around a marked point $m \in \bM$ coincides with the tropicalization of the cluster action of $H_\A$ on $\A_{\fsl_3,\Sigma}$. 

For a puncture $p \in \bP$, take a labeled $\fsl_3$-triangulation $(\tri(p),\ell(p))$ as in \cref{fig:puncture_quiver_body}. Then the peripheral components  around $p$ with weight $1$ have the shear coordinates as shown in \cref{fig:peripheral_coordinates}. In particular, given any $\fsl_3$-lamination $\widetilde{L} \in \cL^a(\Sigma,\bQ)$, the action adding the peripheral component around $p$ with weight $t \in \bQ$ and counter-clockwise orientation shifts the tropical $\A$-coordinates as 
\begin{align*}
    \sfa_i(\widetilde{L}) \mapsto 
    \begin{cases}
    \sfa_i(\widetilde{L}) + \frac{2}{3}t, & \mbox{if $i=3,4,5,6$}, \\
    \sfa_i(\widetilde{L}) + \frac{1}{3}t, & \mbox{if $i=1,2$}, \\
    \sfa_i(\widetilde{L}) & \mbox{otherwise}.
    \end{cases}
\end{align*}
This is exactly the tropicalized version of the cluster action \eqref{eq:action_flow_form} induced by the element $\varpi^\vee_1=\frac{2}{3}\alpha^\vee_1 + \frac{1}{3}\alpha^\vee_2 \in K^\vee_\bQ$. Similarly, the action adding the peripheral component with clockwise orientation coincides with the tropicalized action given by $\varpi^\vee_2=\frac{1}{3}\alpha^\vee_1 + \frac{2}{3}\alpha^\vee_2 \in K^\vee$. Thus the $\mathsf{Q}^\vee_\bQ$-action in \cref{prop:exact_seq_lamination} associated with the puncture $p$ coincides with the cluster action. The coincidence of two actions associated with a special point is similarly verified. 
\end{proof}

\begin{rem}
The action given by $\beta \in H_\A(\bQ^T)$ preserves the integral part $\cL^a(\Sigma,\bZ)$ if and only if it belongs to the coroot lattices. 
\end{rem}

\begin{figure}[htbp]
\begin{tikzpicture}
\draw[blue] (0,0) circle(1.5cm);
\draw[blue] (0,1.5) -- (0,-1.5);
\fill(0,1.5) circle(1.5pt);
\fill(0,-1.5) circle(1.5pt);
\filldraw[fill=white](0,0) circle(2pt);
\draw[red,thick,->-] (0,0) circle(0.9cm);
\node at (0,-2) {$0$};
{\color{mygreen}
\draw (0,0.5) circle(2pt) node[right,scale=0.8]{$1/3$};
\draw (0,1.2) circle(2pt) node[right,scale=0.8]{$2/3$};
\draw (0,-0.5) circle(2pt) node[right,scale=0.8]{$1/3$};
\draw (0,-1.2) circle(2pt) node[right,scale=0.8]{$2/3$};
\draw (1.2,0) circle(2pt) node[above,scale=0.8]{$2/3$};
\draw (-1.2,0) circle(2pt) node[above,scale=0.8]{$2/3$};
}
{\begin{scope}[xshift=4cm]
\draw[blue] (0,0) circle(1.5cm);
\draw[blue] (0,1.5) -- (0,-1.5);
\fill(0,1.5) circle(1.5pt);
\fill(0,-1.5) circle(1.5pt);
\filldraw[fill=white](0,0) circle(2pt);
\draw[red,thick,-<-] (0,0) circle(0.9cm);
\node at (0,-2) {$0$};
{\color{mygreen}
\draw (0,0.5) circle(2pt) node[right,scale=0.8]{$2/3$};
\draw (0,1.2) circle(2pt) node[right,scale=0.8]{$1/3$};
\draw (0,-0.5) circle(2pt) node[right,scale=0.8]{$2/3$};
\draw (0,-1.2) circle(2pt) node[right,scale=0.8]{$1/3$};
\draw (1.2,0) circle(2pt) node[above,scale=0.8]{$1/3$};
\draw (-1.2,0) circle(2pt) node[above,scale=0.8]{$1/3$};
}
\end{scope}}
\end{tikzpicture}
    \caption{The shear coordinates of peripheral components around a puncture $p$.}
    \label{fig:peripheral_coordinates}
\end{figure}
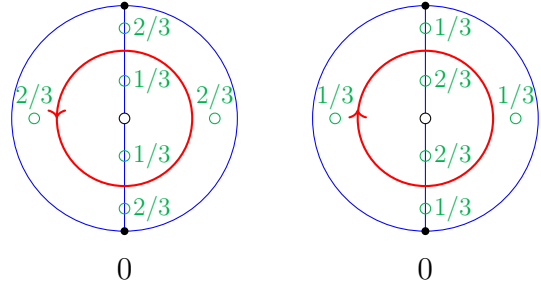

%% file: 5_FP.tex
\section{Relation to the work of Fraser--Pylyavskyy}\label{sec:FP}
Here we indicate some relations between our formulation of unbounded $\fsl_3$-laminations and the related work of Fraser--Pylyavskyy \cite{FP21}. In what follows, we use the identification $\mathsf{P}\cong\mathsf{P}^\vee$, $\varpi_s = \varpi_s^\vee$ via the normalized Killing form of $\fsl_3$. 


\subsection{Relation with the pseudo-tagged diagrams}
Our integral unbounded $\fsl_3$-laminations are related to \emph{pseudo-tagged diagrams} of Fraser--Pylyavskyy, as follows. Here is the $\fsl_3$-case of \cite[Definition 8.11]{FP21}:

\begin{dfn}
A \emph{pseudo-tagged $\fsl_3$-diagram} is a web $W$ on $\Sigma$ together with a subset $\varphi(e) \subset \{1,2,3\}$ assigned to each end $e$ of $W$ incident to a puncture $p \in \bM_\circ$. Here we impose
\begin{align*}
    |\varphi(e)| = \begin{cases}
    1 & \mbox{if $e$ is outgoing from $p$}, \\
    2 & \mbox{if $e$ is incoming to $p$}.
    \end{cases}
\end{align*}
\end{dfn}
Here is a representation-theoretic meaning of the data $\varphi(e)$. Recall that the weight lattice of $\fsl_3$ is realized as
\begin{align*}
    \mathsf{P} = \bZ^3 / \langle \mathbf{e}_1+\mathbf{e}_2+\mathbf{e}_3 \rangle_\bZ,
\end{align*}
where $(\mathbf{e}_1,\mathbf{e}_2,\mathbf{e}_3)$ denotes the standard basis of $\bZ^3$. Then for a subset $A \subset \{1,2,3\}$, we associate the vector $\iota_A:=\sum_{i \in A} \mathbf{e}_i \in \mathsf{P}$. Recall the fundamental weights $\varpi_1=\mathbf{e}_1$, $\varpi_2=\mathbf{e}_1+\mathbf{e}_2 \in \mathsf{P}$. Then we should have a correspondence so that the weight vector $\iota_{\varphi(e)} \in \mathsf{P}$ agrees with the contribution of $e$ to the tropicalized Casimir function $\theta_p$ (\cref{fig:weight-contribution}) of the end $e$. The correspondence is shown in \cref{fig:sign-tag}.

\begin{figure}[ht]
    \centering
\begin{tikzpicture}[scale=.91]
\begin{scope}
\draw[dashed] (0,0) circle(1cm);
\draw[red,thick,->-] (0,0) -- (0,1);
\node[red,scale=0.9] at (-0.2,0.2) {$+$};
\filldraw[draw=black,fill=white] (0,0) circle(2pt);
\draw[thick,<->] (1.25,0) --node[midway,above]{$\varpi_1$} (2.25,0);
\end{scope}
\begin{scope}[xshift=3.5cm]
\draw[dashed] (0,0) circle(1cm);
\draw[red,thick,->-] (0,0) -- (0,1);
\node[red,scale=0.8] at (-0.3,0.2) {$\{1\}$};
\filldraw[draw=black,fill=white] (0,0) circle(2pt);
\end{scope}
\begin{scope}[xshift=6cm]
\draw[dashed] (0,0) circle(1cm);
\draw[red,thick,-<-={0.7}{}] (0,1) -- (0,0.6);
\draw[red,thick,->-] (0,0.6) arc(90:270:0.3);
\draw[red,thick,->-] (0,0.6) arc(90:-90:0.3);
\node[red,scale=0.9] at (-0.3,0) {$+$};
\node[red,scale=0.9] at (0.3,0) {$-$};
\filldraw[draw=black,fill=white] (0,0) circle(2pt);
\draw[thick,<->] (1.25,0) --node[midway,above=0.2em,scale=0.9]{$\varpi_2-\varpi_1$} (2.25,0);
\end{scope}
\begin{scope}[xshift=9.5cm]
\draw[dashed] (0,0) circle(1cm);
\draw[red,thick,->-] (0,0) -- (0,1);
\node[red,scale=0.8] at (-0.3,0.2) {$\{2\}$};
\filldraw[draw=black,fill=white] (0,0) circle(2pt);
\end{scope}
\begin{scope}[xshift=12cm]
\draw[dashed] (0,0) circle(1cm);
\draw[red,thick,->-] (0,0) -- (0,1);
\node[red,scale=0.9] at (-0.2,0.2) {$-$};
\filldraw[draw=black,fill=white] (0,0) circle(2pt);
\draw[thick,<->] (1.25,0) --node[midway,above]{$-\varpi_2$} (2.25,0);
\end{scope}
\begin{scope}[xshift=15.5cm]
\draw[dashed] (0,0) circle(1cm);
\draw[red,thick,->-] (0,0) -- (0,1);
\node[red,scale=0.8] at (-0.3,0.2) {$\{3\}$};
\filldraw[draw=black,fill=white] (0,0) circle(2pt);
\end{scope}
\begin{scope}[xshift=12cm,yshift=-3cm]
\draw[dashed] (0,0) circle(1cm);
\draw[red,thick,-<-] (0,0) -- (0,1);
\node[red,scale=0.9] at (-0.2,0.2) {$+$};
\filldraw[draw=black,fill=white] (0,0) circle(2pt);
\draw[thick,<->] (1.25,0) --node[midway,above]{$\varpi_2$} (2.25,0);
\end{scope}
\begin{scope}[xshift=15.5cm,yshift=-3cm]
\draw[dashed] (0,0) circle(1cm);
\draw[red,thick,-<-] (0,0) -- (0,1);
\node[red,scale=0.8] at (-0.5,0.2) {$\{1,2\}$};
\filldraw[draw=black,fill=white] (0,0) circle(2pt);
\end{scope}
\begin{scope}[xshift=0cm,yshift=-3cm]
\draw[dashed] (0,0) circle(1cm);
\draw[red,thick,-<-] (0,0) -- (0,1);
\node[red,scale=0.9] at (-0.2,0.2) {$-$};
\filldraw[draw=black,fill=white] (0,0) circle(2pt);
\draw[thick,<->] (1.25,0) --node[midway,above]{$-\varpi_1$} (2.25,0);
\end{scope}
\begin{scope}[xshift=3.5cm,yshift=-3cm]
\draw[dashed] (0,0) circle(1cm);
\draw[red,thick,-<-] (0,0) -- (0,1);
\node[red,scale=0.8] at (-0.5,0.2) {$\{2,3\}$};
\filldraw[draw=black,fill=white] (0,0) circle(2pt);
\end{scope}
\begin{scope}[xshift=6cm,yshift=-3cm]
\draw[dashed] (0,0) circle(1cm);
\draw[red,thick,->-={0.7}{}] (0,1) -- (0,0.6);
\draw[red,thick,-<-] (0,0.6) arc(90:270:0.3);
\draw[red,thick,-<-] (0,0.6) arc(90:-90:0.3);
\node[red,scale=0.9] at (-0.3,0) {$+$};
\node[red,scale=0.9] at (0.3,0) {$-$};
\filldraw[draw=black,fill=white] (0,0) circle(2pt);
\draw[thick,<->] (1.25,0) --node[midway,above=0.2em,scale=0.9]{$\varpi_1-\varpi_2$} (2.25,0);
\end{scope}
\begin{scope}[xshift=9.5cm,yshift=-3cm]
\draw[dashed] (0,0) circle(1cm);
\draw[red,thick,-<-] (0,0) -- (0,1);
\node[red,scale=0.8] at (-0.5,0.2) {$\{1,3\}$};
\filldraw[draw=black,fill=white] (0,0) circle(2pt);
\end{scope}
\end{tikzpicture}
    \caption{Correspondence between the sign (Left) and the tag (Right) of Fraser--Pylyavskyy. The value of the tropicalized Casimir is shown above the arrow.}
    \label{fig:sign-tag}
\end{figure}
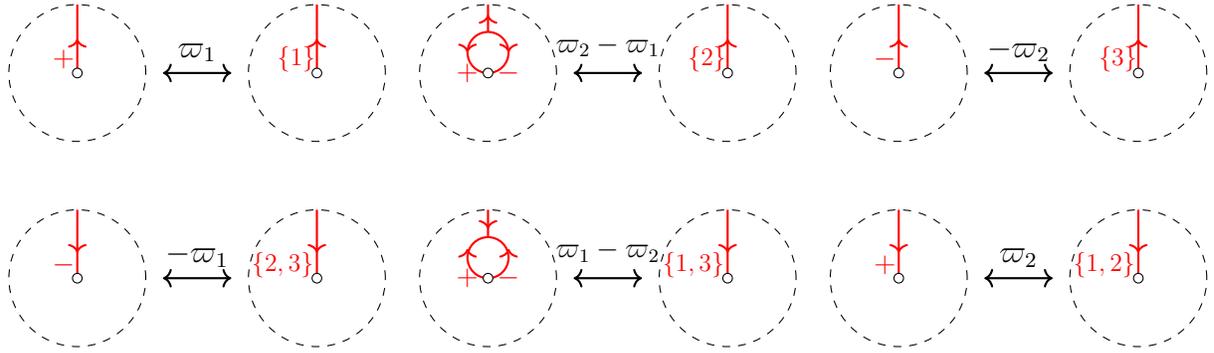

To each pseudo-tagged diagram $(W,\varphi)$, Fraser--Pylyavskyy associated a rational function $[W,\varphi] \in \mathcal{K}(\A_{SL_3,\Sigma})$ on the moduli space $\A_{SL_3,\Sigma}$ of decorated $SL_3$-local systems on $\Sigma$ \cite[Definition 8.12]{FP21}. They describe certain linear relations among these rational functions in their \emph{flattening theorem}, whose concrete form is given in \cite[Lemma 9.4]{FP21}. We are going to discuss a relation between the formulae given there and our resolution move (\cref{def:resolution_move}). We use the notation in \cite[Section 9]{FP21}.

In the $\fsl_3$-case, the non-trivial relations (in which the Grassmannian permutations $w_S$ are non-trivial) are the following:
\begin{itemize}
\item $(a,b,c)=(0,1,1)$: $S=\{2\}$, $w_S=(1\ 2)$, and
\begin{align*}
\mathord{
    \tikz[baseline=-0.6ex,scale=0.9]{
    \draw[dashed] (0,0)  circle(1.2cm);
    \draw[red,thick,->-] (-0.3,0) -- (45:1.2);
    \draw[red,thick,->-] (-0.3,0) -- (-45:1.2);
    \node[red,scale=0.8] at (-0.1,0.6) {$\{2\}$};
    \node[red,scale=0.8] at (-0.1,-0.6) {$\{1\}$};
    \filldraw[fill=white] (-0.3,0) circle(2pt);}
}\ 
= 
-\ \mathord{
    \tikz[baseline=-0.6ex,scale=0.9]{
    \draw[dashed] (0,0)  circle(1.2cm);
    \draw[red,thick,-<-] (-0.3,0) -- (0.3,0);
    \draw[red,thick,->-] (0.3,0) -- (45:1.2);
    \draw[red,thick,->-] (0.3,0) -- (-45:1.2);
    \node[red,scale=0.8] at (-0.2,0.3) {$\{1,2\}$};
    \filldraw[fill=white] (-0.3,0) circle(2pt);}
}
\ + \ 
\mathord{
    \tikz[baseline=-0.6ex,scale=0.9]{
    \draw[dashed] (0,0)  circle(1.2cm);
    \draw[red,thick,-<-] (-0.3,0) -- (0.2,0);
    \draw[red,thick,->-] (0.2,0) arc(0:180:0.5) arc(180:360:0.7) coordinate(A);
    \draw[red,thick,->-] (A) ..controls++(0,0.5) and (45:0.7).. (45:1.2);
    \draw[red,thick,->-={0.8}{}] (0.2,0) -- (-45:1.2);
    \node[red,scale=0.8] at (-0.2,-0.3) {$\{1,2\}$};
    \filldraw[fill=white] (-0.3,0) circle(2pt);}
}
\end{align*}

\item $(a,b,c)=(1,1,1)$: $S=\{1,3\}$, $w_S=(2\ 3)$, and
\begin{align*}
\mathord{
    \tikz[baseline=-0.6ex,scale=0.9]{
    \draw[dashed] (0,0)  circle(1.2cm);
    \draw[red,thick,-<-] (-0.3,0) -- (45:1.2);
    \draw[red,thick,-<-] (-0.3,0) -- (-45:1.2);
    \node[red,scale=0.8] at (-0.2,0.6) {$\{1,2\}$};
    \node[red,scale=0.8] at (-0.2,-0.6) {$\{1,3\}$};
    \filldraw[fill=white] (-0.3,0) circle(2pt);}
}\ 
= 
-\ \mathord{
    \tikz[baseline=-0.6ex,scale=0.9]{
    \draw[dashed] (0,0)  circle(1.2cm);
    \draw[red,thick,->-] (-0.3,0) -- (0.3,0);
    \draw[red,thick,-<-] (0.3,0) -- (45:1.2);
    \draw[red,thick,-<-] (0.3,0) -- (-45:1.2);
    \node[red,scale=0.8] at (-0.2,0.35) {$\{1\}$};
    \filldraw[fill=white] (-0.3,0) circle(2pt);}
}
\ +\  
\mathord{
    \tikz[baseline=-0.6ex,scale=0.9]{
    \draw[dashed] (0,0)  circle(1.2cm);
    \draw[red,thick,->-] (-0.3,0) -- (0.2,0);
    \draw[red,thick,-<-] (0.2,0) arc(0:180:0.5) arc(180:360:0.7) coordinate(A);
    \draw[red,thick,-<-] (A) ..controls++(0,0.5) and (45:0.7).. (45:1.2);
    \draw[red,thick,-<-={0.8}{}] (0.2,0) -- (-45:1.2);
    \node[red,scale=0.8] at (-0.2,-0.35) {$\{1\}$};
    \filldraw[fill=white] (-0.3,0) circle(2pt);}
}
\end{align*}
\item $(a,b,c)=(0,2,1)$: $S=\{3\}$, $w_S=(1\ 3)$, and

\begin{align*}
\mathord{
    \tikz[baseline=-0.6ex,scale=0.9]{
    \draw[dashed] (0,0)  circle(1.2cm);
    \draw[red,thick,-<-] (-0.3,0) -- (45:1.2);
    \draw[red,thick,->-] (-0.3,0) -- (-45:1.2);
    \node[red,scale=0.8] at (-0.2,0.6) {$\{1,2\}$};
    \node[red,scale=0.8] at (-0.2,-0.6) {$\{3\}$};
    \filldraw[fill=white] (-0.3,0) circle(2pt);}
}\ 
&= 
\ \mathord{
    \tikz[baseline=-0.6ex,scale=0.9]{
    \draw[dashed] (0,0)  circle(1.2cm);
		\draw[red,thick,->-] (45:1.2) ..controls (0,0.3) and (-0,-0.3).. (-45:1.2);
    \filldraw[fill=white] (-0.3,0) circle(2pt);}
}
\ - 2\ 
\mathord{
    \tikz[baseline=-0.6ex,scale=0.9]{
    \draw[dashed] (0,0)  circle(1.2cm);
    \draw[red,thick,-<-={0.1}{}] (-45:1.2) ..controls (0.2,-0.5) and (0.1,-0.1).. (0.1,0) arc(0:180:0.4) arc(180:360:0.5) coordinate(A);
    \draw[red,thick,-<-] (A) ..controls++(0,0.5) and (45:0.7).. (45:1.2);
    \filldraw[fill=white] (-0.3,0) circle(2pt);}
}\ 
+\ 
\mathord{
    \tikz[baseline=-0.6ex,scale=0.9]{
    \draw[dashed] (0,0)  circle(1.2cm);
    \draw[red,thick,-<-={0.05}{}] (-45:1.2) ..controls (0.2,-0.5) and (0.1,-0.1).. (0.1,0) arc(0:180:0.4) arc(180:360:0.5) arc(0:180:0.6) arc(180:360:0.7) coordinate(A);
    \draw[red,thick,-<-] (A) ..controls++(0,0.5) and (45:0.7).. (45:1.2);
    \filldraw[fill=white] (-0.3,0) circle(2pt);}
}
\\
&=\ 
\mathord{
    \tikz[baseline=-0.6ex,scale=0.9]{
    \draw[dashed] (0,0)  circle(1.2cm);
		\draw[red,thick,->-] (45:1.2) ..controls (0,0.3) and (-0,-0.3).. (-45:1.2);
    \filldraw[fill=white] (-0.3,0) circle(2pt);}
}\ 
+\ 
\mathord{
    \tikz[baseline=-0.6ex,scale=0.9]{
    \draw[dashed] (0,0)  circle(1.2cm);
		\draw[red,thick,->-={0.13}{},->-={0.87}{}] (45:1.2) ..controls (-1.2,0.3) and (-1.2,-0.3)..node[pos=0.15,inner sep=0](A){} node[pos=0.85,inner sep=0](B){} (-45:1.2);
	\draw[red,thick,-<-] (A) -- (B);
    \filldraw[fill=white] (-0.3,0) circle(2pt);}
}\ 
+\
\mathord{
    \tikz[baseline=-0.6ex,scale=0.9]{
    \draw[dashed] (0,0)  circle(1.2cm);
		\draw[red,thick,->-] (45:1.2) ..controls (-1.2,0.3) and (-1.2,-0.3).. (-45:1.2);
    \filldraw[fill=white] (-0.3,0) circle(2pt);}
}\ .
\end{align*}
Here the right-hand side is computed using the $\fsl_3$-skein relation. 
\item $(a,b,c)=(0,1,2)$: $S=\{2,3\}$, $w_S=(1\ 2)(2\ 3)$, and
\begin{align*}
\mathord{
    \tikz[baseline=-0.6ex,scale=0.9]{
    \draw[dashed] (0,0)  circle(1.2cm);
    \draw[red,thick,->-] (-0.3,0) -- (45:1.2);
    \draw[red,thick,-<-] (-0.3,0) -- (-45:1.2);
    \node[red,scale=0.8] at (-0.2,0.6) {$\{1\}$};
    \node[red,scale=0.8] at (-0.2,-0.6) {$\{2,3\}$};
    \filldraw[fill=white] (-0.3,0) circle(2pt);}
}\ 
&= 
\ \mathord{
    \tikz[baseline=-0.6ex,scale=0.9]{
    \draw[dashed] (0,0)  circle(1.2cm);
		\draw[red,thick,-<-] (45:1.2) ..controls (0,0.3) and (-0,-0.3).. (-45:1.2);
    \filldraw[fill=white] (-0.3,0) circle(2pt);}
}
\ - 2\ 
\mathord{
    \tikz[baseline=-0.6ex,scale=0.9]{
    \draw[dashed] (0,0)  circle(1.2cm);
    \draw[red,thick,->-={0.1}{}] (-45:1.2) ..controls (0.2,-0.5) and (0.1,-0.1).. (0.1,0) arc(0:180:0.4) arc(180:360:0.5) coordinate(A);
    \draw[red,thick,->-] (A) ..controls++(0,0.5) and (45:0.7).. (45:1.2);
    \filldraw[fill=white] (-0.3,0) circle(2pt);}
}\ 
+\ 
\mathord{
    \tikz[baseline=-0.6ex,scale=0.9]{
    \draw[dashed] (0,0)  circle(1.2cm);
    \draw[red,thick,->-={0.05}{}] (-45:1.2) ..controls (0.2,-0.5) and (0.1,-0.1).. (0.1,0) arc(0:180:0.4) arc(180:360:0.5) arc(0:180:0.6) arc(180:360:0.7) coordinate(A);
    \draw[red,thick,->-] (A) ..controls++(0,0.5) and (45:0.7).. (45:1.2);
    \filldraw[fill=white] (-0.3,0) circle(2pt);}
}
\\
&=\ 
\mathord{
    \tikz[baseline=-0.6ex,scale=0.9]{
    \draw[dashed] (0,0)  circle(1.2cm);
		\draw[red,thick,-<-] (45:1.2) ..controls (0,0.3) and (-0,-0.3).. (-45:1.2);
    \filldraw[fill=white] (-0.3,0) circle(2pt);}
}\ 
+\ 
\mathord{
    \tikz[baseline=-0.6ex,scale=0.9]{
    \draw[dashed] (0,0)  circle(1.2cm);
		\draw[red,thick,-<-={0.13}{},-<-={0.87}{}] (45:1.2) ..controls (-1.2,0.3) and (-1.2,-0.3)..node[pos=0.15,inner sep=0](A){} node[pos=0.85,inner sep=0](B){} (-45:1.2);
	\draw[red,thick,->-] (A) -- (B);
    \filldraw[fill=white] (-0.3,0) circle(2pt);}
}\ 
+\
\mathord{
    \tikz[baseline=-0.6ex,scale=0.9]{
    \draw[dashed] (0,0)  circle(1.2cm);
		\draw[red,thick,-<-] (45:1.2) ..controls (-1.2,0.3) and (-1.2,-0.3).. (-45:1.2);
    \filldraw[fill=white] (-0.3,0) circle(2pt);}
}\ .
\end{align*}
Here the computation is similar to the previous case.
\end{itemize}
By connecting $\tikz[baseline=-0.6ex,scale=0.5]{
    \draw[dashed] (0,0)  circle(1.2cm);
		\draw[red,thick,->-] (45:1.2) ..controls (45:1.5) and (1.5,0.5).. (1.5,0) ..controls (1.5,-0.75) and (0.75,-1.5).. (0,-1.5) ..controls (-0.5,-1.5) and (-135:1.5).. (-135:1.2);
    \filldraw[fill=white] (0,0) circle(2pt);}$ 
to the both sides in the first two cases and appropriately rotating the latter two cases, we get the following relations:
\begin{align}
\mathord{
    \tikz[baseline=-0.6ex,scale=0.9]{
    \draw[dashed] (0,0)  circle(1.2cm);
    \draw[red,thick,->-] (-0,0) -- (180:1.2);
    \draw[red,thick,->-] (-0,0) -- (0:1.2);
    \node[red,scale=0.8] at (-0.5,0.4) {$\{2\}$};
    \node[red,scale=0.8] at (0.5,0.4) {$\{1\}$};
    \filldraw[fill=white] (-0,0) circle(2pt);}
}\ 
&= 
\ \mathord{
    \tikz[baseline=-0.6ex,scale=0.9]{
    \draw[dashed] (0,0)  circle(1.2cm);
    \draw[red,thick,-<-] (-0,0) -- (0,-0.5);
    \draw[red,thick,->-] (0,-0.5) to[out=0,in=180] (0:1.2);
    \draw[red,thick,->-] (0,-0.5) to[out=180,in=0] (180:1.2);
    \node[red,scale=0.8] at (-0.4,0.2) {$\{1,2\}$};
    \filldraw[fill=white] (-0,0) circle(2pt);}
}
\ +\  
\mathord{
    \tikz[baseline=-0.6ex,scale=0.9]{
    \draw[dashed] (0,0)  circle(1.2cm);
    \draw[red,thick,-<-] (-0,0) -- (0,0.5);
    \draw[red,thick,->-] (0,0.5) to[out=0,in=180] (0:1.2);
    \draw[red,thick,->-] (0,0.5) to[out=180,in=0] (180:1.2);
    \node[red,scale=0.8] at (-0.4,-0.2) {$\{1,2\}$};
    \filldraw[fill=white] (-0,0) circle(2pt);}
} \label{eq:RY_1}\\
\mathord{
    \tikz[baseline=-0.6ex,scale=0.9]{
    \draw[dashed] (0,0)  circle(1.2cm);
    \draw[red,thick,-<-] (-0,0) -- (180:1.2);
    \draw[red,thick,-<-] (-0,0) -- (0:1.2);
    \node[red,scale=0.8] at (-0.5,0.4) {$\{1,2\}$};
    \node[red,scale=0.8] at (0.5,0.4) {$\{1,3\}$};
    \filldraw[fill=white] (-0,0) circle(2pt);}
}\ 
&= 
\ \mathord{
    \tikz[baseline=-0.6ex,scale=0.9]{
    \draw[dashed] (0,0)  circle(1.2cm);
    \draw[red,thick,->-] (-0,0) -- (0,-0.5);
    \draw[red,thick,-<-] (0,-0.5) to[out=0,in=180] (0:1.2);
    \draw[red,thick,-<-] (0,-0.5) to[out=180,in=0] (180:1.2);
    \node[red,scale=0.8] at (-0.4,0.2) {$\{1\}$};
    \filldraw[fill=white] (-0,0) circle(2pt);}
}
\ +\  
\mathord{
    \tikz[baseline=-0.6ex,scale=0.9]{
    \draw[dashed] (0,0)  circle(1.2cm);
    \draw[red,thick,->-] (-0,0) -- (0,0.5);
    \draw[red,thick,-<-] (0,0.5) to[out=0,in=180] (0:1.2);
    \draw[red,thick,-<-] (0,0.5) to[out=180,in=0] (180:1.2);
    \node[red,scale=0.8] at (-0.4,-0.2) {$\{1\}$};
    \filldraw[fill=white] (-0,0) circle(2pt);}
} \label{eq:RY_2}\\
\mathord{
    \tikz[baseline=-0.6ex,scale=0.9]{
    \draw[dashed] (0,0)  circle(1.2cm);
    \draw[red,thick,-<-] (-0,0) -- (180:1.2);
    \draw[red,thick,->-] (-0,0) -- (0:1.2);
    \node[red,scale=0.8] at (-0.5,0.4) {$\{1,2\}$};
    \node[red,scale=0.8] at (0.5,0.4) {$\{3\}$};
    \filldraw[fill=white] (-0,0) circle(2pt);}
}\ 
&= 
\ 
\mathord{
    \tikz[baseline=-0.6ex,scale=0.9]{
    \draw[dashed] (0,0)  circle(1.2cm);
    \draw[red,thick,->-={0.1}{}] (0,0.5) to[out=0,in=180] (0:1.2);
    \draw[red,thick] (0,0.5) to[out=180,in=0] (180:1.2);
    \filldraw[fill=white] (-0,0) circle(2pt);}
}
\ +\  
\mathord{
    \tikz[baseline=-0.6ex,scale=0.9]{
    \draw[dashed] (0,0)  circle(1.2cm);
    \draw[red,thick,->-] (0.5,0) -- (0:1.2);
    \draw[red,thick,-<-] (-0.5,0) -- (180:1.2);
    \draw[red,thick] (0,0) circle(0.5cm);
    \filldraw[fill=white] (-0,0) circle(2pt);}
		}\ 
+\  
\mathord{
    \tikz[baseline=-0.6ex,scale=0.9]{
    \draw[dashed] (0,0)  circle(1.2cm);
    \draw[red,thick,->-={0.1}{}] (0,-0.5) to[out=0,in=180] (0:1.2);
    \draw[red,thick] (0,-0.5) to[out=180,in=0] (180:1.2);
    \filldraw[fill=white] (-0,0) circle(2pt);}
}\label{eq:RY_3}\\
\mathord{
    \tikz[baseline=-0.6ex,scale=0.9]{
    \draw[dashed] (0,0)  circle(1.2cm);
    \draw[red,thick,->-] (-0,0) -- (180:1.2);
    \draw[red,thick,-<-] (-0,0) -- (0:1.2);
    \node[red,scale=0.8] at (-0.5,0.4) {$\{1\}$};
    \node[red,scale=0.8] at (0.5,0.4) {$\{2,3\}$};
    \filldraw[fill=white] (-0,0) circle(2pt);}
}\ 
&= 
\ 
\mathord{
    \tikz[baseline=-0.6ex,scale=0.9]{
    \draw[dashed] (0,0)  circle(1.2cm);
    \draw[red,thick,-<-={0.1}{}] (0,0.5) to[out=0,in=180] (0:1.2);
    \draw[red,thick] (0,0.5) to[out=180,in=0] (180:1.2);
    \filldraw[fill=white] (-0,0) circle(2pt);}
}
\ +\  
\mathord{
    \tikz[baseline=-0.6ex,scale=0.9]{
    \draw[dashed] (0,0)  circle(1.2cm);
    \draw[red,thick,-<-] (0.5,0) -- (0:1.2);
    \draw[red,thick,->-] (-0.5,0) -- (180:1.2);
    \draw[red,thick] (0,0) circle(0.5cm);
    \filldraw[fill=white] (-0,0) circle(2pt);}
		}\ 
+\  
\mathord{
    \tikz[baseline=-0.6ex,scale=0.9]{
    \draw[dashed] (0,0)  circle(1.2cm);
    \draw[red,thick,-<-={0.1}{}] (0,-0.5) to[out=0,in=180] (0:1.2);
    \draw[red,thick] (0,-0.5) to[out=180,in=0] (180:1.2);
    \filldraw[fill=white] (-0,0) circle(2pt);}
}\label{eq:RY_4}
\end{align}
We remark here that the right-hand sides of \eqref{eq:RY_1}--\eqref{eq:RY_4} coincide with (the classical limit of) the higher-rank analogue of the 
\emph{Roger--Yang skein relations} \cite{RY} found by Shen--Sun--Weng and systematically studied in their upcoming work \cite{SSW_RY}.\footnote{T. I. first learned their $\fsl_3$-version of Roger--Yang relations corresponding to \eqref{eq:RY_3}, \eqref{eq:RY_4} from Zhe Sun. Those corresponding to \eqref{eq:RY_1}, \eqref{eq:RY_2} have been found independently by T.I. and Wataru Yuasa, and Shen--Sun--Weng.}  
From the correspondence given in \cref{fig:sign-tag}, the left-hand sides of \eqref{eq:RY_1}--\eqref{eq:RY_4} correspond to
\begin{align*}
\mathord{
    \tikz[baseline=-0.6ex,scale=0.9]{
    \draw[dashed] (0,0)  circle(1.2cm);
    \draw[red,thick,->-] (-0.5,0) -- (180:1.2);
    \draw[red,thick,-<-={0.4}{},->-={0.7}{}] (-0.25,0) circle(0.25cm);
    \draw[red,thick,->-] (-0,0) -- (0:1.2);
    \node[red,scale=0.8] at (-0.2,0.4) {$+$};
    \node[red,scale=0.8] at (-0.2,-0.4) {$-$};
    \node[red,scale=0.8] at (0.3,0.2) {$+$};
    \filldraw[fill=white] (-0,0) circle(2pt);}
}\ ,\quad \mathord{
    \tikz[baseline=-0.6ex,scale=0.9]{
    \draw[dashed] (0,0)  circle(1.2cm);
    \draw[red,thick,-<-] (-0.5,0) -- (180:1.2);
    \draw[red,thick,->-={0.4}{},-<-={0.7}{}] (-0.25,0) circle(0.25cm);
    \draw[red,thick,-<-] (-0,0) -- (0:1.2);
    \node[red,scale=0.8] at (-0.2,0.4) {$+$};
    \node[red,scale=0.8] at (-0.2,-0.4) {$-$};
    \node[red,scale=0.8] at (0.3,0.2) {$+$};
    \filldraw[fill=white] (-0,0) circle(2pt);}
}\ ,\quad \mathord{
    \tikz[baseline=-0.6ex,scale=0.9]{
    \draw[dashed] (0,0)  circle(1.2cm);
    \draw[red,thick,-<-] (-0,0) -- (180:1.2);
    \draw[red,thick,->-] (-0,0) -- (0:1.2);
    \node[red,scale=0.8] at (-0.3,0.2) {$+$};
    \node[red,scale=0.8] at (0.3,0.2) {$-$};
    \filldraw[fill=white] (-0,0) circle(2pt);}
}\ ,\quad \mathord{
    \tikz[baseline=-0.6ex,scale=0.9]{
    \draw[dashed] (0,0)  circle(1.2cm);
    \draw[red,thick,->-] (-0,0) -- (180:1.2);
    \draw[red,thick,-<-] (-0,0) -- (0:1.2);
    \node[red,scale=0.8] at (-0.3,0.2) {$+$};
    \node[red,scale=0.8] at (0.3,0.2) {$-$};
    \filldraw[fill=white] (-0,0) circle(2pt);}
}\ ,
\end{align*}
respectively. Then we see that after applying the resolution move (\cref{def:resolution_move}) to them, they give rise to one of the terms in the right-hand sides. Therefore our approach seems to be consistent with the Fraser--Pylyavskyy's work, under the belief that our $\fsl_3$-laminations should parametrize the ``highest (or lowest) terms'' of web functions on $\A_{SL_3,\Sigma}$ in some sense. The justification of this point will require further research.

\begin{rem}
While we have the basic correspondence between our signed webs and the pseudo-tagged diagrams of Fraser--Pylyavskyy, we still need care to obtain a well-defined map between them. 
For example, the signed web
\begin{align*}
\tikz[baseline=-0.6ex,scale=0.9]{
    \draw[dashed] (0,0)  circle(1.2cm);
    \draw[red,thick,->-] (0.5,0) -- (0:1.2);
    \draw[red,thick,->-={0.3}{},-<-={0.8}{}] (0.25,0) circle(0.25cm);
    \draw[red,thick,->-] (-0,0) -- (180:1.2);
    \node[red,scale=0.8] at (0.15,0.5) {$+$};
    \node[red,scale=0.8] at (0.15,-0.5) {$-$};
    \node[red,scale=0.8] at (-0.3,0.2) {$+$};
    \filldraw[fill=white] (-0,0) circle(2pt);},
\end{align*}
may be viewed as either of the pseudo-tagged diagrams
\begin{align*}
     \tikz[baseline=-0.6ex,scale=0.9]{
    \draw[dashed] (0,0)  circle(1.2cm);
    \draw[red,thick,->-] (-0,0) -- (180:1.2);
    \draw[red,thick,->-] (-0,0) -- (0:1.2);
    \node[red,scale=0.8] at (-0.5,0.4) {$\{1\}$};
    \node[red,scale=0.8] at (0.5,0.4) {$\{2\}$};
    \filldraw[fill=white] (-0,0) circle(2pt);} \quad \mbox{or} \quad 
    \tikz[baseline=-0.6ex,scale=0.9]{
    \draw[dashed] (0,0)  circle(1.2cm);
    \draw[red,thick,->-] (0.5,0) -- (0:1.2);
    \draw[red,thick,->-={0.3}{},-<-={0.8}{}] (0.25,0) circle(0.25cm);
    \draw[red,thick,->-] (-0,0) -- (180:1.2);
    \node[red,scale=0.8] at (0.2,0.6) {$\{1,2\}$};
    \node[red,scale=0.8] at (0.2,-0.6) {$\{2,3\}$};
    \node[red,scale=0.8] at (-0.5,0.4) {$\{1\}$};
    \filldraw[fill=white] (-0,0) circle(2pt);},
\end{align*}
according to the rule shown in \cref{fig:sign-tag}. 
It seems that the first one is appropriate. Indeed, one can show that the second one is zero using the relation \eqref{eq:RY_4} and the usual $\fsl_3$-skein relations. 
See \cite[Remark 8.16]{FP21} for a seemingly related remark. We do not pursue this point here.
\end{rem}

\subsection{Relation with the dosps}
In \cite[Sections 5--7]{FP21}, they formulate a higher rank analogue of \emph{tagged triangulations} used in the $\fsl_2$-case (\cite{FST}) to describe certain seeds obtained by mutating those associated with the usual ideal triangulations of $\Sigma$. Their idea is as follows.

They firstly consider the weights of cluster variables with respect to the $H$-action on $\A_{SL_3,\Sigma}$ associated with each puncture $p$. They formulated a certain compatibility condition on a collection of $H$-weights, and conjectured that the collection of $H$-weights of any cluster should be compatible \cite[Conjecture 5.4]{FP21}. Such a compatible collection of $H$-weights gives rise to a \emph{decorated ordered set partition} (abbreviated as ``dosp") \cite[Definition 6.6]{FP21}, which is a choice of a Weyl region (\emph{i.e.}, a facet in the Coxeter complex) together with some decoration by signs. The dosps give a higher rank analogue of tags. Then \cite[Theorem 7.3]{FP21} describes the projection from the ``good part" of the exchange graph to the graph of dosps (whose edges are given by \emph{dosp mutations}) at each puncture. Summarizing, each cluster in $\A_{SL_3,\Sigma}$ is (conjecturally) colored by dosps at punctures, and the seed mutations induce dosp mutations. 

Now we are going to relate our construction to this theory. Extending the conjectures by Fomin--Pylyavskyy \cite{FP14,FP16}, Fraser--Pylyavskyy give conjectures \cite[Section 8.5]{FP21} that each cluster is represented by a certain collection of pseudo-tagged diagrams. We then expect that such a collection corresponds to a collection of signed webs. 

\begin{dfn}
The \emph{lamination cluster} associated to $v \in \bExch_{\fsl_3,\Sigma}$ is the collection $\{W_j\}_{j \in I_\uf}$ of signed webs on $\Sigma$ satisfying $\sfx^{(v)}_i(W_j)=\delta_{ij}$ for $i,j \in I_\uf$.
\end{dfn}
In particular, each lamination cluster $\{W_j\}_{j \in I_\uf}$ gives rise to a \emph{$P$-cluster} $\{\theta_p(W_j)^\ast\}_{j \in I_\uf} \subset \mathsf{P}$ in the sense of \cite{FP21}, which determines a dosp. 

Let us consider the punctured disk $D^\ast$ with two special points. Recall the classification of spiralling diagrams in \cref{prop:classification}. 
In this case, their conjectures are shown to be true \cite[Proposition 10.3]{FP21}. 
Following the proof of \cite[Theorem 7.3]{FP21}, we first move on to the cluster that gives the initial seed in the corresponding \emph{Grassmannian moduli space} $\A'_{SL_3,D^\ast}$ \cite[Definition 3.3]{FP21} (cluster type $D_4$), up to the deletion of cluster variables with zero weight. 
The mutation sequence starting from our initial cluster is shown in \cref{fig:transf_Grassmannian}. 

\begin{figure}[ht]
    \centering
\begin{tikzpicture}[>=latex,scale=0.9]
\draw[blue] (0,0) circle(2cm);
\draw[blue] (0,2) -- (0,-2);
\node[fill,circle,inner sep=1pt] at (0,2) {};
\node[fill,circle,inner sep=1pt] at (0,-2) {};
\filldraw[fill=white](0,0) circle(2pt);
{\color{mygreen}
\draw(0,0.667) circle(2pt) coordinate(A1) node[above right,scale=0.8]{$1$};
\draw(0,1.333) circle(2pt) coordinate(A2) node[above right,scale=0.8]{$3$};
\draw(0,-0.667) circle(2pt) coordinate(B1) node[above right=0.2em,scale=0.8]{$2$};
\draw(0,-1.333) circle(2pt) coordinate(B2) node[below right,scale=0.8]{$5$};
\draw(1.333,0) circle(2pt) coordinate(C) node[right,scale=0.8]{$4$};
\draw(-1.333,0) circle(2pt) coordinate(D) node[left,scale=0.8]{$6$};
\qarrow{A2}{C}
\qarrow{C}{B2}
\qarrow{B2}{D}
\qarrow{D}{A2}
\qarrow{B1}{C}
\qarrow{C}{A1}
\qarrow{A1}{D}
\qarrow{D}{B1}
}
\draw[->,thick] (2.5,0) --node[midway,above]{$\mu_5$} (3.5,0);
\begin{scope}[xshift=6cm]
\draw[blue] (0,0) circle(2cm);
\draw[blue] (0,2) -- (0,-2);
\node[fill,circle,inner sep=1pt] at (0,2) {};
\node[fill,circle,inner sep=1pt] at (0,-2) {};
\filldraw[fill=white](0,0) circle(2pt);
{\color{mygreen}
\draw(0,0.667) circle(2pt) coordinate(A1) node[above right,scale=0.8]{$1$};
\draw(0,1.333) circle(2pt) coordinate(A2) node[above right,scale=0.8]{$3$};
\draw(0,-0.667) circle(2pt) coordinate(B1) node[above right=0.2em,scale=0.8]{$2$};
\draw(0,-1.333) circle(2pt) coordinate(B2) node[below right,scale=0.8]{$5$};
\draw(1.333,0) circle(2pt) coordinate(C) node[right,scale=0.8]{$4$};
\draw(-1.333,0) circle(2pt) coordinate(D) node[left,scale=0.8]{$6$};
\qarrow{A2}{C}
\qarrow{B2}{C}
\qarrow{D}{B2}
\qarrow{D}{A2}
\qarrow{C}{D}
\qarrow{B1}{C}
\qarrow{C}{A1}
\qarrow{A1}{D}
\qarrow{D}{B1}
}
\draw[->,thick] (2.5,0) --node[midway,above]{$\mu_3$} (3.5,0);
\end{scope}
\begin{scope}[xshift=12cm]
\draw[blue] (0,0) circle(2cm);
\draw[blue] (0,2) -- (0,-2);
\node[fill,circle,inner sep=1pt] at (0,2) {};
\node[fill,circle,inner sep=1pt] at (0,-2) {};
\filldraw[fill=white](0,0) circle(2pt);
{\color{mygreen}
\draw(0,0.667) circle(2pt) coordinate(A1) node[above right,scale=0.8]{$1$};
\draw(0,1.333) circle(2pt) coordinate(A2) node[above right,scale=0.8]{$3$};
\draw(0,-0.667) circle(2pt) coordinate(B1) node[above right=0.2em,scale=0.8]{$2$};
\draw(0,-1.333) circle(2pt) coordinate(B2) node[below right,scale=0.8]{$5$};
\draw(1.333,0) circle(2pt) coordinate(C) node[right,scale=0.8]{$4$};
\draw(-1.333,0) circle(2pt) coordinate(D) node[left,scale=0.8]{$6$};
\qarrow{C}{A2}
\qarrow{B2}{C}
\qarrow{D}{B2}
\qarrow{A2}{D}
\qarrow{B1}{C}
\qarrow{C}{A1}
\qarrow{A1}{D}
\qarrow{D}{B1}
}
\end{scope}
\end{tikzpicture}
    \caption{Transformation from the initial seed in $\A_{SL_3,D^\ast}$ to the initial seed in the Grassmannian moduli space $\A'_{SL_3,D^\ast}$. }
    \label{fig:transf_Grassmannian}
\end{figure}
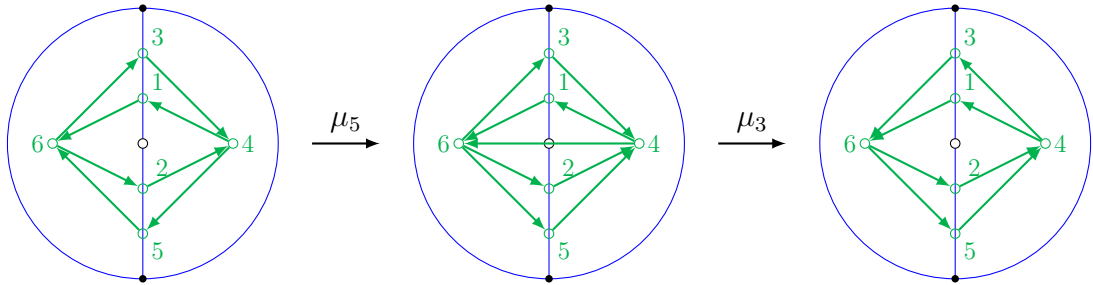

\begin{figure}[ht]
    \centering
\begin{tikzpicture}
\node[scale=0.7] at (0,1.2) {
$\tikz{\draw(0,0) circle(1cm);\foreach \i in {80,260} \fill(\i:1) circle(1.5pt);\draw[red,thick,->-](0,1)--(0,0) node[above left]{$+$};\filldraw[fill=white](0,0) circle(2pt);}$
};
\node[scale=0.7] at (0,3) {
$\tikz{\draw(0,0) circle(1cm);\foreach \i in {80,260} \fill(\i:1) circle(1.5pt);\draw[red,thick,-<-](0,1)--(0,0) node[above left]{$+$};\filldraw[fill=white](0,0) circle(2pt);}$
};
\node[scale=0.7] at (-1.5,0) {
$\tikz{\draw(0,0) circle(1cm);\foreach \i in {80,260} \fill(\i:1) circle(1.5pt);
\draw[red,thick,->-={0.3}{},-<-={0.7}{}] (0,1) to[bend right=50pt] node[pos=0.5,inner sep=0](A){} (0,-1);
\draw[red,thick,-<-](A)--(0,0) node[above]{$+$};
\filldraw[fill=white](0,0) circle(2pt);}$
};
\node[scale=0.7] at (1.5,0) {
$\tikz{\draw(0,0) circle(1cm);\foreach \i in {80,260} \fill(\i:1) circle(1.5pt);
\draw[red,thick,->-={0.3}{},-<-={0.7}{}] (0,1) to[bend left=50pt] node[pos=0.5,inner sep=0](A){} (0,-1);
\draw[red,thick,-<-](A)--(0,0) node[above]{$+$};
\filldraw[fill=white](0,0) circle(2pt);}$
};
\node[scale=0.7] at (0,-1.2) {
$\tikz{\draw(0,0) circle(1cm);\foreach \i in {80,260} \fill(\i:1) circle(1.5pt);\draw[red,thick,->-](0,-1)--(0,0) node[below left]{$+$};\filldraw[fill=white](0,0) circle(2pt);}$
};
\node[scale=0.7] at (0,-3) {
$\tikz{\draw(0,0) circle(1cm);\foreach \i in {80,260} \fill(\i:1) circle(1.5pt);\draw[red,thick,-<-](0,-1)--(0,0) node[below left]{$+$};\filldraw[fill=white](0,0) circle(2pt);}$
};
\node[scale=0.8] at (1,1.2){$\varpi_1$};
\node[scale=0.8] at (1,3){$\varpi_2$};
\node[scale=0.8] at (1,-1.2){$\varpi_1$};
\node[scale=0.8] at (1,-3){$\varpi_2$};
\node[scale=0.8] at (-0.5,0){$\varpi_2$};
\node[scale=0.8] at (2.5,0){$\varpi_2$};
\draw[thick,->] (3,0) --node[midway,above]{$\mu_5$} (4,0);

\begin{scope}[xshift=6.5cm]
\node[scale=0.7] at (0,1.2) {
$\tikz{\draw(0,0) circle(1cm);\foreach \i in {80,260} \fill(\i:1) circle(1.5pt);\draw[red,thick,->-](0,1)--(0,0) node[above left]{$+$};\filldraw[fill=white](0,0) circle(2pt);}$
};
\node[scale=0.7] at (0,3) {
$\tikz{\draw(0,0) circle(1cm);\foreach \i in {80,260} \fill(\i:1) circle(1.5pt);\draw[red,thick,-<-](0,1)--(0,0) node[above left]{$+$};\filldraw[fill=white](0,0) circle(2pt);}$
};
\node[scale=0.7] at (-1.5,0) {
$\tikz{\draw(0,0) circle(1cm);\foreach \i in {80,260} \fill(\i:1) circle(1.5pt);
\draw[red,thick,->-={0.3}{},-<-={0.7}{}] (0,1) to[bend right=50pt] node[pos=0.5,inner sep=0](A){} (0,-1);
\draw[red,thick,-<-](A)--(0,0) node[above]{$+$};
\filldraw[fill=white](0,0) circle(2pt);}$
};
\node[scale=0.7] at (1.5,0) {
$\tikz{\draw(0,0) circle(1cm);\foreach \i in {80,260} \fill(\i:1) circle(1.5pt);
\draw[red,thick,->-={0.3}{},-<-={0.7}{}] (0,1) to[bend left=50pt] node[pos=0.5,inner sep=0](A){} (0,-1);
\draw[red,thick,-<-](A)--(0,0) node[above]{$+$};
\filldraw[fill=white](0,0) circle(2pt);}$
};
\node[scale=0.7] at (0,-1.2) {
$\tikz{\draw(0,0) circle(1cm);\foreach \i in {80,260} \fill(\i:1) circle(1.5pt);\draw[red,thick,->-](0,-1)--(0,0) node[below left]{$+$};\filldraw[fill=white](0,0) circle(2pt);}$
};
\node[scale=0.7] at (0,-3) {
$\tikz{\draw(0,0) circle(1cm);\foreach \i in {80,260} \fill(\i:1) circle(1.5pt);\draw[red,thick,-<-](0,-1)--(0,-0.3); \draw[red,thick,-<-={0.3}{},->-={0.75}{}] (0,1) to[out=-135,in=180] (0,-0.3) to[out=0,in=-45] (0,1);\filldraw[fill=white](0,0) circle(2pt);}$
};
\node[scale=0.8] at (1,1.2){$\varpi_1$};
\node[scale=0.8] at (1,3){$\varpi_2$};
\node[scale=0.8] at (1,-1.2){$\varpi_1$};
\node[scale=0.8] at (1,-3){$0$};
\node[scale=0.8] at (-0.5,0){$\varpi_2$};
\node[scale=0.8] at (2.5,0){$\varpi_2$};
\draw[thick,->] (3,0) --node[midway,above]{$\mu_3$} (4,0);
\end{scope}

\begin{scope}[xshift=13cm]
\node[scale=0.7] at (0,1.2) {
$\tikz{\draw(0,0) circle(1cm);\foreach \i in {80,260} \fill(\i:1) circle(1.5pt);\draw[red,thick,->-](0,1)--(0,0) node[above left]{$+$};\filldraw[fill=white](0,0) circle(2pt);}$
};
\node[scale=0.7] at (0,3) {
$\tikz{\draw(0,0) circle(1cm);\foreach \i in {80,260} \fill(\i:1) circle(1.5pt);\draw[red,thick,-<-](0,1)--(0,0.3); \draw[red,thick,-<-={0.3}{},->-={0.75}{}] (0,-1) to[out=135,in=180] (0,0.3) to[out=0,in=45] (0,-1);\filldraw[fill=white](0,0) circle(2pt);}$
};
\node[scale=0.7] at (-1.5,0) {
$\tikz{\draw(0,0) circle(1cm);\foreach \i in {80,260} \fill(\i:1) circle(1.5pt);
\draw[red,thick,->-={0.3}{},-<-={0.7}{}] (0,1) to[bend right=50pt] node[pos=0.5,inner sep=0](A){} (0,-1);
\draw[red,thick,-<-](A)--(0,0) node[above]{$+$};
\filldraw[fill=white](0,0) circle(2pt);}$
};
\node[scale=0.7] at (1.5,0) {
$\tikz{\draw(0,0) circle(1cm);\foreach \i in {80,260} \fill(\i:1) circle(1.5pt);
\draw[red,thick,->-={0.3}{},-<-={0.7}{}] (0,1) to[bend left=50pt] node[pos=0.5,inner sep=0](A){} (0,-1);
\draw[red,thick,-<-](A)--(0,0) node[above]{$+$};
\filldraw[fill=white](0,0) circle(2pt);}$
};
\node[scale=0.7] at (0,-1.2) {
$\tikz{\draw(0,0) circle(1cm);\foreach \i in {80,260} \fill(\i:1) circle(1.5pt);\draw[red,thick,->-](0,-1)--(0,0) node[below left]{$+$};\filldraw[fill=white](0,0) circle(2pt);}$
};
\node[scale=0.7] at (0,-3) {
$\tikz{\draw(0,0) circle(1cm);\foreach \i in {80,260} \fill(\i:1) circle(1.5pt);\draw[red,thick,-<-](0,-1)--(0,-0.3); \draw[red,thick,-<-={0.3}{},->-={0.75}{}] (0,1) to[out=-135,in=180] (0,-0.3) to[out=0,in=-45] (0,1);\filldraw[fill=white](0,0) circle(2pt);}$
};
\node[scale=0.8] at (1,1.2){$\varpi_1$};
\node[scale=0.8] at (1,3){$0$};
\node[scale=0.8] at (1,-1.2){$\varpi_1$};
\node[scale=0.8] at (1,-3){$0$};
\node[scale=0.8] at (-0.5,0){$\varpi_2$};
\node[scale=0.8] at (2.5,0){$\varpi_2$};
\end{scope}
\end{tikzpicture}
    \caption{Transformation of the lamination clusters corresponding to \cref{fig:transf_Grassmannian}. Here the values of $\theta_p^\ast$ are indicated. During the process, the associated dosp is $1|2|3$.}
    \label{fig:transf_Grassmannian_web}
\end{figure}
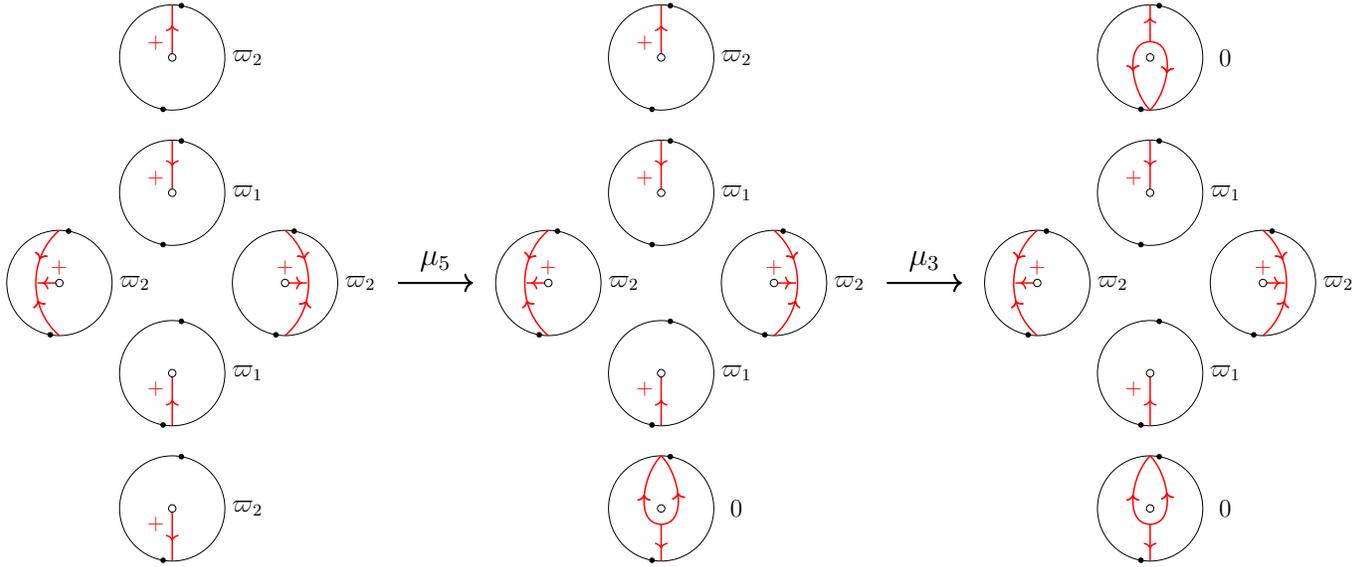

Deleting the vertices $3,5$ in the right-most picture in \cref{fig:transf_Grassmannian}, we get 
the initial quiver in \cite[Example 5.7]{FP21}. The corresponding sequence of lamination clusters is shown in \cref{fig:transf_Grassmannian_web}. Observe that the signed webs associated to the vertices $3,5$ in the right-most picture have $\theta_p^\ast=0$, and thus do not contribute to the dosp at $p$. During this sequence, the associated dosp is $1|2|3$ (dominant chamber). 

Then one can walk around the exchange graph of type $D_4$ (freezing the vertices $3,5$) to get all the dosps following the algorithm illustrated in \cite[Example 7.7 and Lemma 7.8]{FP21}. See also \cite[Figure 2]{FP21}. 
For example, the mutation sequence and the corresponding lamination clusters that realize the dosp mutation $1|2|3 \to 12|3 \to 123^+$ are shown in \cref{fig:seq_dosp,fig:seq_dosp_web}, respectively.

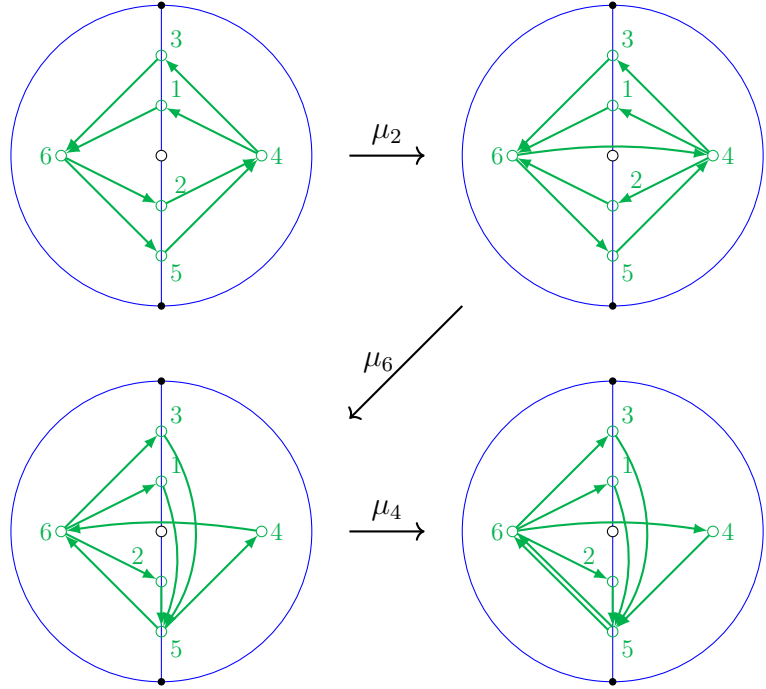
\begin{figure}[ht]
    \centering
\begin{tikzpicture}
\draw[blue] (0,0) circle(2cm);
\draw[blue] (0,2) -- (0,-2);
\node[fill,circle,inner sep=1pt] at (0,2) {};
\node[fill,circle,inner sep=1pt] at (0,-2) {};
\filldraw[fill=white](0,0) circle(2pt);
{\color{mygreen}
\draw(0,0.667) circle(2pt) coordinate(A1) node[above right,scale=0.8]{$1$};
\draw(0,1.333) circle(2pt) coordinate(A2) node[above right,scale=0.8]{$3$};
\draw(0,-0.667) circle(2pt) coordinate(B1) node[above right=0.2em,scale=0.8]{$2$};
\draw(0,-1.333) circle(2pt) coordinate(B2) node[below right,scale=0.8]{$5$};
\draw(1.333,0) circle(2pt) coordinate(C) node[right,scale=0.8]{$4$};
\draw(-1.333,0) circle(2pt) coordinate(D) node[left,scale=0.8]{$6$};
\qarrow{C}{A2}
\qarrow{B2}{C}
\qarrow{D}{B2}
\qarrow{A2}{D}
\qarrow{B1}{C}
\qarrow{C}{A1}
\qarrow{A1}{D}
\qarrow{D}{B1}
}
\draw[->,thick] (2.5,0) --node[midway,above]{$\mu_2$} (3.5,0);
\begin{scope}[xshift=6cm]
\draw[blue] (0,0) circle(2cm);
\draw[blue] (0,2) -- (0,-2);
\node[fill,circle,inner sep=1pt] at (0,2) {};
\node[fill,circle,inner sep=1pt] at (0,-2) {};
\filldraw[fill=white](0,0) circle(2pt);
{\color{mygreen}
\draw(0,0.667) circle(2pt) coordinate(A1) node[above right,scale=0.8]{$1$};
\draw(0,1.333) circle(2pt) coordinate(A2) node[above right,scale=0.8]{$3$};
\draw(0,-0.667) circle(2pt) coordinate(B1) node[above right=0.4em,scale=0.8]{$2$};
\draw(0,-1.333) circle(2pt) coordinate(B2) node[below right,scale=0.8]{$5$};
\draw(1.333,0) circle(2pt) coordinate(C) node[right,scale=0.8]{$4$};
\draw(-1.333,0) circle(2pt) coordinate(D) node[left,scale=0.8]{$6$};
\qarrow{C}{A2}
\qarrow{B2}{C}
\qarrow{D}{B2}
\qarrow{A2}{D}
\qarrow{C}{B1}
\qarrow{C}{A1}
\qarrow{A1}{D}
\qarrow{B1}{D}
\uniarrow{D}{C}{bend left=0.3cm,shorten >=2pt,shorten <=2pt}
}
\draw[->,thick] (-2,-2) --node[midway,left]{$\mu_6$} ++(-1.5,-1.5);
\end{scope}
\begin{scope}[yshift=-5cm]
\draw[blue] (0,0) circle(2cm);
\draw[blue] (0,2) -- (0,-2);
\node[fill,circle,inner sep=1pt] at (0,2) {};
\node[fill,circle,inner sep=1pt] at (0,-2) {};
\filldraw[fill=white](0,0) circle(2pt);
{\color{mygreen}
\draw(0,0.667) circle(2pt) coordinate(A1) node[above right,scale=0.8]{$1$};
\draw(0,1.333) circle(2pt) coordinate(A2) node[above right,scale=0.8]{$3$};
\draw(0,-0.667) circle(2pt) coordinate(B1) node[above left=0.4em,scale=0.8]{$2$};
\draw(0,-1.333) circle(2pt) coordinate(B2) node[below right,scale=0.8]{$5$};
\draw(1.333,0) circle(2pt) coordinate(C) node[right,scale=0.8]{$4$};
\draw(-1.333,0) circle(2pt) coordinate(D) node[left,scale=0.8]{$6$};
\qarrow{B2}{C}
\qarrow{B2}{D}
\qarrow{D}{A2}
\qarrow{D}{A1}
\qarrow{D}{B1}
\uniarrow{C}{D}{bend right=0.3cm,shorten >=2pt,shorten <=2pt}
\uniarrow{B1}{B2}{shorten >=2pt,shorten <=2pt}
\uniarrow{A1}{B2}{bend left=0.7cm,shorten >=2pt,shorten <=2pt}
\uniarrow{A2}{B2}{bend left=1.2cm,shorten >=2pt,shorten <=2pt}
}
\draw[->,thick] (2.5,0) --node[midway,above]{$\mu_4$} (3.5,0);
\end{scope}

\begin{scope}[xshift=6cm,yshift=-5cm]
\draw[blue] (0,0) circle(2cm);
\draw[blue] (0,2) -- (0,-2);
\node[fill,circle,inner sep=1pt] at (0,2) {};
\node[fill,circle,inner sep=1pt] at (0,-2) {};
\filldraw[fill=white](0,0) circle(2pt);
{\color{mygreen}
\draw(0,0.667) circle(2pt) coordinate(A1) node[above right,scale=0.8]{$1$};
\draw(0,1.333) circle(2pt) coordinate(A2) node[above right,scale=0.8]{$3$};
\draw(0,-0.667) circle(2pt) coordinate(B1) node[above left=0.4em,scale=0.8]{$2$};
\draw(0,-1.333) circle(2pt) coordinate(B2) node[below right,scale=0.8]{$5$};
\draw(1.333,0) circle(2pt) coordinate(C) node[right,scale=0.8]{$4$};
\draw(-1.333,0) circle(2pt) coordinate(D) node[left,scale=0.8]{$6$};
\qarrow{C}{B2}
\qarrow{D}{A2}
\qarrow{D}{A1}
\qarrow{D}{B1}
\uniarrow{D}{C}{bend left=0.3cm,shorten >=2pt,shorten <=2pt}
\uniarrow{B1}{B2}{shorten >=2pt,shorten <=2pt}
\uniarrow{A1}{B2}{bend left=0.7cm,shorten >=2pt,shorten <=2pt}
\uniarrow{A2}{B2}{bend left=1.2cm,shorten >=2pt,shorten <=2pt}
\node[inner sep=1.5pt] (B2) at (B2) {};
\node[inner sep=1.5pt] (D) at (D) {};
\uniarrow{B2.90}{D.-20}{}
\uniarrow{B2.160}{D.-80}{}
}
\end{scope}
\end{tikzpicture}
    \caption{Mutation sequence that realizes the dosp mutation $1|2|3 \to 12|3 \to 123^+$.}
    \label{fig:seq_dosp}
\end{figure}

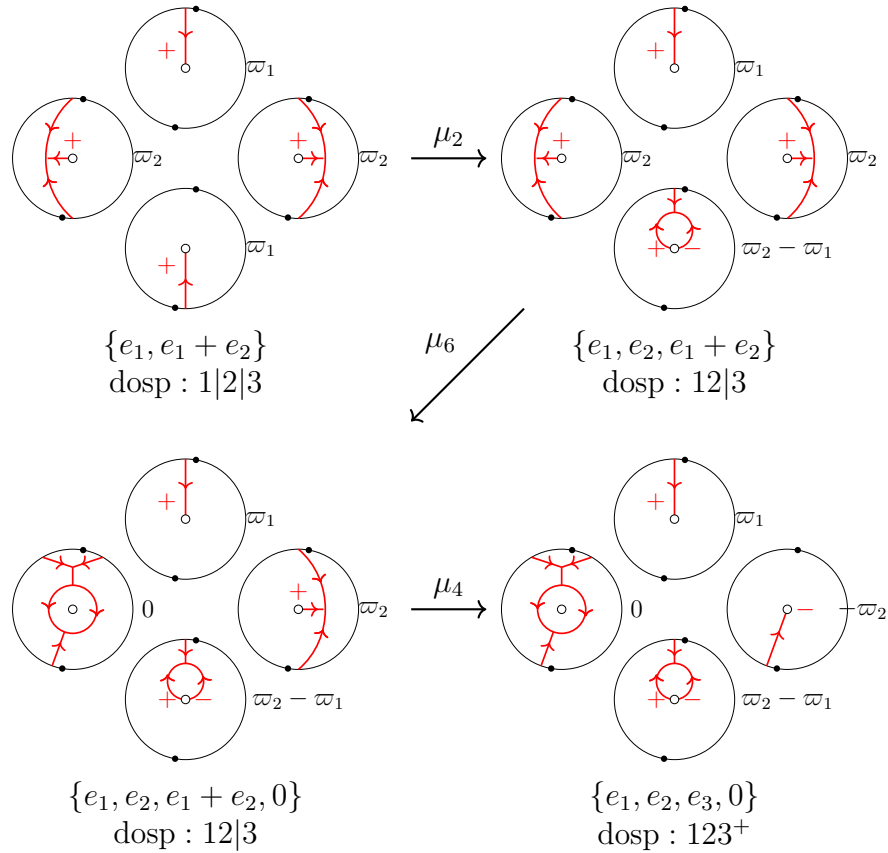
\begin{figure}[ht]
    \centering
\begin{tikzpicture}
\node[scale=0.8] at (0,1.2) {
$\tikz{\draw(0,0) circle(1cm);\foreach \i in {80,260} \fill(\i:1) circle(1.5pt);
\draw[red,thick,->-](0,1)--(0,0) node[above left]{$+$};\filldraw[fill=white](0,0) circle(2pt);}$
};
\node[scale=0.8] at (-1.5,0) {
$\tikz{\draw(0,0) circle(1cm);\foreach \i in {80,260} \fill(\i:1) circle(1.5pt);
\draw[red,thick,->-={0.3}{},-<-={0.7}{}] (0,1) to[bend right=50pt] node[pos=0.5,inner sep=0](A){} (0,-1);
\draw[red,thick,-<-](A)--(0,0) node[above]{$+$};
\filldraw[fill=white](0,0) circle(2pt);}$
};
\node[scale=0.8] at (1.5,0) {
$\tikz{\draw(0,0) circle(1cm);\foreach \i in {80,260} \fill(\i:1) circle(1.5pt);
\draw[red,thick,->-={0.3}{},-<-={0.7}{}] (0,1) to[bend left=50pt] node[pos=0.5,inner sep=0](A){} (0,-1);
\draw[red,thick,-<-](A)--(0,0) node[above]{$+$};
\filldraw[fill=white](0,0) circle(2pt);}$
};
\node[scale=0.8] at (0,-1.2) {
$\tikz{\draw(0,0) circle(1cm);\foreach \i in {80,260} \fill(\i:1) circle(1.5pt);
\draw[red,thick,->-](0,-1)--(0,0) node[below left]{$+$};\filldraw[fill=white](0,0) circle(2pt);}$
};
\node[scale=0.8] at (1,1.2){$\varpi_1$};
\node[scale=0.8] at (1,-1.2){$\varpi_1$};
\node[scale=0.8] at (-0.5,0){$\varpi_2$};
\node[scale=0.8] at (2.5,0){$\varpi_2$};
\node at (0,-2.5) {$\{e_1,e_1+e_2\}$};
\node at (0,-3) {$\mathrm{dosp}: 1|2|3$};
\draw[thick,->] (3,0) --node[midway,above]{$\mu_2$} (4,0);

\begin{scope}[xshift=6.5cm]
\node[scale=0.8] at (0,1.2) {
$\tikz{\draw(0,0) circle(1cm);\foreach \i in {80,260} \fill(\i:1) circle(1.5pt);
\draw[red,thick,->-](0,1)--(0,0) node[above left]{$+$};\filldraw[fill=white](0,0) circle(2pt);}$
};
\node[scale=0.8] at (-1.5,0) {
$\tikz{\draw(0,0) circle(1cm);\foreach \i in {80,260} \fill(\i:1) circle(1.5pt);
\draw[red,thick,->-={0.3}{},-<-={0.7}{}] (0,1) to[bend right=50pt] node[pos=0.5,inner sep=0](A){} (0,-1);
\draw[red,thick,-<-](A)--(0,0) node[above]{$+$};
\filldraw[fill=white](0,0) circle(2pt);}$
};
\node[scale=0.8] at (1.5,0) {
$\tikz{\draw(0,0) circle(1cm);\foreach \i in {80,260} \fill(\i:1) circle(1.5pt);
\draw[red,thick,->-={0.3}{},-<-={0.7}{}] (0,1) to[bend left=50pt] node[pos=0.5,inner sep=0](A){} (0,-1);
\draw[red,thick,-<-](A)--(0,0) node[above]{$+$};
\filldraw[fill=white](0,0) circle(2pt);}$
};
\node[scale=0.8] at (0,-1.2) {
$\tikz{\draw(0,0) circle(1cm);\foreach \i in {80,260} \fill(\i:1) circle(1.5pt);
\draw[red,thick,->-={0.3}{},-<-={0.75}{}](0,0) arc(-90:270:0.3); \draw[red,thick,->-={0.7}{}] (0,1) -- (0,0.6);\node[red] at (-0.3,0) {$+$};\node[red] at (0.3,0) {$-$};\filldraw[fill=white](0,0) circle(2pt);}$
};
\node[scale=0.8] at (1,1.2){$\varpi_1$};
\node[scale=0.8] at (1.5,-1.2){$\varpi_2-\varpi_1$};
\node[scale=0.8] at (-0.5,0){$\varpi_2$};
\node[scale=0.8] at (2.5,0){$\varpi_2$};
\node at (0,-2.5) {$\{e_1,e_2,e_1+e_2\}$};
\node at (0,-3) {$\mathrm{dosp}: 12|3$};
\draw[thick,->] (-2,-2) --node[midway,above left]{$\mu_6$} ++(-1.5,-1.5);
\end{scope}

\begin{scope}[yshift=-6cm]
\node[scale=0.8] at (0,1.2) {
$\tikz{\draw(0,0) circle(1cm);\foreach \i in {80,260} \fill(\i:1) circle(1.5pt);
\draw[red,thick,->-](0,1)--(0,0) node[above left]{$+$};\filldraw[fill=white](0,0) circle(2pt);}$
};
\node[scale=0.8] at (-1.5,0) {
$\tikz{\draw(0,0) circle(1cm);\foreach \i in {80,260} \fill(\i:1) circle(1.5pt);
\draw[red,thick,-<-={0}{},->-={0.5}{}] (0,0) circle(0.4cm);
\draw[red,thick] (0,0.4) -- (0,0.7);
\draw[red,thick,-<-] (0,0.7) -- (60:1);
\draw[red,thick,-<-] (0,0.7) -- (120:1);
\draw[red,thick,-<-] (-110:0.4) -- (-110:1);
\filldraw[fill=white](0,0) circle(2pt);}$
};
\node[scale=0.8] at (1.5,0) {
$\tikz{\draw(0,0) circle(1cm);\foreach \i in {80,260} \fill(\i:1) circle(1.5pt);
\draw[red,thick,->-={0.3}{},-<-={0.7}{}] (0,1) to[bend left=50pt] node[pos=0.5,inner sep=0](A){} (0,-1);
\draw[red,thick,-<-](A)--(0,0) node[above]{$+$};
\filldraw[fill=white](0,0) circle(2pt);}$
};
\node[scale=0.8] at (0,-1.2) {
$\tikz{\draw(0,0) circle(1cm);\foreach \i in {80,260} \fill(\i:1) circle(1.5pt);
\draw[red,thick,->-={0.3}{},-<-={0.75}{}](0,0) arc(-90:270:0.3); \draw[red,thick,->-={0.7}{}] (0,1) -- (0,0.6);\node[red] at (-0.3,0) {$+$};\node[red] at (0.3,0) {$-$};\filldraw[fill=white](0,0) circle(2pt);}$
};
\node[scale=0.8] at (1,1.2){$\varpi_1$};
\node[scale=0.8] at (1.5,-1.2){$\varpi_2-\varpi_1$};
\node[scale=0.8] at (-0.5,0){$0$};
\node[scale=0.8] at (2.5,0){$\varpi_2$};
\node at (0,-2.5) {$\{e_1,e_2,e_1+e_2,0\}$};
\node at (0,-3) {$\mathrm{dosp}: 12|3$};
\draw[thick,->] (3,0) --node[midway,above]{$\mu_4$} (4,0);
\end{scope}

\begin{scope}[xshift=6.5cm,yshift=-6cm]
\node[scale=0.8] at (0,1.2) {
$\tikz{\draw(0,0) circle(1cm);\foreach \i in {80,260} \fill(\i:1) circle(1.5pt);
\draw[red,thick,->-](0,1)--(0,0) node[above left]{$+$};\filldraw[fill=white](0,0) circle(2pt);}$
};
\node[scale=0.8] at (-1.5,0) {
$\tikz{\draw(0,0) circle(1cm);\foreach \i in {80,260} \fill(\i:1) circle(1.5pt);
\draw[red,thick,-<-={0}{},->-={0.5}{}] (0,0) circle(0.4cm);
\draw[red,thick] (0,0.4) -- (0,0.7);
\draw[red,thick,-<-] (0,0.7) -- (60:1);
\draw[red,thick,-<-] (0,0.7) -- (120:1);
\draw[red,thick,-<-] (-110:0.4) -- (-110:1);
\filldraw[fill=white](0,0) circle(2pt);}$
};
\node[scale=0.8] at (1.5,0) {
$\tikz{\draw(0,0) circle(1cm);\foreach \i in {80,260} \fill(\i:1) circle(1.5pt);
\draw[red,thick,-<-] (0,0) -- (-110:1);
\node[red] at (0.3,0){$-$};
\filldraw[fill=white](0,0) circle(2pt);}$
};
\node[scale=0.8] at (0,-1.2) {
$\tikz{\draw(0,0) circle(1cm);\foreach \i in {80,260} \fill(\i:1) circle(1.5pt);
\draw[red,thick,->-={0.3}{},-<-={0.75}{}](0,0) arc(-90:270:0.3); \draw[red,thick,->-={0.7}{}] (0,1) -- (0,0.6);\node[red] at (-0.3,0) {$+$};\node[red] at (0.3,0) {$-$};\filldraw[fill=white](0,0) circle(2pt);}$
};
\node[scale=0.8] at (1,1.2){$\varpi_1$};
\node[scale=0.8] at (1.5,-1.2){$\varpi_2-\varpi_1$};
\node[scale=0.8] at (-0.5,0){$0$};
\node[scale=0.8] at (2.5,0){$-\varpi_2$};
\node at (0,-2.5) {$\{e_1,e_2,e_3,0\}$};
\node at (0,-3) {$\mathrm{dosp}: 123^+$};
\end{scope}
\end{tikzpicture}
    \caption{The sequence of lamination clusters corresponding to \cref{fig:seq_dosp}. The associated $P$-clusters and the dosps are also shown.}
    \label{fig:seq_dosp_web}
\end{figure}

\section{A brief discussion on higher rank generalizations}\label{sec:higher}
Here we propose a possible direction of higher rank generalizations of our formulation of unbounded laminations. See also the first arXiv version of \cite{Kim21}. 

We propose to use the approach of \cite{MOY} and \cite{CKM}. 
For a complex semisimple Lie algebra $\mathfrak{g}$, a $\mathfrak{g}$-web on $\Sigma$ is an immersed oriented uni-trivalent graph $W$ on $\Sigma$, where each edge is colored by a vertex $s \in S$ of the Dynkin diagram of $\mathfrak{g}$, corresponding to the $s$-th fundamental weight $\varpi_s$. 
The boundary condition on the univalent vertices is the same as in the $\fsl_3$-case. 
At each trivalent vertex, the colors $(s,t,u)$ are required to satisfy the condition
\begin{align*}
    \dim_\bC (V(\varpi_s) \otimes V(\varpi_t) \otimes V(\varpi_u))^G =1,
\end{align*}
where $V(\lambda)$ denotes the irreducible representation with the highest weight $\lambda$. 

A signed $\mathfrak{g}$-web is a $\mathfrak{g}$-web together with a sign ($+$ or $-$) assigned to each end incident to a puncture. Then a rational unbounded $\mathfrak{g}$-laminations should be defined as a certain equivalence class of a signed $\mathfrak{g}$-web equipped with rational weights on its components. A correct description of equivalence relations as given in \cref{sec:lamination} will require additional elaboration in the higher rank cases. Assume that the space $\mathcal{L}_\mathfrak{g}^x(\Sigma,\bQ)$ of rational unbounded $\mathfrak{g}$-laminations on $\Sigma$ is appropriately defined. 

For each puncture $p \in \bM_\circ$, we should have the tropicalized Casimir map
\begin{align*}
    \theta_p: \mathcal{L}_\mathfrak{g}^x(\Sigma,\bQ) \to \mathsf{P}_\bQ^\vee.
\end{align*}
The contribution of each signed end to $\theta_p$ should be as shown in \cref{fig:weight-contribution-general}. Here $\ast: S \to S$, $s \mapsto s^\ast$ is the Dynkin involution defined by $w_0\varpi_s=-\varpi_{s^\ast}$, where $w_0 \in W(\mathfrak{g})$ is the longest element of the Weyl group. 

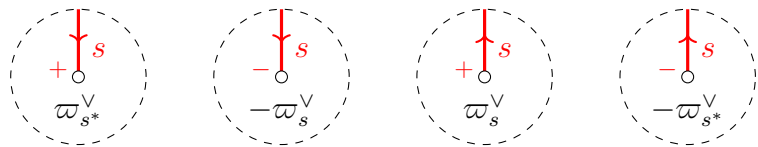
\begin{figure}[ht]
\begin{tikzpicture}[scale=0.9]
    \draw[dashed, fill=white] (0,0) circle [radius=1];
    \draw[red,very thick,->-] (0,1) -- (0,0);
    \filldraw[fill=white](0,0) circle(2.5pt);
    \node[red,scale=0.8] at (-0.3,0.1) {$+$};
    \node[red] at (0.3,0.4) {$s$};
    \node at (0,-0.5) {$\varpi^\vee_{s^\ast}$};
    
    \begin{scope}[xshift=3cm]
    \draw[dashed, fill=white] (0,0) circle [radius=1];
    \draw[red,very thick,->-] (0,1) -- (0,0);
    \filldraw[fill=white](0,0) circle(2.5pt);
    \node[red,scale=0.8] at (-0.3,0.1) {$-$};
    \node[red] at (0.3,0.4) {$s$};
    \node at (0,-0.5) {$-\varpi^\vee_{s}$};
    \end{scope}
    
    \begin{scope}[xshift=6cm]
    \draw[dashed, fill=white] (0,0) circle [radius=1];
    \draw[red,very thick,-<-] (0,1) -- (0,0);
    \filldraw[fill=white](0,0) circle(2.5pt);
    \node[red,scale=0.8] at (-0.3,0.1) {$+$};
    \node[red] at (0.3,0.4) {$s$};
    \node at (0,-0.5) {$\varpi^\vee_s$};
    \end{scope}
    
    \begin{scope}[xshift=9cm]
    \draw[dashed, fill=white] (0,0) circle [radius=1];
    \draw[red,very thick,-<-] (0,1) -- (0,0);
    \filldraw[fill=white](0,0) circle(2.5pt);
    \node[red,scale=0.8] at (-0.3,0.1) {$-$};
    \node[red] at (0.3,0.4) {$s$};
    \node at (0,-0.5) {$-\varpi^\vee_{s^\ast}$};
    \end{scope}
\end{tikzpicture}
    \caption{Contribution to $\theta_p(\hL)$ from each end incident to the puncture $p$.}
    \label{fig:weight-contribution-general}
\end{figure}
Observe that the data of signs and the fundamental weights are enough for the values of $\theta_p$ to exhaust the coweight lattice $\mathsf{P}^\vee$, so we expect no additional data assigned to the ends. 

Recall the construction of cluster variables $\Delta_{\lambda,\mu,\nu}$ on $\A_{G^\vee,\Sigma}$ from \cite[Section 9]{GS19}. Here $G^\vee$ is the simply-connected group with the Langlands dual Lie algebra $\mathfrak{g}^\vee$. In particular, they are determined by the triple $(\lambda,\mu,\nu)$ of weights for $\mathfrak{g}^\vee$ (=coweights for $\mathfrak{g}$) up to a scalar. 
Based on the quantum duality conjecture \cite{FG09}, the $\mathfrak{g}$-webs $W$ should correspond to functions $F_W$ on $\A_{G^\vee,\Sigma}$ in a similar manner as in \cite{IKar,IK22}, at least when $\bM_\circ=\emptyset$. Some of the $\mathfrak{g}$-webs corresponding to cluster variables on the dual side are shown in \cref{fig:elementary laminations}, together with the $\mathfrak{g}^\vee$-weights (=$\mathfrak{g}$-coweights) of the cluster variables.

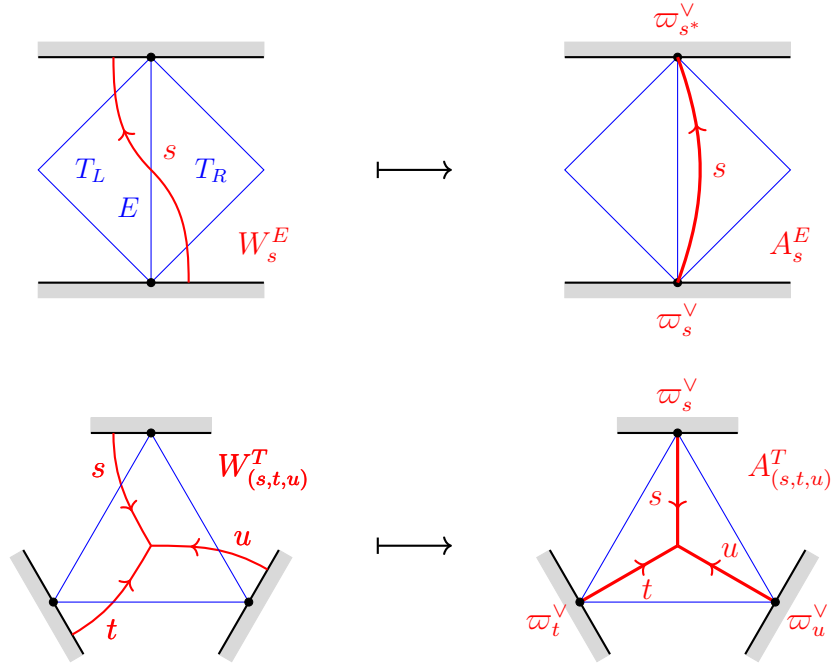
\begin{figure}[htbp]
\centering
\begin{tikzpicture}
\bline{-1.5,-1.5}{1.5,-1.5}{0.2};
\tline{-1.5,1.5}{1.5,1.5}{0.2};
\draw[blue] (-1.5,0) -- (0,-1.5) -- (1.5,0) -- (0,1.5) -- cycle;
\draw[blue] (0,1.5) -- (0,-1.5);
\node[blue,scale=0.9] at (-0.3,-0.5) {$E$};
\node[blue,scale=0.9] at (-0.8,-0) {$T_L$};
\node[blue,scale=0.9] at (0.8,-0) {$T_R$};
\filldraw (0,1.5) circle(1.5pt);
\filldraw (0,-1.5) circle(1.5pt);
\draw[red,thick,->-={0.7}{}] (0.5,-1.5) to[out=90,in=-45] (0,0) node[above right]{$s$} to[out=135,in=-90] (-0.5,1.5);
\node[red] at (1.5,-1.0) {$W^E_s$};

{\begin{scope}[xshift=7cm]
\bline{-1.5,-1.5}{1.5,-1.5}{0.2};
\tline{-1.5,1.5}{1.5,1.5}{0.2};
\draw[blue] (-1.5,0) -- (0,-1.5) -- (1.5,0) -- (0,1.5) -- cycle;
\draw[blue] (0,1.5) -- (0,-1.5);
\filldraw (0,1.5) circle(1.5pt);
\filldraw (0,-1.5) circle(1.5pt);
\draw[red,very thick,->-={0.7}{}] (0,-1.5) to[bend right=20] node[midway,right] {$s$} (0,1.5);
\node[red] at (1.5,-1.0) {$A^E_s$};
\node[red] at (0,-2) {$\varpi_s^\vee$};
\node[red] at (0,2) {$\varpi_{s^\ast}^\vee$};
\end{scope}}
\draw[thick,|->] (3,0) -- (4,0);

{\begin{scope}[yshift=-5cm]
\foreach \i in {0,120,240}
{\begin{scope}[rotate=\i]
\tline{-0.8,1.5}{0.8,1.5}{0.2};
\draw[red,thick,-<-={0.4}{}] (0,0) to[out=120,in=-90] (-0.5,1.5);
\end{scope}
\node[red] at (-0.7,1) {$s$};
\begin{scope}[rotate=120]
\node[red] at (-0.7,1) {$t$};
\end{scope}
\begin{scope}[rotate=240]
\node[red] at (-0.7,1) {$u$};
\end{scope}
\node[red] at (1.5,1.0) {$W^T_{(s,t,u)}$};
}
\draw[blue] (-30:1.5) -- (90:1.5) -- (210:1.5) --cycle;
\foreach \i in {-30,90,210}
\filldraw (\i:1.5) circle(1.5pt);
\draw[thick,|->] (3,0) -- (4,0);
\end{scope}}

{\begin{scope}[xshift=7cm,yshift=-5cm]
\foreach \i in {0,120,240}
{\begin{scope}[rotate=\i]
\tline{-0.8,1.5}{0.8,1.5}{0.2};
\draw[red,very thick,-<-={0.4}{}] (0,0) -- (0,1.5);
\end{scope}}
\node[red] at (1.5,1.0) {$A^T_{(s,t,u)}$};
\node[red] at (115:0.7) {$s$};
\node[red] at (115+120:0.7) {$t$};
\node[red] at (115+240:0.7) {$u$};
\node[red] at (90:2) {$\varpi_s^\vee$};
\node[red] at (90+120:2) {$\varpi_t^\vee$};
\node[red] at (90+240:2) {$\varpi_u^\vee$};
\draw[blue] (-30:1.5) -- (90:1.5) -- (210:1.5) --cycle;
\foreach \i in {-30,90,210}
\filldraw (\i:1.5) circle(1.5pt);
\end{scope}}
\end{tikzpicture}
    \caption{Some of the `elementary' $\mathfrak{g}$-webs (Left) and the corresponding cluster varialbes on $\A_{G^\vee,\Sigma}$ with $\mathfrak{g}^\vee$-weights (Right). When some of the marked points are punctures and the signs of webs are $+$, the correspondence is still the same.}
    \label{fig:elementary laminations}
\end{figure}
In general, the shear coordinates of $W$ should be the same as the $g$-vector of the `pointed' element $F_W$ (cf. \cite{Qin21}). In particular, the shear coordinates of the $\mathfrak{g}$-webs shown in \cref{fig:elementary laminations} are readily determined by this requirement (as they correspond to cluster variables): 
\begin{align*}
    F_W = \prod_{i \in I} A_i^{\check{\sfx}_i(W)},
\end{align*}
where we choose the cluster chart containing the cluster variable $F_W$\footnote{Here we should slightly modify the shear coordinates on the boundary to be the `Langlands dual' version: see \cite[Section 5]{IK22}.}. 
For example, the web $W_s^E$ has exactly one $+1$ shear coordinate at the vertex on the edge $E$ labeled by $s \in S$, while the others are $0$. Observe that the assignment of  the (co-)weights and the web description of cluster variables are consistent with the works \cite{GS19,IYsl3,IYsp4}.

Finally, the Weyl group action on the unbounded $\mathfrak{g}$-laminations should be defined so that the tropicalized Casimir map is equivariant with respect to the natural action on the coweight lattice. This requirement determines the action on a single signed end similarly to \cref{def:Weyl_action}. However, the check for the global consistency will require a hard elaboration as in \cref{subsec:Weyl_action}.

We remark that in the case where $\mathfrak{g}=\mathfrak{sl}_3$, the intersection pairing between the bounded $\fsl_3$-webs and certain unbounded $\fsl_3$-webs is studied by Shen--Sun--Weng \cite{SSW}. It should be the same as the tropicalization of the quantum duality map. 
Intersection coordinates of bounded $\mathfrak{sp}_4$-webs will be studied in the work of Ishibashi--Sun--Yuasa \cite{ISY}.
The higher rank version of shear coordinates for a general $\mathfrak{g}$-web should be defined so that the compatibility relation \cite[Proposition 4.10]{IK22} holds with these intersection coordinates. 

%% file: 6_cluster.tex
\appendix

\section{Cluster varieties associated with the pair \texorpdfstring{$(\fsl_3,\Sigma)$}{(sl3,Sigma)}}\label{sec:appendix}
Here we briefly recall the general theory of cluster varieties \cite{FG09}, and the construction of the seed pattern $\bs(\fsl_3,\Sigma)$ that encodes the cluster structures of the spaces of $\fsl_3$-laminations in consideration. 

\subsection{Cluster varieties}\label{subsec:seeds}

Fix a finite set $I=\{1,\dots,N\}$ of indices, and let $\cF_A$ and $\cF_X$ be fields both isomorphic to the field $\bQ(z_1,\dots,z_N)$ of rational functions on $N$ variables. 
We also fix a subset $I_\uf \subset I$ (\lq\lq unfrozen'') and let $I_\f:=I \setminus I_\uf$ (\lq\lq frozen''). 
Recall that a \emph{(labeled, skew-symmetric) seed} in $(\cF_A,\cF_X)$ is a triple $(\ve,\mathbf{A},\mathbf{X})$, where
\begin{itemize}
    \item $\ve=(\ve_{ij})_{i,j \in I}$ is a skew-symmetric matrix (\emph{exchange matrix}) with values in $\frac{1}{2}\bZ$ such that $\ve_{ij} \in \bZ$ unless $(i,j) \in I_\f \times I_\f$. 
    \item $\mathbf{A}=(A_i)_{i \in I}$ and $\mathbf{X}=(X_i)_{i \in I}$ are tuples of algebraically independent elements (\emph{cluster $\A$-} and \emph{$\X$-variables}) in $\cF_A$ and $\cF_X$, respectively.
\end{itemize}
We say that two seeds in $(\cF_A,\cF_X)$ are \emph{mutation-equivalent} if they are transformed to each other by a finite sequence of mutations and permutations. The equivalence class is called a \emph{mutation class}. Given a mutation class $\sfs$ of seeds, we get:
\begin{description}
    \item[The labeled exchange graph] a graph $\bExch_\sfs$ that parametrizes the seeds in $\sfs$. When no confusion can occur, we simply denote its vertices by $v \in \bExch_\sfs$ instead of $v \in V(\bExch_\sfs)$. 
    For $v \in V(\bExch_\sfs)$, we denote by $\sfs^{(v)}=(\ve^{(v)},\mathbf{A}^{(v)},\mathbf{X}^{(v)})$ the corresponding seed.
    \item[The cluster varieties] schemes
    \begin{align*}
    \X_\sfs:= \bigcup_{v \in \bExch_\sfs} \X_{(v)}, \quad  \A_\sfs:= \bigcup_{v \in \bExch_\sfs} \A_{(v)}.
    \end{align*}
    Here $\X_{(v)} := \bT_{N^{(v)}}
    $ and $\A_{(v)} := \bT_{M^{(v)}}
    $; $N^{(v)}=\bigoplus_{i \in I}\bZ e_i^{(v)}$ and $M^{(v)}=\bigoplus_{i \in I}\bZ f_i^{(v)}$ are lattices dual to each other and $T_\Lambda:=\Hom(\Lambda^*, \bG_m)$ denotes the algebraic torus corresponding to a lattice $\Lambda$.
    The tori $\X_{(v)}$, $v \in \bExch_\sfs$ are patched together according to the cluster $\X$-transformations, and similarly for the $\A$-side. 
    We also have the cluster variety $\X_\sfs^\uf:=\bigcup_{v \in \bExch_\sfs} \X^\uf_{(v)}$ without frozen coordinates.
    Here $\X_{(v)}^\uf:=\bT_{N^{(v)}_\uf} 
    $ and $N_\uf^{(v)} \subset N^{(v)}$ is the sub-lattice spanned by $e_i^{(v)}$ for $i \in I_\uf$. 
    \item[The ensemble map] a morphism $p: \A_\sfs \to \X_\sfs^\uf$ given by the linear map 
    \begin{align}\label{eq:ensemble_map}
    p_{(v)}^{*}: N^{(v)}_\uf \to M^{(v)}, \quad e_i^{(v)} \mapsto \sum_{j \in I}\ve_{ij}^{(v)}f_j^{(v)}
    \end{align}
    on each torus at $v \in \bExch_\sfs$. 
    \item[The cluster modular group] a subgroup $\Gamma_\sfs \subset \mathrm{Aut}(\bExch_\sfs)$ consisting of graph automorphisms $\phi$ which preserve the exchange matrices $\ve^{(v)}$ assigned to vertices and the labels on the edges. 
    The cluster modular group acts on the cluster varieties $\A_\sfs$ and $\X_\sfs$ so that $\phi^*Z_i^{(v)}=Z_i^{(\phi^{-1}(v))}$ for all $\phi \in \Gamma_\bs$, $v \in \bExch_\bs$ and $i \in I$, where $(Z,\cZ) \in \{(A,\A),(X,\X)\}$. These actions commute with the ensemble map. 
\end{description}
The triple $(\A_\sfs,\X_\sfs^\uf,p)$ is called a \emph{cluster ensemble}. In this paper, we skip the discussion on the relation between $\A_\sfs$ and $\X_\sfs$. See \cite[Appendix A]{IK22} for a detail.

\bigskip
\paragraph{\textbf{Cluster exact sequence}}

For $v,v' \in \bExch_\sfs$, there are canonical identifications $\ker p_{(v)}^* \cong \ker p_{(v')}^*$ and $\coker p_{(v)}^* \cong \coker p_{(v')}^*$ given by the seed mutations in the sense of Fock--Goncharov \cite[(7),(8)]{FG09} and permutations.
Under these identifications, 
let 
\begin{align*}
    K := \ker p^*_{(v)}, \qquad
    K^\vee := (\coker p^*_{(v)})^*,
\end{align*}
and $H_\X:=\bT_{K}$, $H_\A:= \bT_{(K^\vee)^\ast} (=\bT_{\coker p_{(v)}^\ast})$.
The short exact sequences 
for the maps \eqref{eq:ensemble_map} commute with the cluster transformations by \cite[Lemmas 2.7 and 2.10]{FG09}, and thus combine to give the following \emph{cluster exact sequence}:
\begin{align}\label{eq:cluster_exact_sequence}
    1 \to H_\A \to \A_\sfs \xrightarrow{p} \X^\uf_\sfs \xrightarrow{\theta} H_\X \to 1,
\end{align}
by which we mean that the ensemble map $p$ is a principal $H_\A$-bundle over its image, and the image coincides with the fiber $\theta^{-1}(1)$ of the map $\theta: \X_\sfs \to H_\X$ induced by the inclusion $K \to N_\uf^{(v)}$.
By taking the character lattice $X^*$, we recover the exact sequence of \eqref{eq:ensemble_map} and by taking the cocharacter lattice $X_*$, we get its dual:
\begin{align*}
X^* \eqref{eq:cluster_exact_sequence}&: \qquad 
0 \ot (K^\vee)^* \ot M^{(v)} \xleftarrow{p^*} N^{(v)}_\uf \ot K \ot 0,\\
X_* \eqref{eq:cluster_exact_sequence}&: \qquad
0 \to K^\vee \to N^{(v)} \xrightarrow{(p^*)^\tr} M^{(v)}_\uf \to K^* \to 0.
\end{align*}
The sub-scheme $\cU^\uf_\sfs:=\theta^{-1}(1)=p(\A_\sfs)$ is a symplectic leaf for the canonical Poisson structure determined by the exchange matrices. The cluster modular group also acts on the tori $H_\X$, $H_\A$ by monomial automorphisms, making the cluster exact sequence \eqref{eq:cluster_exact_sequence} $\Gamma_\sfs$-equivariant. 

An explicit description of the action of $H_\A$ on $\A_\sfs$ is as follows (\cite[Section 2.3]{FG09}). On the level of cocharacter lattices, it is just given by the additive action $K^\vee \otimes N^{(v)} \to N^{(v)}$.
Given an element
$\beta=\sum_i \beta_i^{(v)}e_i^{(v)} \in K^\vee$, let $\bG_m \to H_\A$, $t \mapsto \beta(t)$ be the corresponding one-parameter subgroup. Then the action $\beta(t):\A_{(v)} \to \A_{(v)}$ is written in terms of the cluster coordinates as
\begin{align}\label{eq:action_flow_form}
    \beta(t)^* A_i^{(v)} = t^{\beta_i^{(v)}} A_i^{(v)}, \quad t \in \bG_m.
\end{align}

\bigskip
\paragraph{\textbf{Tropicalizations.}}
The positive structures on the cluster varieties allow us to consider their semifield-valued points. For $\bA=\bZ,\bQ$ or $\bR$, let $\bA^T:=(\bA,\max,+)$ denote the corresponding \emph{tropical semifield} (or the \emph{max-plus semifield}). For an algebraic torus $H$, let $H(\bA^T):=X_\ast(H)\otimes_\bZ (\bA,+)$. A positive rational map $f:H \to H'$ between algebraic tori naturally induces a piecewise-linear (PL for short) map $f^T: H(\bA^T) \to H'(\bA^T)$. We call $f^T$ the \emph{tropicalized map}. 
In particular we have the tropicalized cluster transformations $\mu_k^T:\cZ_{(v)}(\bA^T) \to \cZ_{(v')}(\bA^T)$ for $(\mathsf{z},\cZ) \in \{(\sfa,\A),(\sfx,\X)\}$
, explicitly given as:
\begin{align}
    (\mu_k^T)^* \sfx_i^{(v')}&= \begin{cases}
    -\sfx_k^{(v)} & \mbox{if $i=k$},\\
    \sfx_i^{(v)} - \ve_{ik}^{(v)}[ -\sgn(\ve_{ik}^{(v)})\sfx_k^{(v)}]_+ & \mbox{if $i\neq k$}, 
    \end{cases} \label{eq:tropical x-transf}\\
    (\mu_k^T)^* \sfa_i^{(v')}&= \begin{cases}
    -\sfa_k^{(v)} + \max\left\{\sum_{j \in I} [\ve_{kj}^{(v)}]_+\sfa_j^{(v)},\sum_{j \in I} [-\ve_{kj}^{(v)}]_+\sfa_j^{(v)}\right\}  & \mbox{if $i=k$},\\
    \sfa_i^{(v)} & \mbox{if $i\neq k$}.
    \end{cases} \label{eq:tropical a-transf}
\end{align}
Here $\sfx_i^{(v)}$ and $\sfa_i^{(v)}$ are the coordinate functions induced by the basis vectors $e_i^{(v)}$ and $f_i^{(v)}$ respectively, and $[u]_+:=\max\{0,u\}$ for $u \in \bA$. 
We can use them to define the \emph{tropical cluster varieties}
\begin{align*}
    \X_\sfs(\bA^T) := \bigcup_{v \in \bExch_\sfs} \X_{(v)}(\bA^T), \quad \A_\sfs(\bA^T) := \bigcup_{v \in \bExch_\sfs} \A_{(v)}(\bA^T),
\end{align*}
which are naturally equipped with canonical PL structures. Since the PL maps are equivariant for the scaling action of $\bA_{>0}$, the tropical cluster varieties inherit this $\bA_{>0}$-action. We also have the tropical $\X$-varieties $\X^\uf_\sfs(\bA^T)$ without frozen coordinates. 

Since the above mentioned structures only involve positive rational maps, they are naturally inherited by the tropical cluster varieties. Thus we have:
\begin{itemize}
    \item the cluster exact sequence 
    \begin{align}\label{eq:cluster_exact_seq_general}
        0 \to H_\A(\bA^T) \to \A_\sfs(\bA^T) \xrightarrow{p} \X^\uf_\sfs(\bA^T) \xrightarrow{\theta} H_\X(\bA^T) \to 0
    \end{align}
    in the PL category, and
    \item the actions of $\Gamma_\sfs$ on $\A_\sfs(\bA^T)$ and $\X_\sfs(\bA^T)$ by PL automorphisms and on $H_\A(\bA^T)$ and $H_\X(\bA^T)$ by linear automorphisms, which make the cluster exact sequence $\Gamma_\sfs$-equivariant. Moreover, these actions commute with the rescaling action of $\bA_{>0}$.
\end{itemize}


\subsection{The cluster ensemble associated with the pair \texorpdfstring{$(\fsl_3,\Sigma)$}{(sl(3),S)}}\label{subsec:cluster_sl3}
Here we quickly recall the cluster structures on the moduli spaces $\A_{SL_3,\Sigma}$ and $\X_{PGL_3,\Sigma}$ constructed in \cite{FG03}. We are going to recall the \emph{Fock--Goncharov atlas} associated with ideal triangulations of $\Sigma$ and their mutation-equivalences, since it is typically difficult to describe the entire cluster structure.

Let $\tri$ be an ideal triangulation of $\Sigma$. Then we construct a quiver $Q_\tri$ with the vertex set $I(\tri)$ by drawing the quiver
\begin{align*}
\begin{tikzpicture}[scale=0.7]
\draw[blue] (210:3) -- (-30:3) -- (90:3) --cycle;
\foreach \x in {-30,90,210}
\path(\x:3) node [fill, circle, inner sep=1.2pt]{};
\begin{scope}[color=mygreen,>=latex]
\quiverplus{210:3}{-30:3}{90:3};
\qdlarrow{x122}{x121}
\qdlarrow{x232}{x231}
\qdlarrow{x312}{x311}
\end{scope}
\end{tikzpicture}
\end{align*}
on each triangle, and glue them by the \emph{amalgamation} construction \cite{FG06a}. In our case, this just means that we glue the quivers on adjacent triangles by identifying the two vertices on the shared edge and eliminating the pair of opposite dashed arrows. The vertices on the boundary intervals of $\Sigma$ are declared to be frozen, forming the subset $I_\f(\tri) \subset I(\tri)$ as in \cref{subsec:notation_marked_surface}. Let $\ve^\tri=(\ve_{ij}^\tri)_{i,j \in I(\tri)}$ be the corresponding exchange matrix. 

These quivers $Q^\tri$ (or the exchange matrices $\ve^\tri$) associated with ideal triangulations of $\Sigma$ are mutation-equivalent to each other. Indeed, the quivers $Q^\tri$, $Q^{\tri'}$ associated with two triangulations $\tri$, $\tri'$ connected by a single flip $f_E:\tri \to \tri'$ are transformed to each other via one of the mutation sequences shown in \cref{fig:flip sequence}. Then the assertion follows from the classical fact that any two ideal triangulations of the same marked surface can be transformed to each other by a finite sequence of flips. 

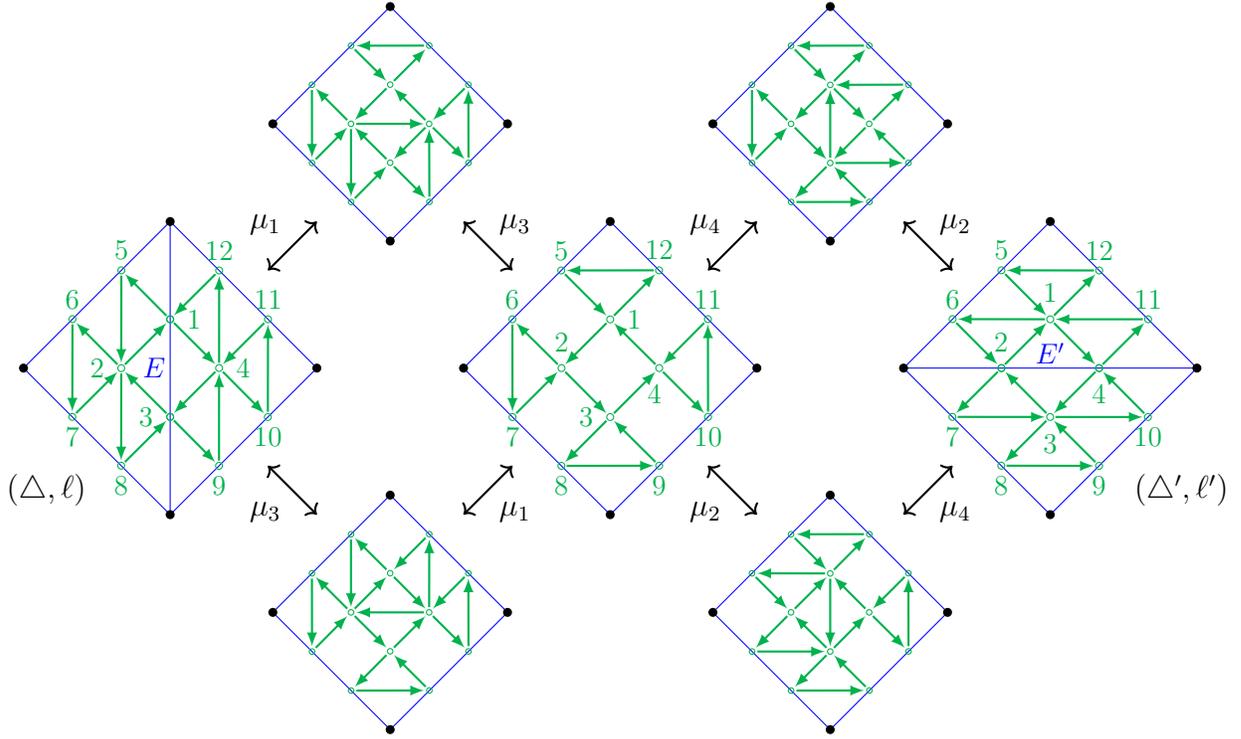
\begin{figure}
\centering
\hspace{-5mm}
\begin{tikzpicture}[scale=0.65]
{\color{blue}
\draw (3,0) -- (0,3) -- (-3,0) -- (0,-3) --cycle;
\draw[blue] (0,-3) --node[midway,left=-0.2em]{\scalebox{0.9}{$E$}} (0,3);
}
\draw (-1.5,-2.5) node[anchor=east]{$(\tri,\ell)$};
\foreach \x in {0,90,180,270}
\path(\x:3) node [fill, circle, inner sep=1.2pt]{};
\quiverplus{3,0}{0,3}{0,-3}
\draw[mygreen](x231) node[right=0.2em]{\scalebox{0.9}{$1$}};
\draw[mygreen](G) node[right=0.2em]{\scalebox{0.9}{$4$}};
\draw[mygreen](x232) node[left=0.2em]{\scalebox{0.9}{$3$}};
\draw[mygreen](x311) node[below]{\scalebox{0.9}{$9$}};
\draw[mygreen](x312) node[below]{\scalebox{0.9}{$10$}};
\draw[mygreen](x121) node[above]{\scalebox{0.9}{$11$}};
\draw[mygreen](x122) node[above]{\scalebox{0.9}{$12$}};
\quiverplus{-3,0}{0,-3}{0,3}
\draw[mygreen](G) node[left=0.2em]{\scalebox{0.9}{$2$}};
\draw[mygreen](x311) node[above]{\scalebox{0.9}{$5$}};
\draw[mygreen](x312) node[above]{\scalebox{0.9}{$6$}};
\draw[mygreen](x121) node[below]{\scalebox{0.9}{$7$}};
\draw[mygreen](x122) node[below]{\scalebox{0.9}{$8$}};

\begin{scope}[xshift=4.5cm,yshift=-5cm,scale=0.8,>=latex,yscale=-1]
\draw[blue] (3,0) -- (0,3) -- (-3,0) -- (0,-3) --cycle;
\foreach \x in {0,90,180,270}
\path(\x:3) node [fill, circle, inner sep=1.2pt]{};
{\color{mygreen}
\quiversquare{-3,0}{0,-3}{3,0}{0,3}
    \qarrow{x241}{x122}
				\qarrow{x122}{x131}
				\qarrow{x131}{x121}
				\qarrow{x121}{x412}
				\qarrow{x412}{x131}
				\qarrow{x131}{x241}
				\qarrow{x241}{x132}
				\qarrow{x132}{x341}
				\qarrow{x341}{x232}
				\qarrow{x232}{x132}
				\qarrow{x132}{x231}
				\qarrow{x231}{x241}
				
				\qarrow{x342}{x242}
				\qarrow{x242}{x411}
				\qarrow{x411}{x342}
				\qarrow{x131}{x242}
				\qarrow{x242}{x132}
				\qarrow{x132}{x131}
}

\end{scope}

\begin{scope}[xshift=4.5cm,yshift=5cm,scale=0.8,>=latex,yscale=-1]
\draw[blue] (3,0) -- (0,3) -- (-3,0) -- (0,-3) --cycle;
\foreach \x in {0,90,180,270}
\path(\x:3) node [fill, circle, inner sep=1.2pt]{};
{\color{mygreen}
\quiversquare{-3,0}{0,-3}{3,0}{0,3}
    \qarrow{x242}{x342}
				\qarrow{x342}{x132}
				\qarrow{x132}{x341}
				\qarrow{x341}{x232}
				\qarrow{x232}{x132}
				\qarrow{x132}{x242}
				\qarrow{x242}{x131}
				\qarrow{x131}{x121}
				\qarrow{x121}{x412}
				\qarrow{x412}{x131}
				\qarrow{x131}{x411}
				\qarrow{x411}{x242}
				
				\qarrow{x132}{x241}
				\qarrow{x241}{x131}
				\qarrow{x131}{x132}
				\qarrow{x122}{x241}
				\qarrow{x241}{x231}
				\qarrow{x231}{x122}
}
\end{scope}

\begin{scope}[xshift=9cm,>=latex,xscale=-1]
\draw[blue] (3,0) -- (0,3) -- (-3,0) -- (0,-3) --cycle;
\foreach \x in {0,90,180,270}
\path(\x:3) node [fill, circle, inner sep=1.2pt]{};
{\color{mygreen}
\quiversquare{-3,0}{0,-3}{3,0}{0,3}
    \qarrow{x122}{x241}
				\qarrow{x131}{x121}
				\qarrow{x121}{x412}
				\qarrow{x412}{x131}
				\qarrow{x241}{x131}
				\qarrow{x132}{x241}
				\qarrow{x132}{x341}
				\qarrow{x341}{x232}
				\qarrow{x232}{x132}
				\qarrow{x241}{x231}
				\qarrow{x231}{x122}
				
				\qarrow{x342}{x242}
				\qarrow{x242}{x411}
				\qarrow{x411}{x342}
				\qarrow{x131}{x242}
				\qarrow{x242}{x132}
}
\draw[mygreen](x242) node[right=0.2em]{\scalebox{0.9}{$1$
}};
\draw[mygreen](x241) node[left=0.2em]{\scalebox{0.9}{$3$}};
\draw[mygreen](x131) node[below=0.2em]{\scalebox{0.9}{$4$
}};
\draw[mygreen](x132) node[above=0.2em]{\scalebox{0.9}{$2$}};
\draw[mygreen](x342) node[above]{\scalebox{0.9}{$5$}};
\draw[mygreen](x341) node[above]{\scalebox{0.9}{$6$}};
\draw[mygreen](x232) node[below]{\scalebox{0.9}{$7$}};
\draw[mygreen](x231) node[below]{\scalebox{0.9}{$8$}};
\draw[mygreen](x122) node[below]{\scalebox{0.9}{$9$}};
\draw[mygreen](x121) node[below]{\scalebox{0.9}{$10$}};
\draw[mygreen](x412) node[above]{\scalebox{0.9}{$11$}};
\draw[mygreen](x411) node[above]{\scalebox{0.9}{$12$}};

\end{scope}

\begin{scope}[xshift=13.5cm,yshift=5cm,scale=0.8,xscale=-1,rotate=90,>=latex]
\draw[blue] (3,0) -- (0,3) -- (-3,0) -- (0,-3) --cycle;
\foreach \x in {0,90,180,270}
\path(\x:3) node [fill, circle, inner sep=1.2pt]{};
{\color{mygreen}
\quiversquare{-3,0}{0,-3}{3,0}{0,3}
    \qarrow{x242}{x342}
				\qarrow{x342}{x132}
				\qarrow{x132}{x341}
				\qarrow{x341}{x232}
				\qarrow{x232}{x132}
				\qarrow{x132}{x242}
				\qarrow{x242}{x131}
				\qarrow{x131}{x121}
				\qarrow{x121}{x412}
				\qarrow{x412}{x131}
				\qarrow{x131}{x411}
				\qarrow{x411}{x242}
				
				\qarrow{x132}{x241}
				\qarrow{x241}{x131}
				\qarrow{x131}{x132}
				\qarrow{x122}{x241}
				\qarrow{x241}{x231}
				\qarrow{x231}{x122}
}
\end{scope}

\begin{scope}[xshift=13.5cm,yshift=-5cm,scale=0.8,yscale=-1,rotate=-90,>=latex]
\draw[blue] (3,0) -- (0,3) -- (-3,0) -- (0,-3) --cycle;
\foreach \x in {0,90,180,270}
\path(\x:3) node [fill, circle, inner sep=1.2pt]{};
{\color{mygreen}
\quiversquare{-3,0}{0,-3}{3,0}{0,3}
    \qarrow{x241}{x122}
				\qarrow{x122}{x131}
				\qarrow{x131}{x121}
				\qarrow{x121}{x412}
				\qarrow{x412}{x131}
				\qarrow{x131}{x241}
				\qarrow{x241}{x132}
				\qarrow{x132}{x341}
				\qarrow{x341}{x232}
				\qarrow{x232}{x132}
				\qarrow{x132}{x231}
				\qarrow{x231}{x241}
				
				\qarrow{x342}{x242}
				\qarrow{x242}{x411}
				\qarrow{x411}{x342}
				\qarrow{x131}{x242}
				\qarrow{x242}{x132}
				\qarrow{x132}{x131}
}
\end{scope}

\begin{scope}[xshift=18cm]
{\color{blue}
\draw (3,0) -- (0,3) -- (-3,0) -- (0,-3) --cycle;
\draw (3,0) --node[midway,above=-0.2em]{\scalebox{0.9}{$E'$}} (-3,0);
}
\foreach \x in {0,90,180,270}
\path(\x:3) node [fill, circle, inner sep=1.2pt]{};
\quiverplus{-3,0}{3,0}{0,3}
\draw[mygreen](G) node[above=0.2em]{\scalebox{0.9}{$1$}};
\draw[mygreen](x121) node[above=0.2em]{\scalebox{0.9}{$2$}};
\draw[mygreen](x122) node[below=0.2em]{\scalebox{0.9}{$4$}};
\draw[mygreen](x311) node[above]{\scalebox{0.9}{$5$}};
\draw[mygreen](x312) node[above]{\scalebox{0.9}{$6$}};
\draw[mygreen](x231) node[above]{\scalebox{0.9}{$11$}};
\draw[mygreen](x232) node[above]{\scalebox{0.9}{$12$}};
\quiverplus{-3,0}{0,-3}{3,0}
\draw[mygreen](G) node[below=0.2em]{\scalebox{0.9}{$3$}};
\draw[mygreen](x231) node[below]{\scalebox{0.9}{$9$}};
\draw[mygreen](x232) node[below]{\scalebox{0.9}{$10$}};
\draw[mygreen](x121) node[below]{\scalebox{0.9}{$7$}};
\draw[mygreen](x122) node[below]{\scalebox{0.9}{$8$}};
\draw (1.5,-2.5) node[anchor=west]{$(\tri',\ell')$};
\end{scope}

\draw[thick,<->] (2,2) --node[midway,above left]{$\mu_1$} (3,3);
\draw[thick,<->] (2,-2) --node[midway,below left]{$\mu_3$} (3,-3);
\draw[thick,<->] (6,3) --node[midway,above right]{$\mu_3$} (7,2);
\draw[thick,<->] (6,-3) --node[midway,below right]{$\mu_1$} (7,-2);
\draw[thick,<->] (16,2) --node[midway,above right]{$\mu_2$} (15,3);
\draw[thick,<->] (16,-2) --node[midway,below right]{$\mu_4$} (15,-3);
\draw[thick,<->] (12,3) --node[midway,above left]{$\mu_4$} (11,2);
\draw[thick,<->] (12,-3) --node[midway,below left]{$\mu_2$} (11,-2);
\end{tikzpicture}
    \caption{Some of the sequences of mutations that realize the flip $f_E: \tri \to \tri'$. Here we partially fix labelings $\ell,\ell'$ of vertices in $I(\tri)$, $I(\tri')$, respectively.}
    \label{fig:flip sequence}
\end{figure}

Then there exists a unique mutation class $\sfs(\fsl_3,\Sigma)$ containing the seeds associated with any ideal triangulations $\tri$. 
More precisely, a \emph{labeled $\fsl_3$-triangulation} $(\tri,\ell)$, namely an ideal triangulation $\tri$ together with a bijection $\ell:I(\tri) \to \{1,\dots,N\}$, give rise to vertices of the labeled exchange graph $\bExch_{\sfs(\fsl_3,\Sigma)}$. \cref{fig:flip sequence} describes a subgraph containing $(\tri,\ell)$ and $(\tri',\ell')$, where the labels $\ell$, $\ell'$ are consistently chosen. Let us simply denote the objects related to $\sfs(\fsl_3,\Sigma)$ by
\begin{align*}
    \A_{\fsl_3,\Sigma}:=\A_{\sfs(\fsl_3,\Sigma)},\quad \X_{\fsl_3,\Sigma}:=\X_{\sfs(\fsl_3,\Sigma)}, \quad \bExch_{\fsl_3,\Sigma}:=\bExch_{\sfs(\fsl_3,\Sigma)}, \quad 
    \Gamma_{\fsl_3,\Sigma}:=\Gamma_{\sfs(\fsl_3,\Sigma)},
\end{align*}
and so on.

\bigskip
\paragraph{\textbf{Cluster modular group.}}
Although the entire structure of the cluster modular group $\Gamma_{\fsl_3,\Sigma}$ is yet unknown, it is known to include the subgroup $(MC(\Sigma) \times \mathrm{Out}(SL_3)) \ltimes W(\fsl_3)^{\bP} \subset \Gamma_{\fsl_3,\Sigma}$ \cite{GS18}. Here $MC(\Sigma)$ denotes the mapping class group of the marked surface $\Sigma$, $\mathrm{Out}(SL_3))=\mathrm{Aut}(SL_3)/\mathrm{Inn}(SL_3)$ is the outer automorphism group of $SL_3$, and $W(\fsl_3)$ is the Weyl group of the Lie algebra $\fsl_3$. The group $\mathrm{Out}(SL_3)$ has order $2$, and generated by the \emph{Dynkin involution} $\ast: G \to G$, $g \mapsto (g^{-1})^\mathsf{T}$. 

Recall that the Weyl group $W(\fsl_3)$ is the group generated by two involutions $r_1,r_2$, subject to the braid relation $r_1 r_2 r_1 = r_2 r_1 r_2$. The inclusion of $W(\fsl_3)^{\bP}$ into the cluster modular group is given in \cite[Section 8]{GS18}. 
For simplicity, let us choose a labeled triangulation $(\tri,\ell)$ such that its local picture around $p$ is as shown in \cref{fig:puncture_quiver}. Let $r_s^{(p)}$ denote the action of each generator assigned at $p$ for $s=1,2$. Then the edge path $\gamma_s^{(p)}$ from $(\tri,\ell)$ to $r_s^{(p)}.(\tri,\ell)$ is described by the sequence
\begin{align}\label{eq:Weyl_action_cluster}
    &\mu_{\gamma_1^{(p)}}:=\mu_3 \mu_4 \mu_5 (5\ 6) \mu_5 \mu_4 \mu_3,\\
    &\mu_{\gamma_2^{(p)}}:=\mu_1 (1\ 2) \mu_1.
\end{align}
It turns out that the orders of mutations along the cycles $\{3,4,5,6\}$ and $\{1,2\}$ do not matter, and all define the same mutation loops \cite[Theorem 7.1]{GS18}. 

For each element $\phi$ in this subgroup, let us call the induced PL action $\phi:\cZ_{\fsl_3,\Sigma}(\bQ^T) \to \cZ_{\fsl_3,\Sigma}(\bQ^T)$ the \emph{cluster action}, in comparison to the geometric action we introduce in the body of this paper in terms of signed $\fsl_3$-webs. 

\bigskip
\paragraph{\textbf{Cluster exact sequence.}}
We are going to investigate the tori $H_\X$ and $H_\A$ for the cluster ensemble associated with $\sfs(\fsl_3,\Sigma)$. 

For each puncture $p \in \bP$, there exists an ideal triangulation $\tri(p)$ of $\Sigma$ such that the star neighborhood of $p$ is a punctured disk with two marked points on its boundary. See, for instance, \cite[Lemma 5.7 (2)]{IIO21}. The corresponding quiver is shown in \cref{fig:puncture_quiver}, and we fix a labeling $\ell(p):I(\tri(p)) \to \{1,\dots,N\}$ as partially shown there. Then it can be easily seen from the quiver that the vectors
\begin{align*}
\alpha_1^{(p)}&:=e_3^{(v)}+e_4^{(v)}+e_5^{(v)}+e_6^{(v)},\\
\alpha_2^{(p)}&:=e_1^{(v)}+e_2^{(v)}
\end{align*}
lie in $\ker p^* \subset N_\uf^{(v)}$, where $v=(\tri(p),\ell(p)) \in \bExch_{\fsl_3,\Sigma}$. The vectors $\alpha_s^{(p)}$ will behave as simple roots under the Weyl group actions (see below).

\begin{figure}[ht]
\begin{tikzpicture}[>=latex]
\draw[blue] (0,0) circle(2cm);
\draw[blue] (0,2) -- (0,-2);
\node[fill,inner sep=1.5pt] at (0,2) {};
\node[fill,inner sep=1.5pt] at (0,-2) {};
\filldraw[fill=white](0,0) circle(2pt);
{\color{mygreen}
\draw(0,0.667) circle(2pt) coordinate(A1) node[above right,scale=0.8]{$1$};
\draw(0,1.333) circle(2pt) coordinate(A2) node[above right,scale=0.8]{$3$};
\draw(0,-0.667) circle(2pt) coordinate(B1) node[above right=0.2em,scale=0.8]{$2$};
\draw(0,-1.333) circle(2pt) coordinate(B2) node[below right,scale=0.8]{$5$};
\draw(1.333,0) circle(2pt) coordinate(C) node[right,scale=0.8]{$4$};
\draw(-1.333,0) circle(2pt) coordinate(D) node[left,scale=0.8]{$6$};
\draw(45:2) circle(2pt) coordinate(E1);
\draw(-45:2) circle(2pt) coordinate(E2);
\draw(135:2) circle(2pt) coordinate(F1);
\draw(-135:2) circle(2pt) coordinate(F2);
\qarrow{A2}{C}
\qarrow{C}{B2}
\qarrow{B2}{D}
\qarrow{D}{A2}
\qarrow{B1}{C}
\qarrow{C}{A1}
\qarrow{A1}{D}
\qarrow{D}{B1}
\qarrow{C}{E1}
\qarrow{E1}{A2}
\qarrow{B2}{E2}
\qarrow{E2}{C}
\qarrow{A2}{F1}
\qarrow{F1}{D}
\qarrow{D}{F2}
\qarrow{F2}{B2}
}
\end{tikzpicture}
    \caption{The quiver $Q^{\tri(p)}$ around a puncture $p$. Here the possible (half-)arrows between boundary vertices are omitted.}
    \label{fig:puncture_quiver}
\end{figure}

\begin{lem}[cf.~{\cite[Theorem 2.10 (2)]{GS19}}]\label{lem:tropical_torus_X}
The set $\{\alpha_s^{(p)} \mid p \in \bP,~s=1,2\}$ is a $\bQ$-basis of the vector space $X^\ast(H_\X)_\bQ=(\ker p^*)_\bQ$. In particular, we have $H_\X(\bQ^T) \cong \bigoplus_{p \in \bP}\mathsf{P}_\bQ^\vee$ by identifying the dual basis of $\alpha_s^{(p)}$ with the fundamental coweight $\varpi^\vee_s \in \mathsf{P}^\vee$.  
\end{lem}
Indeed, the statement is verified by counting the dimensions of the relevant moduli spaces: see \cite[(2.10)]{FG03}. 
The torus $H_\X$ covers the product ${H'}^{\bP}$ of copies of the Cartan subgroup $H' \subset PGL_3$. Via the birational isomorphism $\X^\uf_{\fsl_3,\Sigma} \cong \X_{PGL_3,\Sigma}$ provided by the cluster structure, each Cartan subgroup parametrizes the semisimple part of the monodromy around a puncture.

The torus $H_\A$ is described as follows. For each puncture $p$, choose $v=(\tri(p),\ell(p))$ as above and consider 
\begin{align*}
(\alpha_1^\vee)^{(p)}&:=e_3^{(v)}+e_4^{(v)}+e_5^{(v)}+e_6^{(v)},\\
(\alpha_2^\vee)^{(p)}&:=e_1^{(v)}+e_2^{(v)},
\end{align*}
which are now viewed as vectors in $K^\vee=
(M^{(v)}/p^*(N_\uf^{(v)}))^\ast 
\subset N^{(v)}$. For each special point $m \in M_\partial$, choose an ideal triangulation $\tri(m)$ so that $m$ belongs to a unique triangle, together with a labeling $\ell(m)$ as shown in \cref{fig:special_quiver}. Then the vectors 
\begin{align*}
(\alpha_1^\vee)^{(m)}&:=e_3^{(v)}+e_4^{(v)}+e_5^{(v)},\\
(\alpha_2^\vee)^{(m)}&:=e_1^{(v)}+e_2^{(v)}
\end{align*}
lie in $K^\vee$, where $v=(\tri(m),\ell(m))$. The vectors $(\alpha_s^\vee)^{(m)}$ are viewed as simple coroots.

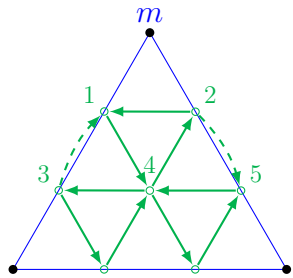
\begin{figure}[ht]
\begin{tikzpicture}[scale=0.7]
\draw[blue] (210:3) -- (-30:3) -- (90:3) node[above]{$m$} --cycle;
\foreach \x in {-30,90,210}
\path(\x:3) node [fill, circle, inner sep=1.2pt]{};
\begin{scope}[color=mygreen,>=latex]
\quiverplus{210:3}{-30:3}{90:3};
\node[above left,scale=0.8] at (x311) {$1$};
\node[above left,scale=0.8] at (x312) {$3$};
\node[above right,scale=0.8] at (x232) {$2$};
\node[above right,scale=0.8] at (x231) {$5$};
\node[above=0.3em,scale=0.8] at (G) {$4$};
\qdlarrow{x232}{x231}
\qdlarrow{x312}{x311}
\end{scope}
\end{tikzpicture}
    \caption{The quiver $Q^{\tri(m)}$ around a special point $m$. Here possible (half-)arrows on the bottom edge are omitted.}
    \label{fig:special_quiver}
\end{figure}

\begin{lem}\label{lem:tropical_torus_A}
The set $\{(\alpha_s^\vee)^{(m)} \mid m \in \bM,~s=1,2\}$ is a $\bQ$-basis of the vector space $X_\ast(H_\A)_\bQ=K^\vee_\bQ$. In particular, we have $H_\A(\bQ^T) \cong \bigoplus_{p \in \bP}\mathsf{Q}^\vee_\bQ$ by identifying the basis vector $(\alpha_s^\vee)^{(p)}$ with the simple coroot $\alpha_s^\vee \in \mathsf{Q}^\vee$.  
\end{lem}
The statement is again verified by counting the dimension. The torus $H_\A$ covers the product $H^{\bM}$ of copies of the Cartan subgroup $H \subset SL_3$. Via the birational isomorphism $\A_{\fsl_3,\Sigma} \cong \A_{SL_3,\Sigma}$ provided by the cluster structure, the action of each Cartan subgroup rescales the decoration assigned at a marked point. 

The relations among the vector spaces are summarized as:
\begin{equation*}
    \begin{tikzcd}
      & \bigoplus_{m \in \bM} \mathsf{P}_\bQ & & \bigoplus_{p \in \bP} \mathsf{Q}_\bQ \\
    H_\A \ar[ur,mapsto,"X^\ast(\bullet)_\bQ"] \ar[dr,mapsto,"X_\ast(\bullet)_\bQ"']& & H_\X \ar[ur,mapsto,"X^\ast(\bullet)_\bQ"] \ar[dr,mapsto,"X_\ast(\bullet)_\bQ"']& \\
      & \bigoplus_{m \in \bM} \mathsf{Q}^\vee_\bQ & & \bigoplus_{p \in \bP} \mathsf{P}^\vee_\bQ.
    \end{tikzcd}
\end{equation*}
In particular, we get the cluster exact sequence of the form 
\begin{align}\label{eq:cluster_exact_seq_sl3}
    0 \to \bigoplus_{m \in \bM} \mathsf{Q}^\vee_\bQ \to \A_{\fsl_3,\Sigma}(\bQ^T) \xrightarrow{p} \X^\uf_{\fsl_3,\Sigma}(\bQ^T) \xrightarrow{\theta} \bigoplus_{p \in \bP} \mathsf{P}^\vee_\bQ \to 0.
\end{align}

%% file: 7_Weyl_cluster.tex
\section{Lamination clusters along the mutation sequence representing $r_{p,s}$}
We take the initial cluster as in \cref{fig:puncture_quiver}. Recall that the mutation sequences representing the actions $r_{p,2}$ and $r_{p,1}$ are given by
\begin{align*}
    \mu_{\gamma_{p,2}}:= (1\ 2) \mu_2 \mu_1,\quad 
    \mu_{\gamma_{p,1}}:=(5\ 6)\mu_3 \mu_4 \mu_6 \mu_5 \mu_4 \mu_3,
\end{align*}
respectively. We show the lamination clusters along these mutation sequences in \cref{fig:seq_r2} and \cref{fig:seq_r1}, respectively. Observe that the final cluster is indeed obtained from the initial one by applying the corresponding geometric actions together with the transpositions $(1\ 2)$ and $(5\ 6)$, respectively.

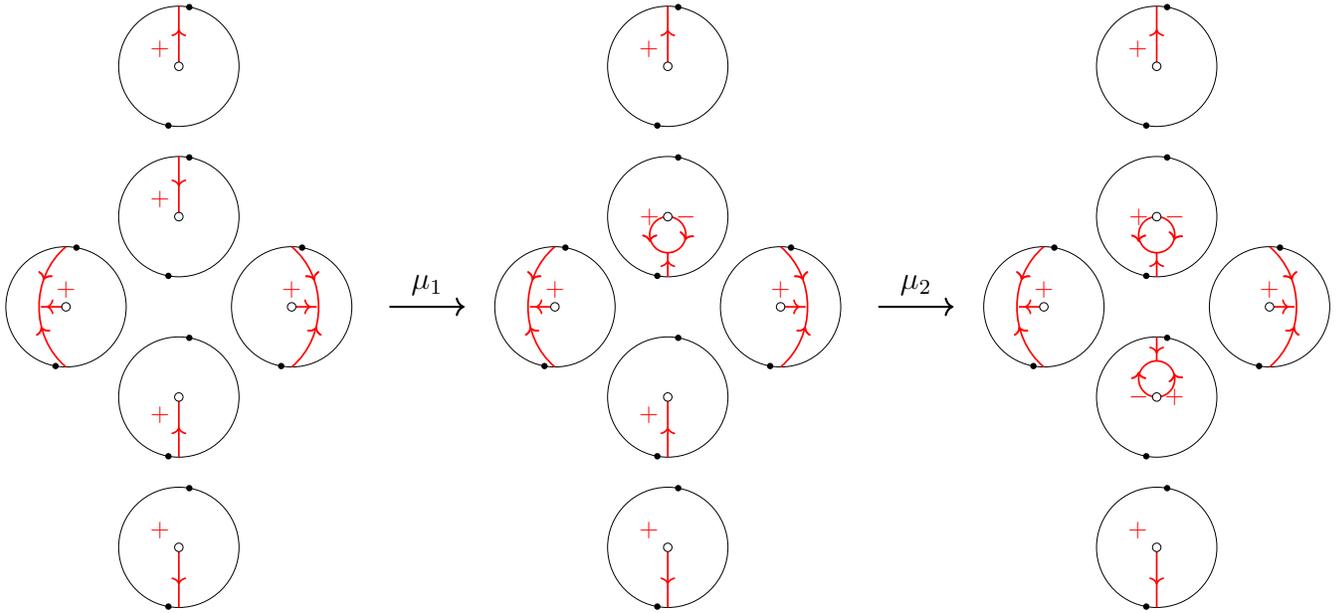
\begin{figure}[ht]
    \centering
\begin{tikzpicture}
\node[scale=0.8] at (0,3.2) {
$\tikz{\draw(0,0) circle(1cm);\foreach \i in {80,260} \fill(\i:1) circle(1.5pt);
\draw[red,thick,-<-](0,1)--(0,0) node[above left]{$+$};\filldraw[fill=white](0,0) circle(2pt);}$
};
\node[scale=0.8] at (0,1.2) {
$\tikz{\draw(0,0) circle(1cm);\foreach \i in {80,260} \fill(\i:1) circle(1.5pt);
\draw[red,thick,->-](0,1)--(0,0) node[above left]{$+$};\filldraw[fill=white](0,0) circle(2pt);}$
};
\node[scale=0.8] at (-1.5,0) {
$\tikz{\draw(0,0) circle(1cm);\foreach \i in {80,260} \fill(\i:1) circle(1.5pt);
\draw[red,thick,->-={0.3}{},-<-={0.7}{}] (0,1) to[bend right=50pt] node[pos=0.5,inner sep=0](A){} (0,-1);
\draw[red,thick,-<-](A)--(0,0) node[above]{$+$};
\filldraw[fill=white](0,0) circle(2pt);}$
};
\node[scale=0.8] at (1.5,0) {
$\tikz{\draw(0,0) circle(1cm);\foreach \i in {80,260} \fill(\i:1) circle(1.5pt);
\draw[red,thick,->-={0.3}{},-<-={0.7}{}] (0,1) to[bend left=50pt] node[pos=0.5,inner sep=0](A){} (0,-1);
\draw[red,thick,-<-](A)--(0,0) node[above]{$+$};
\filldraw[fill=white](0,0) circle(2pt);}$
};
\node[scale=0.8] at (0,-1.2) {
$\tikz{\draw(0,0) circle(1cm);\foreach \i in {80,260} \fill(\i:1) circle(1.5pt);
\draw[red,thick,->-](0,-1)--(0,0) node[below left]{$+$};\filldraw[fill=white](0,0) circle(2pt);}$
};
\node[scale=0.8] at (0,-3.2) {
$\tikz{\draw(0,0) circle(1cm);\foreach \i in {80,260} \fill(\i:1) circle(1.5pt);
\draw[red,thick,-<-](0,-1)--(0,0) node[above left]{$+$};\filldraw[fill=white](0,0) circle(2pt);}$
};
\draw[thick,->] (2.8,0) --node[midway,above]{$\mu_1$} ++(1,0);

\begin{scope}[xshift=6.5cm]
\node[scale=0.8] at (0,3.2) {
$\tikz{\draw(0,0) circle(1cm);\foreach \i in {80,260} \fill(\i:1) circle(1.5pt);
\draw[red,thick,-<-](0,1)--(0,0) node[above left]{$+$};\filldraw[fill=white](0,0) circle(2pt);}$
};
\node[scale=0.8] at (0,1.2) {
$\tikz{\draw(0,0) circle(1cm);\foreach \i in {80,260} \fill(\i:1) circle(1.5pt);
\draw[red,thick,->-={0.3}{},-<-={0.75}{}](0,0) arc(90:-270:0.3); \draw[red,thick,->-={0.7}{}] (0,-1) -- (0,-0.6);\node[red] at (-0.3,0) {$+$};\node[red] at (0.3,0) {$-$};\filldraw[fill=white](0,0) circle(2pt);}$
};
\node[scale=0.8] at (-1.5,0) {
$\tikz{\draw(0,0) circle(1cm);\foreach \i in {80,260} \fill(\i:1) circle(1.5pt);
\draw[red,thick,->-={0.3}{},-<-={0.7}{}] (0,1) to[bend right=50pt] node[pos=0.5,inner sep=0](A){} (0,-1);
\draw[red,thick,-<-](A)--(0,0) node[above]{$+$};
\filldraw[fill=white](0,0) circle(2pt);}$
};
\node[scale=0.8] at (1.5,0) {
$\tikz{\draw(0,0) circle(1cm);\foreach \i in {80,260} \fill(\i:1) circle(1.5pt);
\draw[red,thick,->-={0.3}{},-<-={0.7}{}] (0,1) to[bend left=50pt] node[pos=0.5,inner sep=0](A){} (0,-1);
\draw[red,thick,-<-](A)--(0,0) node[above]{$+$};
\filldraw[fill=white](0,0) circle(2pt);}$
};
\node[scale=0.8] at (0,-1.2) {
$\tikz{\draw(0,0) circle(1cm);\foreach \i in {80,260} \fill(\i:1) circle(1.5pt);
\draw[red,thick,->-](0,-1)--(0,0) node[below left]{$+$};\filldraw[fill=white](0,0) circle(2pt);}$
};
\node[scale=0.8] at (0,-3.2) {
$\tikz{\draw(0,0) circle(1cm);\foreach \i in {80,260} \fill(\i:1) circle(1.5pt);
\draw[red,thick,-<-](0,-1)--(0,0) node[above left]{$+$};\filldraw[fill=white](0,0) circle(2pt);}$
};
\draw[thick,->] (2.8,0) --node[midway,above]{$\mu_2$} ++(1,0);
\end{scope}

\begin{scope}[xshift=13cm]
\node[scale=0.8] at (0,3.2) {
$\tikz{\draw(0,0) circle(1cm);\foreach \i in {80,260} \fill(\i:1) circle(1.5pt);
\draw[red,thick,-<-](0,1)--(0,0) node[above left]{$+$};\filldraw[fill=white](0,0) circle(2pt);}$
};
\node[scale=0.8] at (0,1.2) {
$\tikz{\draw(0,0) circle(1cm);\foreach \i in {80,260} \fill(\i:1) circle(1.5pt);
\draw[red,thick,->-={0.3}{},-<-={0.75}{}](0,0) arc(90:-270:0.3); \draw[red,thick,->-={0.7}{}] (0,-1) -- (0,-0.6);\node[red] at (-0.3,0) {$+$};\node[red] at (0.3,0) {$-$};\filldraw[fill=white](0,0) circle(2pt);}$
};
\node[scale=0.8] at (-1.5,0) {
$\tikz{\draw(0,0) circle(1cm);\foreach \i in {80,260} \fill(\i:1) circle(1.5pt);
\draw[red,thick,->-={0.3}{},-<-={0.7}{}] (0,1) to[bend right=50pt] node[pos=0.5,inner sep=0](A){} (0,-1);
\draw[red,thick,-<-](A)--(0,0) node[above]{$+$};
\filldraw[fill=white](0,0) circle(2pt);}$
};
\node[scale=0.8] at (1.5,0) {
$\tikz{\draw(0,0) circle(1cm);\foreach \i in {80,260} \fill(\i:1) circle(1.5pt);
\draw[red,thick,->-={0.3}{},-<-={0.7}{}] (0,1) to[bend left=50pt] node[pos=0.5,inner sep=0](A){} (0,-1);
\draw[red,thick,-<-](A)--(0,0) node[above]{$+$};
\filldraw[fill=white](0,0) circle(2pt);}$
};
\node[scale=0.8] at (0,-1.2) {
$\tikz{\draw(0,0) circle(1cm);\foreach \i in {80,260} \fill(\i:1) circle(1.5pt);
\draw[red,thick,->-={0.3}{},-<-={0.75}{}](0,0) arc(-90:270:0.3); \draw[red,thick,->-={0.7}{}] (0,1) -- (0,0.6);\node[red] at (0.3,0) {$+$};\node[red] at (-0.3,0) {$-$};\filldraw[fill=white](0,0) circle(2pt);}$
};
\node[scale=0.8] at (0,-3.2) {
$\tikz{\draw(0,0) circle(1cm);\foreach \i in {80,260} \fill(\i:1) circle(1.5pt);
\draw[red,thick,-<-](0,-1)--(0,0) node[above left]{$+$};\filldraw[fill=white](0,0) circle(2pt);}$
};
\end{scope}
\end{tikzpicture}
    \caption{The sequence of lamination clusters along the mutation sequence for $r_{p,2}$.}
    \label{fig:seq_r2}
\end{figure}

\begin{figure}[htbp]
    \centering
\begin{tikzpicture}[scale=0.95]
\node[scale=0.8] at (0,3.2) {
$\tikz{\draw(0,0) circle(1cm);\foreach \i in {80,260} \fill(\i:1) circle(1.5pt);
\draw[red,thick,-<-](0,1)--(0,0) node[above left]{$+$};\filldraw[fill=white](0,0) circle(2pt);}$
};
\node[scale=0.8] at (0,1.2) {
$\tikz{\draw(0,0) circle(1cm);\foreach \i in {80,260} \fill(\i:1) circle(1.5pt);
\draw[red,thick,->-](0,1)--(0,0) node[above left]{$+$};\filldraw[fill=white](0,0) circle(2pt);}$
};
\node[scale=0.8] at (-1.5,0) {
$\tikz{\draw(0,0) circle(1cm);\foreach \i in {80,260} \fill(\i:1) circle(1.5pt);
\draw[red,thick,->-={0.3}{},-<-={0.7}{}] (0,1) to[bend right=50pt] node[pos=0.5,inner sep=0](A){} (0,-1);
\draw[red,thick,-<-](A)--(0,0) node[above]{$+$};
\filldraw[fill=white](0,0) circle(2pt);}$
};
\node[scale=0.8] at (1.5,0) {
$\tikz{\draw(0,0) circle(1cm);\foreach \i in {80,260} \fill(\i:1) circle(1.5pt);
\draw[red,thick,->-={0.3}{},-<-={0.7}{}] (0,1) to[bend left=50pt] node[pos=0.5,inner sep=0](A){} (0,-1);
\draw[red,thick,-<-](A)--(0,0) node[above]{$+$};
\filldraw[fill=white](0,0) circle(2pt);}$
};
\node[scale=0.8] at (0,-1.2) {
$\tikz{\draw(0,0) circle(1cm);\foreach \i in {80,260} \fill(\i:1) circle(1.5pt);
\draw[red,thick,->-](0,-1)--(0,0) node[below left]{$+$};\filldraw[fill=white](0,0) circle(2pt);}$
};
\node[scale=0.8] at (0,-3.2) {
$\tikz{\draw(0,0) circle(1cm);\foreach \i in {80,260} \fill(\i:1) circle(1.5pt);
\draw[red,thick,-<-](0,-1)--(0,0) node[above left]{$+$};\filldraw[fill=white](0,0) circle(2pt);}$
};
\draw[thick,->] (2.8,0) --node[midway,above]{$\mu_3$} ++(1,0);

\begin{scope}[xshift=6.5cm]
\node[scale=0.8] at (0,3.2) {
$\tikz{\draw(0,0) circle(1cm);\foreach \i in {80,260} \fill(\i:1) circle(1.5pt);
\draw[red,thick,-<-](0,1)--(0,0.4);\draw[red,thick,->-] (0,0.4) to[out=0,in=45] (0,-1);\draw[red,thick,->-] (0,0.4) to[out=180,in=135] (0,-1);
\filldraw[fill=white](0,0) circle(2pt);}$
};
\node[scale=0.8] at (0,1.2) {
$\tikz{\draw(0,0) circle(1cm);\foreach \i in {80,260} \fill(\i:1) circle(1.5pt);
\draw[red,thick,->-](0,1)--(0,0) node[above left]{$+$};\filldraw[fill=white](0,0) circle(2pt);}$
};
\node[scale=0.8] at (-1.5,0) {
$\tikz{\draw(0,0) circle(1cm);\foreach \i in {80,260} \fill(\i:1) circle(1.5pt);
\draw[red,thick,->-={0.3}{},-<-={0.7}{}] (0,1) to[bend right=50pt] node[pos=0.5,inner sep=0](A){} (0,-1);
\draw[red,thick,-<-](A)--(0,0) node[above]{$+$};
\filldraw[fill=white](0,0) circle(2pt);}$
};
\node[scale=0.8] at (1.5,0) {
$\tikz{\draw(0,0) circle(1cm);\foreach \i in {80,260} \fill(\i:1) circle(1.5pt);
\draw[red,thick,->-={0.3}{},-<-={0.7}{}] (0,1) to[bend left=50pt] node[pos=0.5,inner sep=0](A){} (0,-1);
\draw[red,thick,-<-](A)--(0,0) node[above]{$+$};
\filldraw[fill=white](0,0) circle(2pt);}$
};
\node[scale=0.8] at (0,-1.2) {
$\tikz{\draw(0,0) circle(1cm);\foreach \i in {80,260} \fill(\i:1) circle(1.5pt);
\draw[red,thick,->-](0,-1)--(0,0) node[below left]{$+$};\filldraw[fill=white](0,0) circle(2pt);}$
};
\node[scale=0.8] at (0,-3.2) {
$\tikz{\draw(0,0) circle(1cm);\foreach \i in {80,260} \fill(\i:1) circle(1.5pt);
\draw[red,thick,-<-](0,-1)--(0,0) node[above left]{$+$};\filldraw[fill=white](0,0) circle(2pt);}$
};
\draw[thick,->] (2.8,0) --node[midway,above]{$\mu_4$} ++(1,0);
\end{scope}

\begin{scope}[xshift=13cm]
\node[scale=0.8] at (0,3.2) {
$\tikz{\draw(0,0) circle(1cm);\foreach \i in {80,260} \fill(\i:1) circle(1.5pt);
\draw[red,thick,-<-](0,1)--(0,0.4);\draw[red,thick,->-] (0,0.4) to[out=0,in=45] (0,-1);\draw[red,thick,->-] (0,0.4) to[out=180,in=135] (0,-1);
\filldraw[fill=white](0,0) circle(2pt);}$
};
\node[scale=0.8] at (0,1.2) {
$\tikz{\draw(0,0) circle(1cm);\foreach \i in {80,260} \fill(\i:1) circle(1.5pt);
\draw[red,thick,->-](0,1)--(0,0) node[above left]{$+$};\filldraw[fill=white](0,0) circle(2pt);}$
};
\node[scale=0.8] at (-1.5,0) {
$\tikz{\draw(0,0) circle(1cm);\foreach \i in {80,260} \fill(\i:1) circle(1.5pt);
\draw[red,thick,->-={0.3}{},-<-={0.7}{}] (0,1) to[bend right=50pt] node[pos=0.5,inner sep=0](A){} (0,-1);
\draw[red,thick,-<-](A)--(0,0) node[above]{$+$};
\filldraw[fill=white](0,0) circle(2pt);}$
};
\node[scale=0.8] at (1.5,0) {
$\tikz{\draw(0,0) circle(1cm);\foreach \i in {80,260} \fill(\i:1) circle(1.5pt);
\draw[red,thick,-<-](0,1)--(0,0.4);\draw[red,thick,-<-] (0,0.4) to[out=0,in=90] (0.6,0);\draw[red,thick,->-](0.6,0)--(0,0);\draw[red,thick,->-](0.6,0) to[out=-90,in=45] (0,-1);
\draw[red,thick,->-] (0,0.4) to[out=180,in=135] (0,-1);\node[red] at (0,-0.3){$+$};
\filldraw[fill=white](0,0) circle(2pt);}$
};
\node[scale=0.8] at (0,-1.2) {
$\tikz{\draw(0,0) circle(1cm);\foreach \i in {80,260} \fill(\i:1) circle(1.5pt);
\draw[red,thick,->-](0,-1)--(0,0) node[below left]{$+$};\filldraw[fill=white](0,0) circle(2pt);}$
};
\node[scale=0.8] at (0,-3.2) {
$\tikz{\draw(0,0) circle(1cm);\foreach \i in {80,260} \fill(\i:1) circle(1.5pt);
\draw[red,thick,-<-](0,-1)--(0,0) node[above left]{$+$};\filldraw[fill=white](0,0) circle(2pt);}$
};
\end{scope}

\begin{scope}[xshift=3.4cm,yshift=-6.5cm]
\draw[thick,->] (-3.8,0) --node[midway,above]{$\mu_5$} ++(1,0);
\node[scale=0.8] at (0,3.2) {
$\tikz{\draw(0,0) circle(1cm);\foreach \i in {80,260} \fill(\i:1) circle(1.5pt);
\draw[red,thick,-<-](0,1)--(0,0.4);\draw[red,thick,->-] (0,0.4) to[out=0,in=45] (0,-1);\draw[red,thick,->-] (0,0.4) to[out=180,in=135] (0,-1);
\filldraw[fill=white](0,0) circle(2pt);}$
};
\node[scale=0.8] at (0,1.2) {
$\tikz{\draw(0,0) circle(1cm);\foreach \i in {80,260} \fill(\i:1) circle(1.5pt);
\draw[red,thick,->-](0,1)--(0,0) node[above left]{$+$};\filldraw[fill=white](0,0) circle(2pt);}$
};
\node[scale=0.8] at (-1.5,0) {
$\tikz{\draw(0,0) circle(1cm);\foreach \i in {80,260} \fill(\i:1) circle(1.5pt);
\draw[red,thick,->-={0.3}{},-<-={0.7}{}] (0,1) to[bend right=50pt] node[pos=0.5,inner sep=0](A){} (0,-1);
\draw[red,thick,-<-](A)--(0,0) node[above]{$+$};
\filldraw[fill=white](0,0) circle(2pt);}$
};
\node[scale=0.8] at (1.5,0) {
$\tikz{\draw(0,0) circle(1cm);\foreach \i in {80,260} \fill(\i:1) circle(1.5pt);
\draw[red,thick,-<-](0,1)--(0,0.4);\draw[red,thick,-<-] (0,0.4) to[out=0,in=90] (0.6,0);\draw[red,thick,->-](0.6,0)--(0,0);\draw[red,thick,->-](0.6,0) to[out=-90,in=45] (0,-1);
\draw[red,thick,->-] (0,0.4) to[out=180,in=135] (0,-1);\node[red] at (0,-0.3){$+$};
\filldraw[fill=white](0,0) circle(2pt);}$
};
\node[scale=0.8] at (0,-1.2) {
$\tikz{\draw(0,0) circle(1cm);\foreach \i in {80,260} \fill(\i:1) circle(1.5pt);
\draw[red,thick,->-](0,-1)--(0,0) node[below left]{$+$};\filldraw[fill=white](0,0) circle(2pt);}$
};
\node[scale=0.8] at (0,-3.2) {
$\tikz{\draw(0,0) circle(1cm);\foreach \i in {80,260} \fill(\i:1) circle(1.5pt);
\draw[red,thick,-<-={0.3}{},->-={0.75}{}](0,0) arc(0:-360:0.2); \draw[red,thick,-<-={0.7}{}] (-0.8,0) -- (-0.4,0);\node[red] at (0,0.3) {$+$};\node[red] at (0,-0.3) {$-$};\draw[red,thick,-<-={0.7}{}] (-0.8,0) to[out=90,in=-150] (0,1);\draw[red,thick,-<-={0.7}{}] (-0.8,0) to[out=-90,in=150] (0,-1);\filldraw[fill=white](0,0) circle(2pt);}$
};
\end{scope}

\begin{scope}[xshift=10.2cm,yshift=-6.5cm]
\draw[thick,->] (-3.8,0) --node[midway,above]{$\mu_6$} ++(1,0);
\node[scale=0.8] at (0,3.2) {
$\tikz{\draw(0,0) circle(1cm);\foreach \i in {80,260} \fill(\i:1) circle(1.5pt);
\draw[red,thick,-<-](0,1)--(0,0.4);\draw[red,thick,->-] (0,0.4) to[out=0,in=45] (0,-1);\draw[red,thick,->-] (0,0.4) to[out=180,in=135] (0,-1);
\filldraw[fill=white](0,0) circle(2pt);}$
};
\node[scale=0.8] at (0,1.2) {
$\tikz{\draw(0,0) circle(1cm);\foreach \i in {80,260} \fill(\i:1) circle(1.5pt);
\draw[red,thick,->-](0,1)--(0,0) node[above left]{$+$};\filldraw[fill=white](0,0) circle(2pt);}$
};
\node[scale=0.8] at (-1.5,0) {
$\tikz{\draw(0,0) circle(1cm);\foreach \i in {80,260} \fill(\i:1) circle(1.5pt);
\draw[red,thick,-<-={0.3}{},->-={0.75}{}](0,0) arc(90:-270:0.3); \draw[red,thick,-<-={0.7}{}] (0,-1) -- (0,-0.6);\node[red] at (-0.3,0) {$+$};\node[red] at (0.3,0) {$-$};
\filldraw[fill=white](0,0) circle(2pt);}$
};
\node[scale=0.8] at (1.5,0) {
$\tikz{\draw(0,0) circle(1cm);\foreach \i in {80,260} \fill(\i:1) circle(1.5pt);
\draw[red,thick,-<-](0,1)--(0,0.4);\draw[red,thick,-<-] (0,0.4) to[out=0,in=90] (0.6,0);\draw[red,thick,->-](0.6,0)--(0,0);\draw[red,thick,->-](0.6,0) to[out=-90,in=45] (0,-1);
\draw[red,thick,->-] (0,0.4) to[out=180,in=135] (0,-1);\node[red] at (0,-0.3){$+$};
\filldraw[fill=white](0,0) circle(2pt);}$
};
\node[scale=0.8] at (0,-1.2) {
$\tikz{\draw(0,0) circle(1cm);\foreach \i in {80,260} \fill(\i:1) circle(1.5pt);
\draw[red,thick,->-](0,-1)--(0,0) node[below left]{$+$};\filldraw[fill=white](0,0) circle(2pt);}$
};
\node[scale=0.8] at (0,-3.2) {
$\tikz{\draw(0,0) circle(1cm);\foreach \i in {80,260} \fill(\i:1) circle(1.5pt);
\draw[red,thick,-<-={0.3}{},->-={0.75}{}](0,0) arc(0:-360:0.2); \draw[red,thick,-<-={0.7}{}] (-0.8,0) -- (-0.4,0);\node[red] at (0,0.3) {$+$};\node[red] at (0,-0.3) {$-$};\draw[red,thick,-<-={0.7}{}] (-0.8,0) to[out=90,in=-150] (0,1);\draw[red,thick,-<-={0.7}{}] (-0.8,0) to[out=-90,in=150] (0,-1);\filldraw[fill=white](0,0) circle(2pt);}$
};
\end{scope}

\begin{scope}[xshift=6.5cm,yshift=-13cm]
\draw[thick,->] (-3.8,0) --node[midway,above]{$\mu_4$} ++(1,0);
\node[scale=0.8] at (0,3.2) {
$\tikz{\draw(0,0) circle(1cm);\foreach \i in {80,260} \fill(\i:1) circle(1.5pt);
\draw[red,thick,-<-](0,1)--(0,0.4);\draw[red,thick,->-] (0,0.4) to[out=0,in=45] (0,-1);\draw[red,thick,->-] (0,0.4) to[out=180,in=135] (0,-1);
\filldraw[fill=white](0,0) circle(2pt);}$
};
\node[scale=0.8] at (0,1.2) {
$\tikz{\draw(0,0) circle(1cm);\foreach \i in {80,260} \fill(\i:1) circle(1.5pt);
\draw[red,thick,->-](0,1)--(0,0) node[above left]{$+$};\filldraw[fill=white](0,0) circle(2pt);}$
};
\node[scale=0.8] at (-1.5,0) {
$\tikz{\draw(0,0) circle(1cm);\foreach \i in {80,260} \fill(\i:1) circle(1.5pt);
\draw[red,thick,-<-={0.3}{},->-={0.75}{}](0,0) arc(90:-270:0.3); \draw[red,thick,-<-={0.7}{}] (0,-1) -- (0,-0.6);\node[red] at (-0.3,0) {$+$};\node[red] at (0.3,0) {$-$};
\filldraw[fill=white](0,0) circle(2pt);}$
};
\node[scale=0.8] at (1.5,0) {
$\tikz{\draw(0,0) circle(1cm);\foreach \i in {80,260} \fill(\i:1) circle(1.5pt);
\draw[red,thick,-<-={0.3}{},->-={0.75}{}](0,0) arc(180:-180:0.2); \draw[red,thick,-<-={0.7}{}] (0.8,0) -- (0.4,0);\node[red] at (0,-0.3) {$+$};\node[red] at (0,0.3) {$-$};\draw[red,thick,-<-={0.7}{}] (0.8,0) to[out=90,in=-30] (0,1);\draw[red,thick,-<-={0.7}{}] (0.8,0) to[out=-90,in=30] (0,-1);
\filldraw[fill=white](0,0) circle(2pt);}$
};
\node[scale=0.8] at (0,-1.2) {
$\tikz{\draw(0,0) circle(1cm);\foreach \i in {80,260} \fill(\i:1) circle(1.5pt);
\draw[red,thick,->-](0,-1)--(0,0) node[below left]{$+$};\filldraw[fill=white](0,0) circle(2pt);}$
};
\node[scale=0.8] at (0,-3.2) {
$\tikz{\draw(0,0) circle(1cm);\foreach \i in {80,260} \fill(\i:1) circle(1.5pt);
\draw[red,thick,-<-={0.3}{},->-={0.75}{}](0,0) arc(0:-360:0.2); \draw[red,thick,-<-={0.7}{}] (-0.8,0) -- (-0.4,0);\node[red] at (0,0.3) {$+$};\node[red] at (0,-0.3) {$-$};\draw[red,thick,-<-={0.7}{}] (-0.8,0) to[out=90,in=-150] (0,1);\draw[red,thick,-<-={0.7}{}] (-0.8,0) to[out=-90,in=150] (0,-1);\filldraw[fill=white](0,0) circle(2pt);}$
};
\end{scope}

\begin{scope}[xshift=13cm,yshift=-13cm]
\draw[thick,->] (-3.8,0) --node[midway,above]{$\mu_3$} ++(1,0);
\node[scale=0.8] at (0,3.2) {
$\tikz{\draw(0,0) circle(1cm);\foreach \i in {80,260} \fill(\i:1) circle(1.5pt);
\draw[red,thick,-<-={0.3}{},->-={0.75}{}](0,0) arc(-90:270:0.3); \draw[red,thick,-<-={0.7}{}] (0,1) -- (0,0.6);\node[red] at (0.3,0) {$+$};\node[red] at (-0.3,0) {$-$};
\filldraw[fill=white](0,0) circle(2pt);}$
};
\node[scale=0.8] at (0,1.2) {
$\tikz{\draw(0,0) circle(1cm);\foreach \i in {80,260} \fill(\i:1) circle(1.5pt);
\draw[red,thick,->-](0,1)--(0,0) node[above left]{$+$};\filldraw[fill=white](0,0) circle(2pt);}$
};
\node[scale=0.8] at (-1.5,0) {
$\tikz{\draw(0,0) circle(1cm);\foreach \i in {80,260} \fill(\i:1) circle(1.5pt);
\draw[red,thick,-<-={0.3}{},->-={0.75}{}](0,0) arc(90:-270:0.3); \draw[red,thick,-<-={0.7}{}] (0,-1) -- (0,-0.6);\node[red] at (-0.3,0) {$+$};\node[red] at (0.3,0) {$-$};
\filldraw[fill=white](0,0) circle(2pt);}$
};
\node[scale=0.8] at (1.5,0) {
$\tikz{\draw(0,0) circle(1cm);\foreach \i in {80,260} \fill(\i:1) circle(1.5pt);
\draw[red,thick,-<-={0.3}{},->-={0.75}{}](0,0) arc(180:-180:0.2); \draw[red,thick,-<-={0.7}{}] (0.8,0) -- (0.4,0);\node[red] at (0,-0.3) {$+$};\node[red] at (0,0.3) {$-$};\draw[red,thick,-<-={0.7}{}] (0.8,0) to[out=90,in=-30] (0,1);\draw[red,thick,-<-={0.7}{}] (0.8,0) to[out=-90,in=30] (0,-1);
\filldraw[fill=white](0,0) circle(2pt);}$
};
\node[scale=0.8] at (0,-1.2) {
$\tikz{\draw(0,0) circle(1cm);\foreach \i in {80,260} \fill(\i:1) circle(1.5pt);
\draw[red,thick,->-](0,-1)--(0,0) node[below left]{$+$};\filldraw[fill=white](0,0) circle(2pt);}$
};
\node[scale=0.8] at (0,-3.2) {
$\tikz{\draw(0,0) circle(1cm);\foreach \i in {80,260} \fill(\i:1) circle(1.5pt);
\draw[red,thick,-<-={0.3}{},->-={0.75}{}](0,0) arc(0:-360:0.2); \draw[red,thick,-<-={0.7}{}] (-0.8,0) -- (-0.4,0);\node[red] at (0,0.3) {$+$};\node[red] at (0,-0.3) {$-$};\draw[red,thick,-<-={0.7}{}] (-0.8,0) to[out=90,in=-150] (0,1);\draw[red,thick,-<-={0.7}{}] (-0.8,0) to[out=-90,in=150] (0,-1);\filldraw[fill=white](0,0) circle(2pt);}$
};
\end{scope}
\end{tikzpicture}
    \caption{The sequence of lamination clusters along the mutation sequence for $r_{p,1}$.}
    \label{fig:seq_r1}
\end{figure}
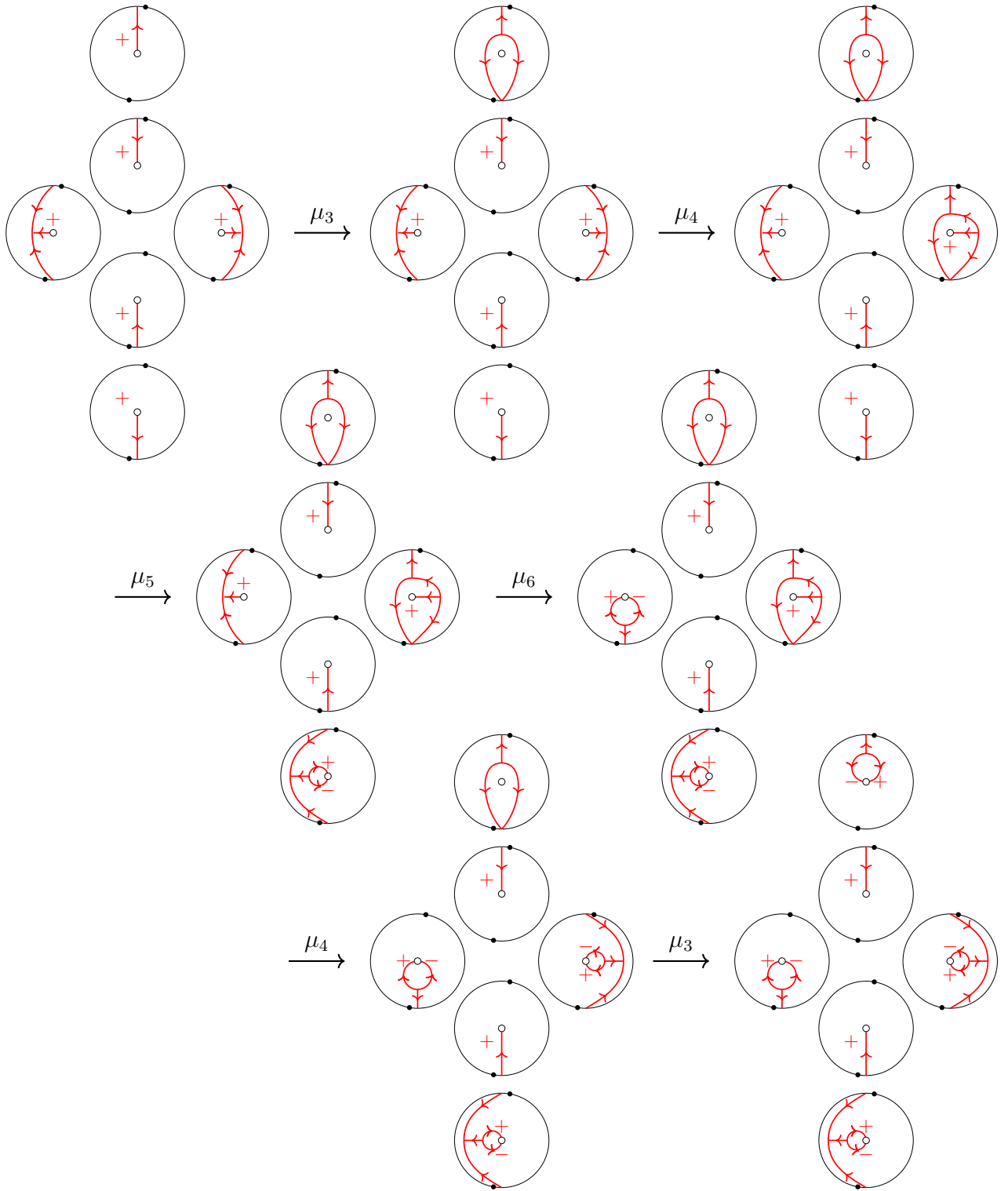

\vfill

%% file: 0_main.bbl
\begin{thebibliography}{GHKK18}


\bibitem[CKM14]{CKM}
S. Cautis, J. Kamnitzer, S. Morrison,
{\em Webs and quantum skew Howe duality},
Math. Ann. \textbf{360} (2014), no. 1-2, 351--390.

\bibitem[DS20I]{DS20I}
D. C. Douglas and Z. Sun,
{\em Tropical {F}ock-{G}oncharov coordinates for $SL_3$-webs on surfaces {I}: construction},
arXiv:2011.01768.

\bibitem[DS20II]{DS20II}
D. C. Douglas and Z. Sun,
{\em Tropical {F}ock-{G}oncharov coordinates for $SL_3$-webs on surfaces {II}: naturality},
arXiv:2012.14202. 

\bibitem[FG06a]{FG03} 
V. V. Fock and A. B. Goncharov, 
{\em Moduli spaces of local systems and higher Teichm\"uller theory},
Publ. Math. Inst. Hautes \'Etudes Sci., \textbf{103} (2006), 1--211.

\bibitem[FG06b]{FG06a} 
V. V. Fock and A. B. Goncharov, 
 {\em Cluster $\mathcal{X}$-varieties, amalgamation and Poisson-Lie groups},
Algebraic geometry and number theory, volume 253 of Progr. Math., 
PP 27--68, Birkh\"auser Boston, Boston, MA, 2006.


\bibitem[FG07b]{FG07c}
V. V. Fock and A. B. Goncharov, 
{\em Moduli spaces of convex projective structures on surfaces},
Adv. Math. \textbf{208} (2007), no. 1, 249--273. 

\bibitem[FG09]{FG09}
V. V. Fock and A. B. Goncharov, 
\emph{Cluster ensembles, quantization and the dilogarithm},
Ann. Sci. \'Ec. Norm. Sup\'er., \textbf{42} (2009), 865--930.



\bibitem[FP14]{FP14}
S. Fomin and P. Pylyavskyy, 
\emph{Webs on surfaces, rings of invariants, and clusters},
Proc. Natl. Acad. Sci. USA \textbf{111} (2014), 9680--9687.

\bibitem[FP16]{FP16}
S. Fomin and P. Pylyavskyy, 
\emph{Tensor diagrams and cluster algebras},
Adv. Math. \textbf{300} (2016), 717--787.

\bibitem[FP21]{FP21}
C. Fraser and P. Pykyavskyy,
{\em Tensor diagrams and cluster combinatorics at punctures},
arXiv:2107.13069.

\bibitem[FST08]{FST}
S. Fomin, M. Shapiro and D. Thurston,
{\em Cluster algebras and triangulated surfaces. {I}. Cluster complexes},
Acta Math. \textbf{201} (2008), 83--146.

\bibitem[FS20]{FS20}
C. Frohman and A. S. Sikora,
\emph{$SU(3)$-skein algebras and webs on surfaces},
Math. Z. \textbf{300} (2022), no. 1, 33--56.




\bibitem[GS15]{GS15}
A. B. Goncharov and L. Shen,
{\em Geometry of canonical bases and mirror symmetry},
Invent. Math. \textbf{202} (2015), no. 2, 487--633. 

\bibitem[GS18]{GS18} 
A. B. Goncharov and L. Shen,
 {\em Donaldson-Thomas transformations of moduli spaces of $G$-local systems},
Adv. Math. \textbf{327} (2018), 225--348. 

\bibitem[GS19]{GS19}
A. B. Goncharov and L. Shen,
\emph{Quantum geometry of moduli spaces of local systems and representation theory},
arXiv:1904.10491v3.





\bibitem[IK22]{IK22}
T. Ishibashi and S. Kano,
\emph{Unbounded $\mathfrak{sl}_3$-laminations and their shear coordinates},
arXiv:2204.08947; to appear in Algebr. Geom. Topol.

\bibitem[IK23]{IKar}
T. Ishibashi and H. Karuo,
\emph{Quantum duality maps, skein algebras and their ensemble compatibility},
arXiv:2305.19074.

\bibitem[IIO21]{IIO21}
R. Inoue, T. Ishibashi and H. Oya,
{\em Cluster realizations of Weyl groups and higher Teichmüller theory}, 
Sel. Math. New Ser. \textbf{27}, 37 (2021).

\bibitem[IOS23]{IOS}
T. Ishibashi, H. Oya and L. Shen,
\emph{$\mathscr{A}=\mathscr{U}$ for cluster algebras from moduli spaces of $G$-local systems},
Adv. Math. \textbf{431} (2023).

\bibitem[ISY]{ISY}
T. Ishibashi, Z. Sun and W. Yuasa,
\emph{Bounded $\mathfrak{sp}_4$-laminations and their intersection coordinates},
in preparation.

\bibitem[IY23]{IYsl3}
T. Ishibashi and W. Yuasa,
{\em Skein and cluster algebras of unpunctured surfaces for $\mathfrak{sl}_3$},
Math. Z. \textbf{303}, 72 (2023).

\bibitem[IY22]{IYsp4}
T. Ishibashi and W. Yuasa,
{\em Skein and cluster algebras of unpunctured surfaces for $\mathfrak{sp}_4$},
arXiv:2207.01540.

\bibitem[Kim21]{Kim21}
H. K. Kim,
{\em $SL_3$-laminations as bases for $PGL_3$ cluster varieties for surfaces},
arXiv:2011.14765; to appear in Mem. Am. Math. Soc.	

\bibitem[Kup96]{Kuperberg}
G. Kuperberg,
\emph{Spiders for rank $2$ {L}ie groups},
Comm. Math. Phys. \textbf{180} (1996), no. 1, 109--151.



\bibitem[MOY98]{MOY}
H.~Murakami, T.~Ohtsuki and S.~Yamada,
{\em Homfly polynomial via an invariant of colored plane graphs},
Enseign. Math. (2) \textbf{44} (1998), no. 3-4, 325--360.



\bibitem[Qin21]{Qin21}
F. Qin,
{\em Cluster algebras and their bases},
arXiv:2108.09279.

\bibitem[RY14]{RY}
J. Roger and T. Yang,
\emph{The skein algebra of arcs and links and the decorated Teichm\"{u}ller space},
J. Differential Geom. \textbf{96} (2014), no.~1, 95--140.

\bibitem[SSW23]{SSW}
L. Shen, Z. Sun and D. Weng, 
\emph{Intersections of Dual $SL_3$-Webs},
arXiv:2311.15466.

\bibitem[SSW]{SSW_RY}
L. Shen, Z. Sun and D. Weng, 
\emph{The punctured $\operatorname{SL}_3$ skein algebra and the quantization of $\mathcal{A}_{\operatorname{SL}_3,\hat{S}}$ moduli space},
in preparation.



\end{thebibliography}
